\documentclass[11pt,a4paper]{amsart}

\usepackage{amsmath,amscd}
\usepackage{amssymb}
\usepackage{amsthm}

\usepackage{graphicx}

\usepackage{mathrsfs}
\usepackage[ocgcolorlinks, linkcolor=blue]{hyperref}

\usepackage{calc}
             {\begin{list}{\arabic{enumi}.}{\usecounter{enumi}%
              \setlength{\labelsep}{0.5em}%
              \settowidth{\labelwidth}{(\arabic{enumi})}%
              \setlength{\leftmargin}{\labelwidth+\labelsep}}}%
             {\end{list}}

\usepackage{comment}

\DeclareRobustCommand{\SkipTocEntry}[5]{}

\usepackage{xcolor}
\definecolor{LOcolor}{RGB}{150,100,0}

\DeclareMathOperator{\WF}{WF}

\newtheorem{Theorem}{Theorem}[section]

\newtheorem{Lemma}[Theorem]{Lemma}

\newtheorem{Proposition}[Theorem]{Proposition}

\theoremstyle{definition}

\newtheorem*{DefinitionNoNumber}{Definition}

\newtheorem{Example}[Theorem]{Example}
\newtheorem*{ExampleNoNumber}{Example}
\newtheorem*{ExamplesNoNumber}{Examples}
\newtheorem{Remark}[Theorem]{Remark}
\newtheorem*{RemarkNoNumber}{Remark}

\numberwithin{equation}{section}

\newcommand{\mR}{\mathbb{R}}                    % Formatting for R
\newcommand{\mC}{\mathbb{C}}                    % Formatting for C
\newcommand{\mZ}{\mathbb{Z}}                    % Formatting for Z
\newcommand{\abs}[1]{\lvert #1 \rvert}          % Formatting for the absolute value
\newcommand{\norm}[1]{\lVert #1 \rVert}         % Formatting for the norm
\newcommand{\ol}[1]{\overline{#1}}

\newcommand{\im}{\mathrm{Im}}

\newcommand{\mDp}{\mathscr{D}^{\prime}}

\newcommand{\mF}{\mathscr{F}}

\newcommand{\supp}{\mathrm{supp}}

\newcommand{\eps}{\varepsilon}

\newcommand{\p}{\partial}

\DeclareMathOperator{\tr}{tr}
\def\I{\mathcal I}
\def\R{\mathcal R}

\def\subp{\text{\rm sub}}
\def\lat{\text{\rm lat}}

\newcommand{\pr}{\mathrm{pr}}

\newcounter{sidenote}

\begin{document}

\title{Inverse problems for real principal type operators} 

\author[L. Oksanen]{Lauri Oksanen}
\address{Department of Mathematics, University College London}
\email{l.oksanen@ucl.ac.uk}

\author[M. Salo]{Mikko Salo}
\address{Department of Mathematics and Statistics, University of Jyv\"askyl\"a}
\email{mikko.j.salo@jyu.fi}

\author[P. Stefanov]{Plamen Stefanov}
\address{Department of Mathematics, Purdue University}
\email{stefanop@purdue.edu}

\author[G. Uhlmann]{Gunther Uhlmann}
\address{Department of Mathematics, University of Washington / Jockey Club Institute for Advanced Study, HKUST, Hong Kong}
\email{gunther@math.washington.edu}

%\subjclass{53C25, 53C21, 58F17, 35J15}

%\date{\today}

%\maketitle

\begin{abstract}
We consider inverse boundary value problems for general real principal type differential operators. The first results state that the Cauchy data set uniquely determines the scattering relation of the operator and bicharacteristic ray transforms of lower order coefficients. We also give two different boundary determination methods for general operators, and prove global uniqueness results for determining coefficients in nonlinear real principal type equations. The article presents a unified approach for treating inverse boundary problems for transport and wave equations, and highlights the role of propagation of singularities in the solution of related inverse problems.
\end{abstract}

\maketitle

{{
\parskip=.155em
\tableofcontents
}}

\vspace{-39pt}

\section{Introduction and main results} \label{sec_introduction}

\addtocontents{toc}{\SkipTocEntry}
\subsection{Motivation}

This article studies inverse boundary value problems for linear and nonlinear differential operators with variable coefficients. In many works in the theory of inverse boundary value problems, it has been customary to focus on specific problems (e.g.\ the Calder\'on problem for Laplace type equations, 
the inverse boundary problem for the wave equation that we will call \emph{Gel'fand problem} due to the roughly equivalent formulation posed in \cite{Gelfand}, boundary and scattering rigidity problems for transport equations, or geodesic X-ray transforms). 
However, there have been unexpected connections between problems having different character, such as 
\begin{itemize}
\item 
the explicit appearance of geodesic X-ray transforms in the Calder\'on problem \cite{DKSaU}; %, in addition to classical X-ray transform methods in the Gel'fand and boundary rigidity problems;
\item 
a reduction from the Calder\'on problem to the Gel'fand problem involving the boundary control method for the wave equation \cite{DKLS}; and
\item 
a direct relation between the boundary rigidity problem and the Calder\'on problem \cite{PestovUhlmann}.
\end{itemize}

In this article we will work in the spirit of a general point of view to linear partial differential equations, as advocated in \cite{Hormander}. The objective is to find structural conditions that allow us to consider large classes of operators, and to identify fundamental principles that make it possible to solve related inverse problems. In addition to obtaining results for general equations, we hope that the new point of view will yield a better understanding of the methods that are currently available, the connections between different problems, and the extent to which it is possible to push the methods. Of course this is a large program and we view the present work as a starting point.

Here, we will consider operators of real principal type and the consequences of propagation of singularities for related inverse problems. This will include inverse problems for transport equations (boundary and scattering rigidity) and wave equations (Gel'fand problem) as special cases. As an example, if one knows the outgoing Dirichlet-to-Neumann map for the operator $P$ appearing in the wave equation on a Riemannian manifold $(M,g)$ with time-independent magnetic potential $A$ and electric potential $V$, then up to suitable gauge transformations it is possible to recover 
\begin{itemize}
\item 
the Taylor series of the coefficients at the boundary, up to the natural gauge,
\item 
the scattering relation for the metric $g$,
\item 
the geodesic X-ray transform of the magnetic potential $A$ up to $2\pi \mZ$, and 
\item 
the geodesic X-ray transform of the electric potential $V$.
\end{itemize}
We refer to \cite{StefanovYang}, where these results are obtained for Lorentzian wave equations, including classical ones  with time-dependent coefficients, and to references therein for earlier results. This reduces the problem of finding the coefficients of $P$ to the well studied boundary/lens rigidity one and to the problem of inverting the geodesic X-ray transform on functions or more generally  on tensor fields. 

In this work we prove counterparts of the above results for any real principal type differential operator $P$. The scattering relation for the metric (i.e.\ for geodesic flow) will be replaced by the scattering relation for the null bicharacteristic flow for $P$. The geodesic X-ray transform will in turn be replaced by the relevant null bicharacteristic ray transform. In order to recover the actual coefficients it remains to solve the related scattering rigidity and bicharacteristic ray transform problems, if possible. These issues are left to forthcoming works. However, we will also prove boundary determination results for real principal type operators (these can be used to recover real-analytic coefficients), and interior determination results for related nonlinear operators.

\addtocontents{toc}{\SkipTocEntry}
\subsection{Real principal type operators}

Before stating the main results, we will need to give some facts concerning real principal type operators. First we recall the definition of these operators from \cite[Definition~26.1.8 in Section 26.1]{Hormander}.

\begin{DefinitionNoNumber}
Let $X$ be an open $C^{\infty}$ manifold. A properly supported pseudodifferential operator $P \in \Psi^m(X)$ is of \emph{real principal type} if it has real homogeneous principal symbol $p_m$ of order $m$, and if no complete null bicharacteristic of $P$ stays over a compact subset of $X$.
\end{DefinitionNoNumber}

\begin{ExamplesNoNumber}
Any real vector field with no trapped integral curves is of real principal type. The wave operator $\partial_t^2 - \Delta_{g_0}$ in a space-time cylinder $\mathbb R \times M_0$, with $(M_0, g_0)$ a Riemannian manifold, is of real principal type, and so are more general Lorentzian wave operators and higher order hyperbolic operators if the nontrapping condition is satisfied. Tricomi type operators $x_2 \partial_{x_1}^2 + \partial_{x_2}^2$ are real principal type, but Keldysh type operators $x_1 \partial_{x_1}^2 + \partial_{x_2}^2$ are not. Elliptic operators with real principal symbol are real principal type, but they have no null bicharacteristics so for such operators most of our results are void. Our results do apply to operators with elliptic factors such as $(\p_t^2 - \Delta_{g_0})(\p_t^2 + \Delta_{g_0})$. Of course, lower order terms with complex coefficients can be included (heuristically, principal type operators are those whose main behaviour is described by the principal symbol alone). We remark that different definitions of real principal type exist in the literature (e.g.\ the stronger local condition $d_{\xi} p \neq 0$ appears, ensuring that null bicharacteristic curves have no cusps). Moreover, operators such as the Schr\"odinger operator $i\p_t - \Delta$ or the plate operator $\p_t^2 + \Delta^2$ are not real principal type, but using a suitable (anisotropic) weighting for the time derivative makes it possible to treat such operators in a similar way as real principal type operators, although new phenomena appear (see e.g.\ \cite{Tataru_survey}). One could also consider systems of real principal type, and there are several interesting physical examples (e.g.\ Maxwell and elastic wave equations).
\end{ExamplesNoNumber}

A major feature of operators with real principal symbol is the fact that singularities propagate along \emph{null bicharacteristics} (i.e.\ integral curves of the Hamilton vector field $H_{p_m}$ in the characteristic set $\mathrm{Char}(P) := \{ (x,\xi) \in T^* X \setminus 0 \,;\, p_m(x,\xi) = 0 \}$): if $u$ solves $Pu = 0$ in $X$, then the wave front set of $u$ is contained in $\mathrm{Char}(P)$ and it is invariant under the bicharacteristic flow there. In the definition of real principal type operators, the condition that no complete null bicharacteristic of $P$ stays over a compact subset of $X$ is a \emph{nontrapping condition} for the bicharacteristic flow. It ensures that the equation $Pu = f$ can be solved in any compact set if $f$ satisfies finitely many linear constraints. These facts from \cite{DH72} and \cite[Section 26.1]{Hormander} are recalled in Section \ref{sec_preliminaries}.

We wish to study inverse boundary value problems for real principal type operators. In order to avoid  technicalities related to the presence of a boundary, we will mostly focus on \emph{differential operators} on \emph{compact manifolds with smooth boundary}. We will next examine the real principal type condition and introduce several related definitions in the boundary case.

Let $M$ be a compact manifold with smooth boundary and let $P$ be a differential operator of order $m \geq 1$ having smooth coefficients on $M$. Assume that $P$ has real principal symbol $\sigma_{\pr}[P] = p_m$. Then for any $(x,\xi) \in \mathrm{Char}(P)$, there is a maximal interval $[-\tau_-(x,\xi),\tau_+(x,\xi)]$, where $\tau_{\pm}(x,\xi) \in [0,\infty]$, such that the null bicharacteristic curve 
\[
\gamma = \gamma_{x,\xi}: [-\tau_-(x,\xi),\tau_+(x,\xi)] \to T^*M \setminus 0
\]
with $\gamma_{x,\xi}(0) = (x,\xi)$ cannot be extended to a strictly larger closed interval as a smooth bicharacteristic curve in $T^* M \setminus 0$. A bicharacteristic curve is called \emph{maximal} if it is defined in its maximal interval. %, and \emph{regular} if the curve $x(t) = \pi(\gamma(t))$ satisfies $\dot{x}(t) \neq 0$ for all $t$. A null bicharacteristic curve is called \emph{nontangential} if 
%\begin{itemize}
%\item 
%it is maximal, and 
%\item 
%$\dot{x}(t)$ is not tangential to $\p M$ at the endpoints (but can be tangential at other points).
%\end{itemize}
The null bicharacteristic flow for $P$ is called \emph{nontrapping} if $\tau_{\pm}(x,\xi) < \infty$ for every  $(x,\xi) \in \mathrm{Char}(P)$. The operator $P$  
 is said to be of \emph{real principal type} on a compact manifold $M$ with boundary if it has real principal symbol, and if the null bicharacteristic flow is nontrapping. 
%\HOX{repetition}
We remark that the above definition is compatible with that of real principal type operators on an open manifold, in the sense that any real principal type differential operator $P$ on a compact manifold $M$ with boundary can be extended smoothly as a real principal type differential operator in some open manifold $X$ containing $M$ (see \cite[proof of Theorem 26.1.7]{Hormander}).

We now define the (null bicharacteristic) scattering relation for $P$. This will be a map that takes the initial point of a maximal null bicharacteristic to its end point, and vice versa. The definition is chosen so that it allows the bicharacteristics to have arbitrary tangential intersections with the boundary.

\begin{DefinitionNoNumber}
Let $M$ be a compact manifold with boundary, and let $P$ be a real principal type differential operator in $M$. The (null) \emph{incoming} and \emph{outgoing} boundaries are defined as 
\begin{align*}
\p_{\mathrm{null}}^{\pm}(T^* M) &:= \{ (x,\xi) \in \mathrm{Char}(P) \ ;\, \tau_{\mp}(x,\xi) = 0 \}.
\end{align*}
We set $\p'_{\mathrm{null}}(T^* M) := \p_{\mathrm{null}}^+(T^* M) \cup \p_{\mathrm{null}}^-(T^* M)$. The \emph{scattering relation} for $P$ is the map 
\begin{gather*}
\alpha_P: \p'_{\mathrm{null}}(T^* M) \to \p'_{\mathrm{null}}(T^* M), \\
\alpha_P(x,\xi) := \left\{ \begin{array}{cl} \gamma_{x,\xi}(\tau_+(x,\xi)), & (x,\xi) \in \p_{\mathrm{null}}^+(T^* M), \\[3pt] \gamma_{x,\xi}(-\tau_-(x,\xi)), & (x,\xi) \in \p_{\mathrm{null}}^-(T^* M). \end{array} \right.
\end{gather*}
\end{DefinitionNoNumber}

It is immediate that $\p'_{\mathrm{null}}(T^* M) \subset \p(T^* M) \cap \mathrm{Char}(P)$, and that the scattering relation has the following properties:
\begin{gather*}
\alpha_P(\p_{\mathrm{null}}^{\pm}(T^* M)) = \p_{\mathrm{null}}^{\mp}(T^* M), \\
\alpha_P \circ \alpha_P = \mathrm{Id}.
\end{gather*}

The next step is to consider boundary measurements for $P$. For second order operators it is customary to consider a Dirichlet-to-Neumann map, which maps the Dirichlet boundary value (on a suitable part of the boundary) of a solution of $Pu = 0$ in $M$ (the outgoing one for wave equations) to the Neumann boundary value. For a general real principal type operator $P$ of order $m \geq 1$, it is not clear if there is a suitable well-posed boundary value problem or a corresponding boundary map. However, one can always define boundary measurements in terms of the Cauchy data set.

For the next definition and for later purposes, we fix an auxiliary smooth Riemannian metric $g$ on $M$ and let $\nabla$ be the covariant derivative with respect to $g$. Then $\nabla^k u$ is the $k$th total covariant derivative of $u$, and the Cauchy data set below encodes the derivatives of solutions of $Pu = 0$ in $M$ up to order $m-1$ on the boundary.

\begin{DefinitionNoNumber}
The \emph{Cauchy data set} of $P$ is defined as 
\[
C_P = \{ (u|_{\p M}, \nabla u|_{\p M}, \ldots, \nabla^{m-1} u|_{\p M}) \,;\, u \in H^m(M) \text{ solves } Pu = 0 \text{ in $M$} \}.
\]
We say that two differential operators $P_1$ and $P_2$ \emph{agree to infinite order on $\p M$}, if $(P_1 - P_2)w$ vanishes to infinite order on $\p M$ for any $w \in C^{\infty}(M)$.
\end{DefinitionNoNumber}

Note that even though $C_P$ depends on the choice of $g$, the property $C_P=C_{\tilde P}$ for two operators is independent of $g$.

\addtocontents{toc}{\SkipTocEntry}
\subsection{Determining the scattering relation and bicharacteristic ray transforms}

Let $P$ be a differential operator of order $m$ on $M$ and let $\mu$ be a nonvanishing half density on $M$. We obtain a differential operator $P^\mu$ acting on half densities by the following definition 
    \begin{align}\label{def_Pmu}
P^\mu = \mu P \mu^{-1},
    \end{align}
where the operators $\mu: u \mapsto u \mu$ and $\mu^{-1}: u \mu \mapsto u$ give an isomorphism between functions and half densities. 
Writing $\sum_{j=0}^m p_j^\mu$ for the polyhomogeneous full symbol of $P^\mu$ in a local coordinate system, 
the subprincipal symbol 
    \begin{align}\label{def_subprin_symb}
\sigma_\subp[P^\mu] = p_{m-1}^\mu + \frac{i}{2} \sum_{j=1}^n \p_{x_j \xi_j} p_m^\mu
    \end{align}
is an invariantly defined function on $T^* M \setminus 0$, see e.g.\ \cite[Theorem 18.1.33]{Hormander}.

We can now state the first main theorem, showing that the boundary measurements for any real principal type differential operator $P$ uniquely determine the scattering relation which depends on the principal symbol $p_m$. Moreover, if $p_m$ is known then the integrals of the subprincipal symbol $\sigma_\subp[P^\mu]$ 
over maximal null bicharacteristics are uniquely determined modulo $2 \pi \mathbb Z$.

The result assumes that the coefficients of $P$ have already been determined up to infinite order on $\p M$ (after suitable gauge transformations); this is typically done using boundary determination results, which are discussed later.

\begin{Theorem} \label{thm_main1}
Let $M$ be a compact manifold with smooth boundary, and let $P_1, P_2$ be real principal type differential operators of order $m \geq 1$ on $M$. If 
\[
C_{P_1} = C_{P_2}
\]
and if $P_1 = P_2$ to infinite order on $\p M$, then 
\begin{align*}
\p_{\mathrm{null},P_1}'(T^*M) &= \p_{\mathrm{null},P_2}'(T^*M), \\
\alpha_{P_1} &= \alpha_{P_2}.
\end{align*}
Moreover, if the principal symbols of $P_1$ and $P_2$ coincide, then
for any nonvanishing half density $\mu$ on $M$ one has 
   \begin{align}\label{parallel_tr}
\exp \left[ i\int_0^T \sigma_\subp[P_1^\mu](\gamma(t)) \,dt \right] 
= 
\exp \left[ i\int_0^T \sigma_\subp[P_2^\mu](\gamma(t)) \,dt \right]
    \end{align}
whenever $\gamma : [0,T] \to T^* M$ is a maximal null bicharacteristic curve for $P_1$.
\end{Theorem}

\begin{RemarkNoNumber}
The equation (\ref{parallel_tr}) can be formulated equivalently without half densities as 
    \begin{align*}
\exp \left[ i\int_0^T (p_{m-1,1} - p_{m-1,2})(\gamma(t)) \,dt \right] 
= 0,
    \end{align*}
where $\sum_{j=0}^m p_{j,1}$ and $\sum_{j=0}^m p_{j,2}$ are polyhomogeneous full symbols of $P_1$ and $P_2$, acting on functions, in a local coordinate system. 
The above expression is coordinate invariant by Lemma \ref{lem_coordinv_subdiff} under the assumptions above, and the equivalence follows from \eqref{subprin_diff} below.
\end{RemarkNoNumber}

The conclusion \eqref{parallel_tr} is of course equivalent to  
\[
\int_0^T \sigma_\subp[P_1^\mu](\gamma(t)) \,dt = \int_0^T \sigma_\subp[P_2^\mu](\gamma(t)) \,dt \quad \text{ modulo $2\pi \mZ$}.
\]
This nonuniqueness modulo $2\pi \mZ$ is related to the \emph{Aharonov-Bohm effect}, which appears when one tries to determine subprincipal terms on domains with nontrivial topology (see Lemma \ref{lemma_aharonov_bohm_general}).

The next result shows that it is also possible to determine bicharacteristic ray transforms of the coefficients of order $\leq m-2$, and that no nonuniqueness modulo $2\pi \mZ$ appears in this case.

\begin{Theorem} \label{thm_main2}
Let $M$ be a compact manifold with smooth boundary, let $P$ be a real principal type differential operator of order $m \geq 2$ on $M$, and let $Q_1$ and $Q_2$ be differential operators of order $\le m - 2$ on $M$ with $Q_1=Q_2$ to infinite order on $\p M$. If 
\[
C_{P+Q_1} = C_{P+Q_2},
\]
then 
    \begin{align}\label{ray_tr}
\int_{0}^T \sigma_{\pr}[Q_1](\gamma(t)) \,dt = \int_{0}^T \sigma_{\pr}[Q_2](\gamma(t)) \,dt 
    \end{align}
whenever $\gamma : [0,T] \to T^* M$ is a maximal null bicharacteristic curve for $P$.
\end{Theorem}

\begin{RemarkNoNumber}
Note that $\sigma_{\pr}[Q_j](x,\xi) = q_j^{j_1\dots j_k}(x) \xi_{j_1} \cdots \xi_{j_k}$ is a polynomial in $\xi$ in local coordinates. Hence \eqref{ray_tr} may be written as 
\[
\int_{0}^T q_1^{j_1\dots j_k}(x(t)) \xi_{j_1}(t)\dots \xi_{j_k}(t) \,dt = \int_{0}^T q_2^{j_1\dots j_k}(x(t)) \xi_{j_1}(t)\dots \xi_{j_k}(t) \,dt
\]
where $\gamma(t) = (x(t), \xi(t))$.
\end{RemarkNoNumber}

Note that the above theorems are rather general, in the sense that they are valid for 
\begin{itemize}
\item 
differential operators of any order, with real principal symbol and satisfying the nontrapping condition (in particular, no wellposedness assumptions are required); and 
\item 
any maximal null bicharacteristic, possibly with cusp points (such as for Tricomi operators) or tangential intersections with the boundary.
\end{itemize}

The proofs of the above theorems extend in a straightforward way to the case where one only has access to measurements on a subset $\Gamma \subset \partial M$. Then one obtains information along null bicharacteristics that do not meet $\p M$ away from $\Gamma$. We refer to Section \ref{sec_scattering_relation} for precise statements.

Another related case is that of hyperbolic operators, for instance the wave operator in a space-time cylinder $M_0 \times (0,T)$, where the boundary measurements are typically given on the lateral boundary $\p M_0 \times (0,T)$ for outgoing solutions that vanish near $t=0$. This case can also be included with small modifications, and it is a special case of the following result.

Consider a differential operator $P$ of order $m$ on a smooth open manifold $X$ equipped with a smooth function $\phi$. We say that $P$ is \emph{strictly hyperbolic} with respect to the level surfaces of $\phi$ (see \cite[Definition 23.2.3]{Hormander}) if, writing $p_m = \sigma_{\pr}[P]$, one has $p_m(x,d\phi(x)) \ne 0$ for $x \in X$ and if the polynomial $\tau \mapsto p_m(x, \xi + \tau d\phi(x)) = 0$ has $m$ distinct real roots for all $(x,\xi) \in T^* X$ such that $\xi$ is linearly independent from $d\phi(x)$. Without loss of generality we may assume that $p_m$ is real valued (see \cite[Lemma 8.7.3]{Hormander}).

Let $M_1 \subset X$ be a compact domain with smooth boundary and 
let $T > 0$.
We consider the cylinder%\footnote{Note that replacing $\phi$ by $\phi + c$ where $c$ is a constant gives the same notion of global hyperbolicity. Thus there is no loss of generality in using $0$ as the lower bound.} 
    \begin{align}\label{def_M_cylinder_intro}
M = \{x \in M_1 \,;\, 0 \le \phi(x) \le T\},
    \end{align}
and write $\Gamma = \p M \cap \p M_1$, $\Gamma_- = \p M \cap \{x \in M_1 \,;\, \phi(x) = 0\}$ and $\Gamma_+ = \p M \cap \{x \in M_1 \,;\, \phi(x) = T\}$.
If $\Sigma \subset \p M$, we use the shorthand notation
    \begin{align*}
\tr_{\Sigma}^{m-1} u = (u|_{\Sigma}, \nabla u|_{\Sigma}, \ldots, \nabla^{m-1} u|_{\Sigma}).
    \end{align*}
 Finally we define the lateral Cauchy data set
    \begin{align*}
C_P^\lat = \{ \tr_{\Gamma}^{m-1} u \,;\, u \in H^m(M) \text{ solves } Pu = 0 \text{ in $M$ and $\tr_{\Gamma_-}^{m-1} u = 0$} \}.
    \end{align*}

We have the following analogue of Theorems \ref{thm_main1} and  \ref{thm_main2}.

\begin{Theorem} \label{thm_main_hyper_intro}
Let $M$ be as in \eqref{def_M_cylinder_intro}, and let $P_1, P_2$ be differential operators of order $m \geq 1$,
strictly hyperbolic with respect to the level surfaces of $\phi$.
Suppose that $P_1 = P_2$ to infinite order on $\Gamma$.

If $C_{P_1}^\lat = C_{P_2}^\lat$, then 
\begin{align*}
\alpha_{P_1}(x,\xi) = \alpha_{P_2}(x,\xi)
\end{align*}
for any $(x,\xi) \in \p'_{\mathrm{null},P_1}(T^* M)$ such that both for $j=1,2$, the maximal null $P_j$-bicharacteristic curve through $(x,\xi)$ does not meet $\p(T^* M)$ away from $\Gamma$.

Moreover, if the principal symbols of $P_1$ and $P_2$ coincide, then
for any nonvanishing half density $\mu$ on $M$ one has 
   \begin{align}\label{hyper_parallel_tr}
\exp \left[ i\int_0^T \sigma_\subp[P_1^\mu](\gamma(t)) \,dt \right] 
= 
\exp \left[ i\int_0^T \sigma_\subp[P_2^\mu](\gamma(t)) \,dt \right]
    \end{align}
whenever $\gamma : [0,T] \to T^* M$ is a maximal null bicharacteristic curve for $P_1$ whose spatial projection does not meet $\p M$ away from $\Gamma$.

Finally, if $m \geq 2$ and if $P_j = P + Q_j$ where $Q_j$ has order $\leq m-2$ for $j=1,2$, then 
  \[
\int_{0}^T \sigma_{\pr}[Q_1](\gamma(t)) \,dt = \int_{0}^T \sigma_{\pr}[Q_2](\gamma(t)) \,dt 
\]
whenever $\gamma : [0,T] \to T^* M$ is a maximal null bicharacteristic curve for $P$ whose spatial projection does not meet $\p M$ away from $\Gamma$.
\end{Theorem}

\addtocontents{toc}{\SkipTocEntry}
\subsection{Boundary determination results}

We proceed to the statement of boundary determination results. Given $(x, \xi) \in T^*(\p M)$, the boundary determination results at $x$ will depend on properties of the characteristic polynomial in the normal direction, i.e.\ on the polynomial 
\[
\tau \mapsto p_m(x,\xi+\tau \nu),
\]
where $\nu$ is the \emph{inward pointing unit conormal} to $\p M$ with respect to the auxiliary metric $g$. 
This is a polynomial of order $m$ with real coefficients, and hence has $m$ complex roots counted with multiplicity so that the non-real roots occur in complex conjugate pairs. We will employ two methods for determining boundary values at $x$:
\begin{itemize}
\item 
(Elliptic region) If $\tau \mapsto p_m(x,\xi + \tau \nu)$ has a simple non-real root, we use exponentially decaying solutions that concentrate near $x$ to give an analogue of boundary determination results for second order elliptic equations.
\item 
(Hyperbolic region) If $\tau \mapsto p_m(x,\xi + \tau \nu)$ has at least two distinct real roots, we use solutions concentrating near two null bicharacteristics through $x$ and obtain an analogue of boundary determination results for the wave equation.
\end{itemize}
In fact the regions could be mixed, and we will use a combination of both methods. The terminology above is analogous to the case of second order operators \cite[Section 24.2]{Hormander}: if $P$ is of second order with real principal symbol $p_2$ and the boundary is noncharacteristic, then the elliptic region (resp.\ hyperbolic region) is the set of those points in $(x,\xi) \in T^*(\p M)$ such that the map $\tau \mapsto p_2(x,\xi+\tau \nu)$ has no real zeros (resp.\ has two distinct real zeros).

The next result proves that the Cauchy data set determines the Taylor series of the principal symbol of $P$ at $x$, modulo a natural gauge invariance given by $C_P = C_{cP}$ where $c \in C^{\infty}(M,\mR)$ is nonvanishing. The result assumes that there is some $\xi \in T_x^*(\p M)$ such that $\tau \mapsto p_m(x,\xi+\tau \nu)$ has only simple roots, with an additional geometric condition for the real roots.

\begin{DefinitionNoNumber}
We say that two smooth curves $\gamma_j = (x_j, \xi_j): [0,T_j] \to T^* M$ \emph{intersect nicely} at $x \in \p M$ if $x_1(0) = x_2(0) = x$ with $\dot{x}_1(0)$ and $\dot{x}_2(0)$ linearly independent, if the curves $x_j(t)$ never intersect at the boundary again, and if $\xi_1(t) \neq \xi_2(s)$ when $x_1(t) = x_2(s)$.
\end{DefinitionNoNumber}

\begin{ExampleNoNumber}
Let $P = \p_t^2 - \Delta_{g_0}$ be the wave operator in the space-time cylinder $M = M_0 \times (0,T)$ where $(M_0,g_0)$ is a compact Riemannian manifold with smooth boundary, and let $(x_0, t_0) \in \p M_0 \times (0,T)$. Consider two null bicharacteristics $\gamma_1$ and $\gamma_2$ through $(x_0,t_0)$ so that $\gamma_1$ goes forward in time and $\gamma_2$ backward in time. Then $\gamma_1$ and $\gamma_2$ intersect nicely at $(x_0,t_0)$. A similar picture holds for strictly hyperbolic operators.
\end{ExampleNoNumber}

\begin{Theorem} \label{thm_boundary_determination_principal_intro}
Let $M$ be a compact manifold with smooth boundary and let $P_1, P_2$ be real principal type differential operators of order $m \geq 2$. Assume that  
\[
C_{P_1} = C_{P_2}.
\]
Suppose that $x_0 \in \p M$ is noncharacteristic for $P_j$, and for some $\xi_0 \in T_{x_0}^*(\p M)$ the maps $\tau \mapsto p_{m,j}(x_0,\xi_0+\tau \nu)$ have only simple roots. Moreover, assume one of the following conditions:
\begin{enumerate}
\item[(1)] 
some root is non-real, and no bicharacteristic corresponding to a real root returns to $x_0$; or 
\item[(2)] 
all roots are real, and whenever $p_{m,1}(x_0,\xi_0+\tau_1 \nu) = p_{m,2}(x_0,\xi_0+\tau_2 \nu) = 0$ for $\tau_1 \neq \tau_2$ the bicharacteristics through $(x_0,\xi_0+\tau_1 \nu)$ and $(x_0,\xi_0+\tau_2 \nu)$ intersect nicely at $x_0$.
\end{enumerate}
Then there is a nonvanishing function $c \in C^{\infty}(M)$ so that in boundary normal coordinates $(x',x_n)$ (with respect to the background Riemannian metric) one has 
\begin{equation} \label{boundary_determination_principal_claim}
\p_{x_n}^j ( p_{m,2}(x,\eta) - c(x) p_{m,1}(x,\eta) )|_{x=x_0} = 0, \qquad j \geq 0, \ \eta \in T_{x_0}^* M.
\end{equation}
\end{Theorem}

%\PS{boundary normal coordinates are not defined so far}

The next result considers boundary determination of Taylor series of lower order terms. It turns out that to determine coefficients of order $r$, one needs that $\tau \mapsto p_m(x_0,\xi_0 + \tau \nu)$ has at least $r+1$ simple roots, with an additional geometric condition for the real roots.

\begin{Theorem} \label{thm_boundary_determination_lower_order_intro}
Let $M$ be compact with smooth boundary, let $P$ be a real principal type differential operator of order $m \geq 2$, and let $Q_1, Q_2$ be differential operators of order $r \leq m-1$. Assume that  
\[
C_{P+Q_1} = C_{P+Q_2}.
\]
Suppose that for some $(x_0,\xi_0) \in T^*(\p M)$ the map $\tau \mapsto p_m(x_0,\xi_0+\tau \nu)$ has at least $s$ simple roots $\tau_1, \ldots, \tau_s$ with nonnegative imaginary parts, and assume one of the following two conditions:
\begin{enumerate}
\item[(1)] 
$s \geq r+1$, one of the $\tau_j$ is non-real, and no bicharacteristic corresponding to a real $\tau_j$ returns to $x_0$; or 
\item[(2)] 
$s \geq \max(r+1,2)$, each $\tau_j$ is real, and for $j \neq k$ the bicharacteristics through $(x_0,\xi_0+\tau_j \nu)$ and $(x_0,\xi_0+\tau_k \nu)$ intersect nicely at $x_0$.
\end{enumerate}
Then in boundary normal coordinates 
\[
\p_{x_n}^j \sigma_{\pr}[Q_1](x_0,\eta) = \p_{x_n}^j \sigma_{\pr}[Q_2](x_0,\eta), \qquad j \geq 0, \ \eta \in T_{x_0}^* M.
\]
Moreover, if $Q_1-Q_2$ has real principal symbol, the result is valid without the assumption that the roots $\tau_j$ have nonnegative imaginary parts.
\end{Theorem}

The results above give an immediate consequence for determining zero order terms.

\begin{Theorem} \label{thm_main_potential_boundary}
Let $M$ be a compact manifold with smooth boundary, and let $P$ be a real principal type differential operator of order $m \geq 2$. Let $V_1, V_2 \in C^{\infty}(M)$, and assume that 
\[
C_{P+V_1} = C_{P+V_2}.
\]
Then the Taylor series of $V_1$ and $V_2$ agree at any point $x \in \p M$ so that for some $\xi \in T_x^*(\p M)$, the map $\tau \mapsto p_m(x,\xi+\tau\nu)$ either has a simple non-real root, or two distinct real roots $\tau_1$ and $\tau_2$ so that the corresponding bicharacteristics intersect nicely at $x$.

In particular, if $M$, $V_1$ and $V_2$ are real-analytic and one such point $x$ exists, then $V_1 = V_2$ in $M$.
\end{Theorem}

The approach for proving the above results yields the following, perhaps surprising observations:
\begin{itemize}
\item 
It is possible to perform boundary determination for general real principal type operators, including wave operators, just as for the Laplace operator by working in the elliptic region. The method also works for elliptic operators of any order if sufficiently many roots are simple.
\item 
Boundary determination in the elliptic region is local in character (one does not need any global assumption on behaviour of bicharacteristics).
\item 
Boundary determination in the hyperbolic region may be local or global in character. If two null bicharacteristics only meet at their initial point (e.g.\ they go in opposite time directions), one recovers the Taylor series at that point. If they meet at two distinct boundary points, then one recovers a sum of contributions from these two points.
\item 
Boundary determination for the wave operator is possible also at the bottom of a space-time cylinder, and this argument is global in character. Measuring local Cauchy data at the bottom is not enough to determine the Taylor series of coefficients, since the Cauchy problem is well-posed (i.e.\ Cauchy data at the bottom does not carry any information about the coefficients). However, it is possible to use two null bicharacteristics intersecting at a point in the bottom, and information measured at the lateral boundary can be used to determine the Taylor series of coefficients at the bottom.
\end{itemize}

\addtocontents{toc}{\SkipTocEntry}
\subsection{Nonlinear equations}

We will next state our results for nonlinear equations whose principal part is a linear real principal type operator. In this case, instead of obtaining conditional results that reduce the problem of recovering coefficients to inverting a scattering relation or ray transform, we can actually determine the unknown coefficients at all points satisfying a geometric condition. Here we adapt the methods initiated in the case of Einstein equations and nonlinear wave equations \cite{KLU_nonlinear, KLOU, LLLS, FO}, which are based on higher order linearizations and use the fact that the presence of nonlinear interactions may make the inverse problem easier to solve.

The model equation that we consider here is 
\[
Pu + a(x,u) = 0 \text{ in $M$}
\]
where $P$ is a linear real principal type differential operator of order $m$ and $a(x,u)$ is a nonlinearity satisfying the following conditions: for a fixed integer $s > \max(m, n/2)$, \begin{gather}
\text{$a(x,z)$ is analytic in $z$ near $0$ as a $H^s(M)$-valued function}, \label{semilinear_condition_one_intro} \\
a(x,0) = \p_z a(x,0) = \p_z^2 a(x,0) = 0. \label{semilinear_condition_two_intro}
\end{gather}
The condition $\p_z^2 a(x,0) = 0$ simplifies the proof, and possibly could be removed with extra work. The condition $\p_z a(x,0) = 0$ however is important, since it ensures that the equation is not linear (if one removes this assumption then one could only determine the genuinely nonlinear parts, i.e.\ $\p_z^j a(x,0)$ for $j \geq 2$).

We consider boundary measurements given in terms of the Cauchy data set with $\delta > 0$ small, 
\[
C_{a,\delta} = \{(u|_{\p M}, \nabla u|_{\p M}, \ldots, \nabla^{m-1} u|_{\p M}) \,;\, u \in H^s(M), \ Pu + a(x,u) = 0, \ \text{and } \norm{u}_{H^s(M)} \leq \delta \}.
\]
Here it is instrumental to work in the regime of small solutions, since the argument relies on higher order linearizations around the zero solution. We will also need to make a weak uniqueness assumption for the linearized equation and its adjoint. The space 
\[
N(P) = \{u \in H^m_0(M) \,;\, Pu = 0 \text{ in $M$} \}
\]
is always finite dimensional (see Proposition \ref{prop_real_principal_type_solvability}). We will assume that this space is trivial, which roughly means that the linearized equation does not admit nontrivial compactly supported solutions.

\begin{Theorem} \label{thm_semilinear_uniqueness_intro}
Let $M$ be a compact manifold with smooth boundary, and let $P$ be a real principal type differential operator on $M$. Let $a$, $\tilde{a}$ satisfy \eqref{semilinear_condition_one_intro}--\eqref{semilinear_condition_two_intro}, and assume that $a$ and $\tilde{a}$ agree to high order on $\p M$ in the sense that 
\begin{equation*} %\label{a_atilde_agree_boundary_intro}
a(\,\cdot\,,z)-\tilde{a}(\,\cdot\,,z) \in H^s_0(M) \text{ for $z$ near $0$.}
\end{equation*}
Assume also that 
\[
N(P) = N(P^*) = \{0 \}.
\]
If for any sufficiently small $\delta > 0$ there is $\delta_1 < \delta$ so that 
\[
C_{a,\delta_1} \subset C_{\tilde{a},\delta},
\]
then 
\[
\text{$a(x_0,z) = \tilde{a}(x_0,z)$ for any $x_0 \in B$ and any $z$ near $0$},
\]
where 
\begin{align*}
B = \{x_0 \in M^{\mathrm{int}} \,;\, &\text{there are two maximal null bicharacteristics whose spatial }\\
 &\text{projections $x_j$, $j=1,2$, only intersect once at $x_0$ and 
}\\
 &\text{the tangent vectors of $x_j$ are linearly independent at $x_0$} \}.
\end{align*}
\end{Theorem}

\begin{ExampleNoNumber}
Let $M$ be a compact subdomain of the space-time cylinder $M_0 \times (0,T)$ where $(M_0,g_0)$ is a compact Riemannian manifold with boundary, and let $P = \p_t^2 - \Delta_{g_0}$ be the wave operator. If $(x_0, t_0) \in M$ and if there is a geodesic $\eta(s)$ in $M_0$ so that $\eta(0) = x_0$ and some Jacobi field along $\eta$ only vanishes when $s=0$, then by looking at the corresponding variation through geodesics one can show that $(x_0,t_0)$ is in the set $B$ of Theorem \ref{thm_semilinear_uniqueness_intro}. In particular, $(x_0,t_0) \in B$ if there is at least one geodesic through $x_0$ with no conjugate points, but one may have $(x_0,t_0) \in B$ even if all geodesics through $x_0$ have conjugate points as long as some variation field only vanishes when $s=0$. A similar argument can be given for general real principal type operators by looking at variation fields of null bicharacteristic curves. However, this argument fails in the special situation when $x_0 \in M_0$ and every $M_0$-geodesic through $x_0$ has a conjugate point $y$ of maximal order (i.e.\ the space of normal Jacobi fields vanishing both at $x_0$ and $y$ has maximal order $\dim(M_0)-1$). The sphere is the standard example of a manifold where any point is maximally conjugate to its antipodal point. If $M_0$ contains a neighborhood of the hemisphere, then the set $B$ in Theorem \ref{thm_semilinear_uniqueness_intro} would be empty.
\end{ExampleNoNumber}

\addtocontents{toc}{\SkipTocEntry}
\subsection{Special cases}

We next give a few basic examples illustrating the above results, related to 
\begin{itemize}
%\item the operator $D_1 + V$;
\item 
the boundary / scattering rigidity problem; and
\item 
the inverse boundary problem for the Lorentzian wave equation.
\end{itemize}

\begin{Example}
(Boundary / scattering rigidity) Let $(M_0, g_0)$ be a compact Riemannian manifold with strictly convex boundary, and assume that the geodesic flow in $(M_0,g_0)$ is nontrapping. Let $M := SM_0 = \{ (x,v) \in T M_0 \,;\, \abs{v}_{g_0} = 1 \}$ be the unit sphere bundle of $M_0$, and let 
\[
P := X + \mathcal{A}
\]
where $X$ is the geodesic vector field in $S M_0$ and $\mathcal{A} \in C^{\infty}(S M_0)$. Since $P$ is a first order operator, projections of null bicharacteristic curves to $M$ are integral curves of $P$ (i.e.\ geodesics in $S M_0$) and the nontrapping assumption ensures that $P$ is of real principal type. Moreover, the scattering relation $\alpha_P$ induces a corresponding relation $\beta: \p(S M_0) \to \p(S M_0)$ so that $\beta \circ \pi_M = \pi_M \circ \alpha_P$ where $\pi_M$ is the projection $T^* M \to M$. Thus $\beta$ is just the scattering relation for the metric $g_0$ in $M_0$. The Cauchy data set is 
\[
C_P = \{ u|_{\p(SM_0)} \,;\, (X+\mathcal{A})u = 0 \text{ in $SM_0$} \}.
\]
Theorem \ref{thm_main1} states that $C_P$ determines $\alpha_P$, and hence $\beta$, as well as the quantities 
\[
\int_{\eta} \mathcal{A}(\eta(t)) \,dt \text{ modulo $2\pi \mZ$}
\]
whenever $\eta$ is a maximal geodesic in $SM_0$. Conversely, the above information uniquely determines $C_P$ (see Theorem \ref{thm_first_order_equivalence}). These results are quite elementary since $P$ is of first order, but it is instructive to observe that the inverse boundary problem for $X + \mathcal{A}$ is equivalent to the scattering rigidity problem of determining a metric $g_0$ up to gauge from its scattering relation $\beta$, and to determining $\mathcal{A}$ from its geodesic X-ray transform. It is well known that a general function $\mathcal{A} \in C^{\infty}(SM_0)$ is not determined uniquely by its X-ray transform. In fact, the kernel of the X-ray transform on $C^{\infty}(S M_0)$ consists precisely of the functions $\mathcal{A} = X \mathcal{B}$ where $\mathcal{B} \in C^{\infty}(SM_0)$ with $\mathcal{B}|_{\p SM_0} = 0$ (the proof is an elementary argument combined with \cite[Proposition 5.2]{PSU2}). However, in applications $\mathcal{A}$ often arises from a function or a tensor field on the base manifold $M_0$ and there are many available results in this case.
\end{Example}

\begin{Example}
(Lorentzian wave equation) Let $(M,g)$ be a compact Lorentzian manifold with boundary, and consider the canonical wave operator $\Box_g$ on $M$. In local coordinates, 
\[
\Box_g u = -|g|^{-1/2} \p_{x_j} (|g|^{1/2} g^{jk}  \p_{x_k} u),
\]
where $\abs{g} = \det(g_{jk})$ and $(g^{jk})$ is the inverse of $(g_{jk})$. 
The principal symbol $p_2$ of $\Box_g$ is 
    \begin{align}\label{prin_symb_box}
p_2(x,\xi) = g^{jk} \xi_j \xi_k = \langle \xi, \xi \rangle_g.
    \end{align}
In this case, results similar to Theorems \ref{thm_main1}-\ref{thm_main_potential_boundary} are contained in \cite{StefanovYang}, with stability estimates. 

The operator $\Box_g$ is not necessarily of real principal type 
 if no further assumptions on the geometry is made. Writing $S^n$ for the unit sphere of dimension $n$, a counter-example is given by $M = S^1 \times M_0$, where $M_0$ is a compact domain with boundary on $S^n$, containing a closed geodesic, and $g(t,x) = -dt^2 + g_0(x)$ for $(t,x) \in S^1 \times M_0$ with $dt$ and $g_0$ the natural metrics on $S^1$ and $S^n$, respectively. Indeed, in this case there is a closed null geodesic on $M$, and $\Box_g$ is not of real principal type.

If all null geodesics exit $M$, then $\Box_g$ is of real principal type. This is the case, for example, if $(X,g)$ is a globally hyperbolic Lorentzian manifold and $M \subset X$ is a compact domain with smooth boundary. We refer to the books \cite{O,R} for the definition and properties of globally hyperbolic Lorentzian manifolds. The classical case where $M = M_0 \times (0,T)$ for some $T > 0$ and 
    \begin{align}\label{Lorentz_prod}
g(x,t) = -dt^2 + g_0(x),
    \end{align}
with $(M_0, g_0)$ a Riemannian manifold with boundary, is covered in Theorem \ref{thm_main_hyper_intro}. 

Let $a$ be a one form on $M$, and consider the wave operator 
$P = (d + i a)^* (d + i a)$. Here the adjoint is induced by the Lorentzian inner product $g$ and $d$ is the exterior derivative.
Then the principal symbol of $P$ is given by $p_2$ in (\ref{prin_symb_box}).
Defining a half density in a local coordinate system by $\mu = |g|^{1/4} \abs{dx_1 \wedge \ldots \wedge dx_n}^{1/2}$, consistent with the volume form, yields that 
    \begin{align*}
\sigma_\subp[P^\mu](x,\xi) = g^{jk}a_j \xi_k = \langle a, \xi \rangle_g.
    \end{align*}
The expression $\exp \left[ i\int_0^T \sigma_\subp[P^\mu](\gamma(t)) \,dt \right]$, cf. (\ref{parallel_tr}), gives the parallel transport map with respect to the connection $d + i a$ along the bicharacteristic $\gamma : [0,T] \to T^* M$ of $P$ as follows: if $z(t) = \pi(\gamma(t))$ where $\pi : T^* M \to M$ is the natural projection, and if $S(t)$ is a parallel section of the trivial bundle $M \times \mC$ along $z(t)$ (i.e.\ $\dot{S}(t) + ia(\dot{z}(t)) S(t) = 0$), then $S(T) = \exp \left[ -i\int_0^T \sigma_\subp[P^\mu](\gamma(t)) \,dt \right] S(0)$.

Let $V \in C^\infty(M)$ have principal symbol $v_0(x,\xi) = V(x)$, and suppose that $P$ has principal symbol $p_2$ in (\ref{prin_symb_box}). Then the map in \eqref{ray_tr} corresponds to 
    \begin{align*}
LV(\gamma) = \int_0^T v_0(\gamma(t)) dt = \int_0^T V(\pi(\gamma(t))) dt,
    \end{align*}
where $\gamma : [0,T] \to T^* M$ is a maximal bicharacteristic of $P$. This 
is the light ray transform of $V$. 
It is known that $L$ is invertible in the case of certain 
stationary, globally hyperbolic manifolds \cite{FIO}, and in the case of certain real analytic manifolds \cite{Plamen_lightray}.
We refer also to \cite{FIKO} for the case of product geometries
of the form (\ref{Lorentz_prod}).

We remark that the gauge invariances in Theorems \ref{thm_main1}-\ref{thm_main_potential_boundary}, when specialized to the Lorentzian wave equation, are slightly different from those in \cite{StefanovYang}. In the boundary determination result in Theorem \ref{thm_boundary_determination_lower_order_intro} there is no gauge invariance $a \to a - d\psi$ where $\psi \in C^{\infty}(M)$ satisfies $\psi|_{\p M} = 0$, whereas such a gauge invariance appears in \cite[Lemma 2]{StefanovYang}. This is due to the fact that our boundary measurements (Cauchy data set) for $P$ are defined in terms of a reference Riemannian metric on $M$, whereas \cite{StefanovYang} defines boundary measurements in terms of the Lorentzian metric $g$ and $1$-form $a$ appearing in $P$. The gauge invariances are discussed in more detail in Section \ref{sec_preliminaries}.
\end{Example}

\addtocontents{toc}{\SkipTocEntry}
\subsection{Methods}

The proofs of our theorems are essentially based on propagation of singularities for real principal type operators. If $P$ is such an operator, propagation of singularities is often understood as a regularity result stating that if a solution $u$ of $Pu = 0$ is smooth at some point on a null bicharacteristic curve, then $u$ has to be smooth at all points on this curve. This result has a complement stating that for any suitable non-trapped null bicharacteristic segment $\gamma$ for $P$, there is a function $u$ solving $Pu=0$ away from the end points of $\gamma$ so that the wave front set of $u$ is precisely on $\gamma$, see \cite{DH72} and \cite[Theorem 26.1.7]{Hormander}. This complementary result, which is basically a quasimode (or approximate solution) construction, produces interesting special solutions to the equation $Pu = 0$. It is this result and its variants that will be useful in the solution of inverse problems.

It is natural to define the Cauchy data set for sufficiently smooth solutions of $Pu = 0$. Thus the singular solutions mentioned above are not directly useful, but rather we will use a semiclassical version of such a quasimode construction. There are various methods for constructing quasimodes. If the null bicharacteristic segment has no cusps, a classical geometrical optics or Gaussian beam type construction is sufficient (see e.g.\ \cite{BabichLazutkin, Hormander1971, Ralston, Ralston_note}). The general case with cusps is more involved, and in \cite{DH72} it is dealt with by conjugating $P$ microlocally in a full neighborhood of the null bicharacteristic segment to the normal form $D_{x_1}$ using Fourier integral operators. In principle the quasimode construction in this article could be done by implementing a semiclassical version of the Fourier integral operator conjugation argument (see \cite{DKLS} for a special case). However, in the proofs of the main theorems we need to perform various computations with these quasimodes on manifolds with boundary. To handle the boundary effects it is beneficial to have a more direct construction. We will use a modification of the Gaussian beam construction based on representing the quasimode as an integral of Gaussian wave packets (coherent states) over the bicharacteristic as in \cite{PaulUribe, KarasevVorobjev}. This approach resolves the problems near cusps and gives an elementary and direct construction. In \cite{GuilleminUribeWang}, which describes a general framework, such quasimodes are called semiclassical states associated with isotropic submanifolds of phase space.

Given a quasimode $v$ concentrating near a null bicharacteristic, we employ the solvability theory of real principal type operators to produce exact solutions of $Pu = 0$ that are close to $v$. We use these special solutions to pass from the Cauchy data set $C_P$ to the scattering relation or null bicharacteristic ray transform by using two different methods: a mix-and-match construction combined with semiclassical propagation of singularities, or an integral identity (Lemma \ref{lemma_integral_identity_potential}) analogous to standard identities used in the Calder\'on and Gel'fand problems. For the boundary determination results we again use integral identities and solutions concentrating near null bicharacteristics in the hyperbolic region. However, in the elliptic region we need different special solutions that concentrate near a boundary point and are exponentially decreasing in the interior (Theorem \ref{prop_exponentially_localized_solution}). These solutions are analogous to those used in boundary determination for elliptic second order equations. Similar arguments appear in the construction of a boundary parametrix for the wave equation in the elliptic region (see e.g.\ \cite[Chapter IX]{Taylor1981}). The coefficients are eventually recovered via the complex stationary phase method, see \cite[Section 7.7]{Hormander} or \cite{MelinSjostrand} for the original approach.

The proof of Theorem \ref{thm_semilinear_uniqueness_intro} for nonlinear equations is based on higher order linearizations. The idea is that given solutions $v_1, \ldots, v_N$ of the linearized equation (i.e.\ $Pv_j = 0$) and small parameters $\eps_1, \ldots, \eps_N$, there is a small solution $u$ of the nonlinear equation $Pu + a(x,u) = 0$ so that $u$ is close to $\eps_1 v_1 + \ldots + \eps_N v_N$. Taking suitable derivatives with respect to $\eps_j$ and using the knowledge of the Cauchy data of small solutions, one recovers information about Taylor coefficients of $a(x,u)$ integrated against products of the solutions $v_j$. The key point is to use special solutions $v_j$ that concentrate along different null bicharacteristic curves, which leads to the geometric condition in Theorem \ref{thm_semilinear_uniqueness_intro}.

\addtocontents{toc}{\SkipTocEntry}
\subsection{Previous literature and outlook}

There is a large literature on inverse boundary problems of the above type, and we only give a few relevant references. The prototype for Theorems \ref{thm_main1}--\ref{thm_main_hyper_intro} are corresponding results for the wave equation, going back to \cite{RakeshSymes, Rakesh_geometric_optics, Stefanov_timedependent} in the Euclidean case. The scattering relation was considered in \cite{Guillemin_scattering_relation}, and \cite{Uhlmann_scattering_relation} outlines how to determine the scattering relation for a Riemannian metric from the hyperbolic Dirichlet-to-Neumann map. The case of Lorentzian wave equations is discussed in \cite{StefanovYang}, together with references to earlier work. Boundary determination results for wave equations may be found in \cite{SylvesterUhlmann_anisotropic, SU_generic, Montalto, StefanovYang}. Boundary determination for second order elliptic equations goes back to \cite{KohnVogelius, SylvesterUhlmann_boundary}, with \cite{LeeUhlmann, DKSaU} considering the geometric case.

Inverse problems for nonlinear hyperbolic equations have been studied in \cite{KLU_nonlinear, KLOU, LUW}, and we have also used ideas from related results for elliptic equations \cite{LLLS, FO}. We also mention the work \cite{Isakov_general} studying inverse boundary problems for certain general equations of the form $P(D)u + Vu = 0$, where $P(D)$ is a constant coefficient operator. The present article addresses similar questions for operators with variable coefficients.

As mentioned above, we consider this article to be a starting point in approaching inverse problems from a more general point of view. This point of view, in addition to the particular theorems stated here, is one of the main contributions of this article. It suggests many future directions including the following:

\begin{enumerate}
\item[1.] 
What kind of information about $P$ (more precisely, about the principal symbol $\sigma_{\pr}[P]$) can be determined from the null bicharacteristic scattering relation $\alpha_P$? If $P$ is the geodesic vector field on the unit sphere bundle of a compact Riemannian manifold with boundary, this is the scattering/lens/boundary rigidity problem studied in \cite{PestovUhlmann, StefanovUhlmannVasy, StefanovUhlmannVasy2}.
\item[2.] 
For which operators $P$ and for which classes of functions on $T^* M$ is the null bicharacteristic ray transform invertible? If $P$ is the geodesic vector field, this transform is the geodesic ray transform for which there is a substantial literature, see the survey paper \cite{IlmavirtaMonard}. If $P$ is the wave operator on a Lorentzian manifold, this transform is the light ray transform studied in 
\cite{FIKO, FIO, LOSU, Plamen_lightray, VasyWang}.
\item[3.] 
The real principal type condition includes a nontrapping assumption for the null bicharacteristic flow. Can one obtain results in the presence of (sufficiently mild) trapping? For the geodesic vector field and hyperbolic trapping, this has been studied in detail in \cite{Guillarmou_trapping}.
\item[4.] 
We have only considered uniqueness results stating that $C_P$ uniquely determines some information about $P$. Is it possible to study stability, reconstruction, range characterization, or further partial data cases? Moreover, instead of looking at scalar operators, can one give similar results for real principal type systems?
\item[5.] 
In this article we have studied the consequences of propagation of singularities for inverse problems for real principal type operators. What other classes of operators or general mechanisms for inverse problems could be studied in this way?
\item[6.] 
In examples such as the Calder\'on or Gel'fand problem, knowing the Cauchy data set is equivalent to knowing a boundary map (Dirichlet-to-Neumann map) which is a pseudodifferential or a Fourier integral operator. In these cases one can determine information about the coefficients by symbol computations. Our proofs are similar, but we use integral identities as a substitute for symbol computations. If $P$ is a general operator, is it possible to associate a symbol or Lagrangian manifold directly to the Cauchy data set $C_P$?
\item[7.] 
The symbol computations mentioned in the previous item only use the singularities of the integral kernel of the boundary map. The arguments in this article are in a similar spirit. Is it possible to extract information from the $C^{\infty}$ part of the integral kernel? This is what happens in the Calder\'on problem, see e.g.\ \cite{Salo_normalforms} for a related discussion. Also the Boundary Control method, originating from \cite{BelishevBC}, uses the smooth part in the case of wave equations with time-independent coefficients. 
\end{enumerate}

The rest of the article is organized as follows. In Section \ref{sec_preliminaries} we collect some preliminaries related to real principal type operators, semiclassical wave front sets, and Cauchy data sets. Section \ref{sec_quasimode_construction} gives the construction of semiclassical quasimodes associated with null bicharacteristic segments. Section \ref{sec_scattering_relation} proves Theorems \ref{thm_main1}--\ref{thm_main_hyper_intro} related to determining the scattering relation and bicharacteristic ray transforms from the Cauchy data set. The boundary determination results, Theorems \ref{thm_boundary_determination_principal_intro}--\ref{thm_main_potential_boundary}, are established in Section \ref{sec_boundary_determination}. Finally, Section \ref{sec_semilinear} considers semilinear equations and proves Theorem \ref{thm_semilinear_uniqueness_intro}. The proofs of some auxiliary results are given in Appendix \ref{sec_appendix}.

\addtocontents{toc}{\SkipTocEntry}
\subsection*{Acknowledgements}
L.O.\ was supported by EPSRC grants EP/P01593X/1 and EP/R002207/1.
M.S.\ was supported by the Academy of Finland (Finnish Centre of Excellence in Inverse Modelling and Imaging, grant numbers 312121 and 309963) and by the European Research Council under Horizon 2020 (ERC CoG 770924). P.S.\ was partly supported by  NSF  Grant DMS-1900475. G.U.\ was partly supported by NSF, a Si-Yuan Professorship at HKUST and a Walker Family Professorship at UW. This material is based upon work supported by the Clay Foundation and NSF under Grant No.\ 1440140, while M.S.\ and G.U.\ were in residence at MSRI in Berkeley, California, during the semester on Microlocal Analysis in 2019.

\section{Preliminaries} \label{sec_preliminaries}

\addtocontents{toc}{\SkipTocEntry}
\subsection*{Real principal type operators}

We begin by recalling basic facts about real principal type operators from \cite[Section 26.1]{Hormander}.

First assume that $P \in \Psi^m(X)$ is properly supported and has real homogeneous principal symbol $p = \sigma_{\pr}[P]$ of order $m$. The bicharacteristic curves of $P$ are integral curves of the Hamilton vector field $H_{p}$ on $T^* X \setminus 0$. In local coordinates, a bicharacteristic curve $\gamma(t) = (x(t), \xi(t))$ solves the Hamilton equations 
\begin{align*}
\dot{x}(t) &= d_{\xi} p(x(t), \xi(t)), \\
\dot{\xi}(t) &= -d_x p(x(t), \xi(t)).
\end{align*}
The function $p$ is constant along any bicharacteristic curve. In particular, the bicharacteristic curves in the characteristic set $p^{-1}(0)$ are called null bicharacteristic curves. A major feature of operators with real principal symbol is the fact that singularities propagate along null bicharacteristics:  if $Pu = f$ in $X$, then $\WF(u) \setminus \WF(f)$ is contained in $p^{-1}(0)$ and is invariant under the flow of $H_p$.

The operator $P$ is of real principal type if it additionally satisfies a nontrapping assumption: no bicharacteristic curve stays over a compact subset of $X$ infinitely long. Combined with propagation of singularities, this implies that any  solution  of $Pu \in C^\infty$ in $X$ with a compact singular set has to be smooth.

\begin{Remark}
Several different conditions related to real principal type operators appear in the literature, such as $d_{\xi} p \neq 0$, or that $dp$ and the canonical $1$-form $\xi_j \,dx^j$ are linearly independent (equivalently, $H_p$ is not \emph{radial}, i.e.\ proportional to the radial direction $\xi_j\,\p_{\xi_j}$). We wish to clarify the relations between these conditions. Clearly 
\[
d_{\xi} p \neq 0 \implies \text{$H_p$ is not radial}.
\]
Moreover, if $d_{\xi} p(x_0,\xi) \neq 0$ then the bicharacteristic through $(x_0,\xi)$ moves away from $x_0$. Thus 
\[
d_{\xi} p \neq 0 \text{ at any $(x_0, \xi) \in p^{-1}(0)$} \implies \text{$P$ is real principal type near $x_0$.}
\]
In the converse direction, 
\[
\text{$P$ is real principal type} \implies \text{$H_p$ is not radial anywhere on $p^{-1}(0)$},
\]
since if $H_p$ were radial at some $(x_0, \xi_0) \in p^{-1}(0)$, then there would be a bicharacteristic $\gamma(t) = (x_0, \xi(t))$ staying over $x_0$ for infinite time. The Tricomi operator $P = x_2 \p_1^2 + \p_2^2$ is of real principal type, but does not satisfy $d_{\xi} p \neq 0$ on $p^{-1}(0)$ (in fact null bicharacteristic curves may have cusps).
\end{Remark}

%\tbl{[MS: \cite{Hormander} states that having $dp$ and $\xi_j \,dx^j$ linearly independent on $p^{-1}(0)$ is locally equivalent to real principal type, but I could not see this. In fact $p(x,\xi) = \xi_1^3 + (\xi_1^2+\xi_2^2)(x_2 \xi_3 - x_3 \xi_2)$ with bichacteristic $x(t) = 0$, $\xi(t) = (0,\cos t, \sin t)$ may give a counterexample.]}

We will next give a solvability result for real principal type operators. This is a direct adaptation of \cite[Theorem 26.1.7]{Hormander}, but we will state it in a form adapted to manifolds with boundary.

Let us introduce some notation for Sobolev spaces and distributions. Let $X$ be an open manifold and let $M$ be a compact subdomain with boundary. We fix some Riemannian metric on $X$ with volume form $dV$, and consider Sobolev spaces $H^s(M)$ and $H^s_{\mathrm{loc}}(X)$ with respect to the $L^2$ norm induced by $dV$. Here we slightly abuse notation and write $H^s(M)$ instead of $H^s(M^{\mathrm{int}})$, where the latter space can be defined as the restriction of $H^s_{\mathrm{loc}}(X)$ to $M^{\mathrm{int}}$. 

The distribution space $\mDp(X)$ is the set of continuous linear functionals on $C^{\infty}_c(X)$, and we identify functions $u \in L^1_{\mathrm{loc}}(X)$ with distributions using the pairing 
\[
( u, \varphi ) = \int_X u \bar{\varphi} \,dV, \qquad \varphi \in C^{\infty}_c(X).
\]
We continue to write $( u, \varphi )$ for the distributional pairing. Here we deviate from standard practice by requiring that the pairing is conjugate linear in the second variable. This convention will be helpful below.

Write $E_M(X) = E_M$ for those elements in some space $E(X)$ that are supported in $M$. With respect to the distributional pairing, one has for any $s \in \mR$ the duality statements (see e.g.\ \cite[Section 2]{AssylbekovStefanov}) 
\[
(H^s(M))^* = H^{-s}_M, \qquad (H^s_M)^* = H^{-s}(M), \qquad (C^{\infty}(M))^* = \mDp_M.
\]

%Note that the spaces $C^{\infty}_M$, $H^s_M$ and $\mDp_M$ do not depend the choice of the larger manifold $X$.
We also write $( u, \varphi )$ for the pairings for spaces on $M$. Any differential operator $P$ on $M$ induces a map $P: \mDp_M \to \mDp_M$ via 
\[
( Pv, \varphi ) = ( v, P^* \varphi ),
\]
where $P^*$ is the formal $L^2$ adjoint of $P$. Finally, given a set $S \subset C^{\infty}_M$, we consider its annihilator  
\[
S^{\perp} = \{ v \in \mDp_M \,;\, ( v, \varphi ) = 0 \text{ for all $\varphi \in S$} \}.
\]

\begin{Proposition} \label{prop_real_principal_type_solvability}
Let $P$ be a real principal type differential operator of order $m \geq 1$ on a compact manifold $M$ with smooth boundary, and let $s \in \mR$.
\begin{enumerate}
\item[(a)]
If $u \in \mDp_M$ solves $Pu = f$ in $X$ where $f \in H^s_M$, then $u \in H^{s+m-1}_M$.
\item[(b)]
The set 
\[
N(P) = \{ v \in \mDp_M \,;\, Pv = 0 \}
\]
is a finite dimensional space contained in $C^{\infty}_M$.
\item[(c)]
Define the space 
\[
Y_s = \{v \in \mDp_M \,;\, P^* v \in H^s_M \}.
\]
Then $Y_s \subset H^{s+m-1}_M$, and for some $C_s > 0$ one has the estimates 
\begin{equation} \label{v_htm_equation_first}
\norm{v}_{H^{s+m-1}_M} \leq C_s ( \norm{P^* v}_{H^s_M} + \norm{v}_{H^{s+m-2}_M} ), \qquad v \in Y_s,
\end{equation}
and 
\begin{equation} \label{v_htm_equation}
\norm{v}_{H^{s+m-1}_M} \leq C_s \norm{P^* v}_{H^s_M}, \qquad v \in Y_s \cap N(P^*)^{\perp}.
\end{equation}
\item[(d)]
Given any $f \in H^s(M) \cap N(P^*)^{\perp}$, there is a solution $u \in H^{s+m-1}(M)$ of the equation 
\[
Pu = f \text{ in $M$}.
\]
Moreover, there is a bounded linear operator $E_s: H^s(M) \cap N(P^*)^{\perp} \to H^{s+m-1}(M)$ such that $P E_s f = f$.
\end{enumerate}
\end{Proposition}

%\tbl{[MS: In order to correct approximate solutions to exact solutions, we need norm estimates in the above solvability result. I got $H^{s+m-1}$ solutions with a norm estimate, but for $C^{\infty}$ Cauchy data sets we would need $C^{\infty}$ solutions with norm estimates. I could not see how to prove this. Some surjective differential operators need not have a right inverse that is continuous on $C^{\infty}$, so this could be nontrivial.]}

\begin{proof}
Parts (a) and (b) follow from \cite[Theorems 26.1.4 and 26.1.7]{Hormander}, after extending $P$ as a real principal type operator to a slightly larger open manifold $X$. To prove part (c), note that by part (a) one has $Y_s \subset H^{s+m-1}_M$, and $Y_s$ with norm $\norm{v}_{Y_s} = \norm{v}_{H^{s+m-2}_M} + \norm{P^* v}_{H^s_M}$ becomes a Hilbert space. The inclusion map $\iota: Y \to H^{s+m-1}_M$ has closed graph, and the closed graph theorem yields \eqref{v_htm_equation_first}. This in turn implies \eqref{v_htm_equation}. Indeed, if \eqref{v_htm_equation} were false one would have a sequence $(v_j)$ in $Y_s \cap N(P^*)^{\perp}$ with $\norm{v_j}_{H^{s+m-1}_M} = 1$ and $\norm{P^* v_j}_{H^s_M} \to 0$, and by compact Sobolev embedding some subsequence would converge strongly in $H^{s+m-2}_M$ to some $v \in H^{s+m-2}_M \cap N(P^*)^{\perp}$ with $P^* v = 0$. Then \eqref{v_htm_equation_first} would give $\norm{v}_{H^{s+m-2}_M} \geq 1/C$, and we would obtain a contradiction since $v$ would be a nonzero element of both $N(P^*)$ and $N(P^*)^{\perp}$.

Finally, we prove part (d). Given $f \in H^s(M) \cap N(P^*)^{\perp}$, we define the (conjugate) linear functional 
\[
\ell_f: P^* C^{\infty}_M \to \mC, \ \ \ell_f(P^* v) = ( f, v ) \text{ for $v \in C^{\infty}_M$}.
\]
This is well defined by the assumption on $f$. One has 
\[
\abs{( f, v )} \leq \norm{f}_{H^s(M)} \norm{v}_{H^{-s}_M}\leq C \norm{f}_{H^s(M)} \norm{P^* v}_{H^{-(s+m-1)}_M}, \qquad v \in C^{\infty}_M,
\]
since this holds for $v \in C^{\infty}_M \cap N(P^*)^{\perp}$ by \eqref{v_htm_equation} (with $s$ replaced by $-s-m-1$) and both sides remain the same when a function in $N(P^*)$ is added to $v$. Thus, denoting by $W$ the closure of $P^* C^{\infty}_M$ in $H^{-(s+m-1)}_M$, $\ell_f$ extends uniquely to a bounded functional on $W$ such that 
\[
\abs{\ell_f(w)} \leq C \norm{f}_{H^s(M)} \norm{w}_{H^{-(s+m-1)}_M}, \qquad w \in W.
\]
Let $Q$ be the orthonormal projection from $H^{-(s+m-1)}_M$ to $W$, and define 
\[
\bar{\ell}_f: H^{-(s+m-1)}_M \to \mC, \ \ \bar{\ell}_f(u) = \ell_f(Qu).
\]
Then $\bar{\ell}_f$ is a bounded linear functional on $H^{-(s+m-1)}_M$, and it depends linearly on $f$ by construction. By duality there exists $u _f \in H^{s+m-1}(M)$ with $\norm{u_f}_{H^{s+m-1}(M)} \leq C \norm{f}_{H^s(M)}$ such that 
\[
( u_f, P^* v ) = ( f, v ), \qquad v \in C^{\infty}_M.
\]
In particular $P u_f = f$ in $M^{\mathrm{int}}$. Letting $E_s f = u_f$ concludes the proof.
\end{proof}

\addtocontents{toc}{\SkipTocEntry}
\subsection*{Semiclassical wave front sets}

Next, we will consider the action of real principal type operators on functions depending on a small parameter $h > 0$. We recall certain semiclassical notions following \cite{Zworski}. A family $u = u_h \subset L^2_{\mathrm{loc}}(X)$, $0 < h \leq 1$, is called \emph{$L^2$-tempered} if for any $\chi \in C^{\infty}_c(X)$ there is $N \geq 0$ so that 
\[
\norm{\chi u_h}_{L^2} = O(h^{-N})
\]
as $h \to 0$. We say that $u_h$ is \emph{semiclassically smooth} at $(x_0, \xi_0) \in T^* X$ if, in some local coordinates near $x_0$, there is $\varphi \in C^{\infty}_c$ with $\chi = 1$ near $x_0$ and $\psi \in C^{\infty}_c$ with $\psi = 1$ near $\xi_0$ so that 
\[
\psi \mF_h (\varphi u_h) = O_{L^2}(h^{\infty})
\]
where $\mF_h$ is the semiclassical Fourier transform, 
\[
\mF_h f(\xi) = (2\pi h)^{-n} \int_{\mR^n} e^{-i\frac{x \cdot \xi}{h}} f(x) \,dx.
\]
The definition is independent of the choice of local coordinates. The \emph{semiclassical wave front set} $\WF_{\text{scl}}(u_h)$ is the complement of those points $(x_0,\xi_0) \in T^* X$ at which $u_h$ is semiclassically smooth. Roughly, one can think of $\WF_{\text{scl}}(u_h)$ as the set where the family $u_h$ is localized in phase space as $h \to 0$.

We will need the following characterization of the semiclassical wave front set in terms of FBI (Fourier-Bros-Iagolnitzer) transforms. This is a local statement and reduces to the corresponding result in Euclidean space \cite[Theorem 13.14]{Zworski}.

\begin{Proposition} \label{prop_wavefront_fbi}
Let $u = u_h$ be $L^2$-tempered in $X$. One has $(x_0,\xi_0) \notin \WF_{\mathrm{scl}}(u)$ if and only if  
\[
\int_X u(y) \ol{e^{i\Psi(y;x,\xi)/h} b(y)} \,dV(y) = O(h^{\infty})
\]
uniformly over $(x,\xi)$ near $(x_0,\xi_0)$, where $\Psi = \Psi(\,\cdot\,;x,\xi)$ and $b = b(\,\cdot\,)$ are smooth functions in $X$ so that, in some local coordinates near $x_0$, 
\[
\Psi(y;x,\xi) = \xi \cdot (y-x) + \frac{i}{2} \abs{y-x}^2, \qquad \supp(b) \text{ is close to $x_0$}, \qquad b = 1 \text{ near $x_0$}.
\]
\end{Proposition}

If $u_h$ and $v_h$ are $L^2$-tempered, the semiclassical wave front set has the following simple properties:
\begin{itemize}
\item 
If $u_h$ and $v_h$ are semiclassically smooth at $(x_0,\xi_0)$, then so is $u_h + v_h$. Thus 
\[
\WF_{\text{scl}}(u_h) \triangle \WF_{\text{scl}}(v_h) \subset \WF_{\text{scl}}(u_h+v_h) \subset \WF_{\text{scl}}(u_h) \cup \WF_{\text{scl}}(v_h), 
\]
where $A \triangle B = (A \setminus B) \cup (B \setminus A)$.
\item 
If $\norm{\chi v_h}_{L^2} = O(h^{\infty})$ for $\chi \in C^{\infty}_c(X)$, then $\WF_{\text{scl}}(v_h) = \emptyset$ and $\WF_{\text{scl}}(u_h + v_h) = \WF_{\text{scl}}(u_h)$.
\end{itemize}

The following result is a semiclassical version of the propagation of singularities theorem.

\begin{Proposition}
Let $P$ be a real principal type differential operator in an open manifold $X$, and let $u_h$ be $L^2$-tempered. If 
\[
Pu_h = f_h \text{ in $X$}
\]
where $f_h$ is $L^2$-tempered, then $\WF_{\text{scl}}(u_h) \setminus \WF_{\text{scl}}(f_h)$ is invariant under the Hamiltonian flow of the principal symbol of $P$.
\end{Proposition}
\begin{proof}
Define $P_h := h^m P$, so that the semiclassical principal symbol of $P_h$ is precisely the same as the classical principal symbol of $P$. Then $P_h u_h = h^m f_h$, and the result follows from \cite[Theorem 12.5]{Zworski}.
\end{proof}

\addtocontents{toc}{\SkipTocEntry}
\subsection*{Cauchy data sets}

Next we will discuss Cauchy data sets, starting with three basic invariance properties.

\begin{Lemma} \label{lemma_cauchy_data_set_invariances}
Let $P$ be a differential operator of order $m$, and let $c \in C^{\infty}(M)$ be nonvanishing. Then 
\[
C_P = C_{cP}.
\]
If additionally $c|_{\p M} = 1$ and $\nabla^j c|_{\p M} = 0$ for $1 \leq j \leq m-1$, then 
\[
C_P = C_{P + c^{-1}[P,c]}.
\]
Moreover, if $\Phi: M \to M$ is a diffeomorphism such that  $\Phi$ and $\mathrm{Id}_M$ agree to order $m-1$ on $\p M$, then 
\[
C_P = C_{\Phi^* P}.
\]
Here $\Phi^* P$ is the differential operator on $M$ defined via $\Phi^* P(\Phi^* v) = \Phi^* (Pv)$ for $v \in C^{\infty}(M)$.
\end{Lemma}
\begin{proof}
The first part is clear, and the second part follows from the formula 
\[
P(cv) = cPv + [P,c]v.
\]
The third part holds since $u$ and $\Phi^* u$ agree to order $m-1$ on $\p M$ by the assumption on $\Phi$.
\end{proof}

It is instructive to verify how these invariances affect the scattering relation and bicharacteristic ray transforms. The next lemma shows that Theorem \ref{thm_main1} is consistent with the above invariances. It also shows that if $M$ is simply connected, then changing $P$ to $P + c^{-1}[P,c]$ changes the subprincipal symbol by a term $\p_{\xi_a} p_m \p_a \varphi$ for some $\varphi \in C^{\infty}(M)$ with $\varphi|_{\p M} = 0$. If $P$ is a second order operator and $A$ is a $1$-form encoding the first order coefficients, this corresponds to the standard invariance $A \to A + d\varphi$ obtained by conjugating $P$ with $e^{\pm i\varphi}$. See e.g.\ \cite{DKSaU, StefanovYang} for the Riemannian and Lorentzian cases. However, if $M$ has nontrivial topology, there is a more general invariance (corresponding to $A \to A + d\varphi + h$ in the second order case) described by harmonic $1$-forms $h$ whose integrals over closed loops are in $2\pi i \mZ$. This is related to the Aharonov-Bohm effect where in fact the Cauchy data sets may differ if the integrals of $h$ over closed loops are not in $2\pi i \mZ$, see \cite{AB} and the review \cite{Eskin_AB}. This fact also explains the nonuniqueness modulo $2\pi \mZ$ in Theorem \ref{thm_main1}.

\begin{Lemma} \label{lemma_aharonov_bohm_general}
All three invariances in Lemma \ref{lemma_cauchy_data_set_invariances} preserve the scattering relation $\alpha_P$ (for the third one one needs to assume $m \geq 2$). Moreover, if $M$ is connected and if $c \in C^{\infty}(M,\mC)$ satisfies $c|_{\p M} = 1$, then 
\begin{equation} \label{subprincipal_gauge_formula}
\sigma_{\subp}[(P+c^{-1}[P,c])^{\mu}] = \sigma_{\subp}(P^{\mu}) - i \sum_{a=1}^n \p_{\xi_a} p_m (\p_a \varphi + h_a)
\end{equation}
for some $\varphi \in C^{\infty}(M,\mC)$ with $\varphi|_{\p M} = 0$, and for some harmonic $1$-form $h = h_a \,dx^a$ on $M$ (with respect to the auxiliary Riemannian metric) whose tangential part vanishes on $\p M$ and which satisfies 
\[
\int_{\eta} h \in 2\pi i \mZ \qquad \text{for any closed loop $\eta$ in $M$}.
\]
The integral of $\sum_{a=1}^n \p_{\xi_a} p_m (\p_a \varphi + h_a)$ over any null bicharacteristic segment between boundary points is in $2\pi i \mZ$.
\end{Lemma}
\begin{proof}
Note that if $P$ has real principal symbol and $c \in C^{\infty}(M,\mR)$ is nonvanishing, then changing $P$ to $cP$ or $P + c^{-1}[P,c]$ does not change the null bicharacteristic flow. Replacing $P$ by $\Phi^* P$ changes the null bicharacteristics in the interior, but not at the boundary if $\Phi$ and $\mathrm{Id}_M$ agree to first order on $\p M$. Thus the three invariances above preserve the scattering relation at least when $m \geq 2$.

If $P$ has full symbol with polyhomogeneous expansion $\sum_{j=0}^m p_j(x,\xi)$ in some local coordinate system, then the order $m-1$ term in the corresponding expansion for $P + c^{-1}[P,c]$ is 
\[
p_{m-1}(x,\xi) -i c^{-1} \sum_{a=1}^n \p_{\xi_a} p_m(x,\xi) \p_a c(x).
\]
Thus \eqref{subprincipal_gauge_formula} follows if we can prove that 
\[
c^{-1} dc = d\varphi + h
\]
where $\varphi$ and $h$ are as stated. This is just the Hodge decomposition for the $1$-form $c^{-1} dc$, which can be obtained by using the auxiliary Riemannian metric $g$ on $M$ and solving 
\[
-\Delta_{g} \varphi = \delta_{g}(c^{-1} dc) \text{ in $M$}, \qquad \varphi|_{\p M} = 0.
\]
It follows that $h := c^{-1} dc - d\varphi$ is harmonic (i.e.\ $dh = \delta_{g} h = 0$) with vanishing tangential part on $\p M$. If $\eta$ is a closed loop in $M$, we have 
\[
\int_{\eta} h = \int_{\eta} c^{-1} dc = \int_{c \circ \eta} \frac{dz}{z}.
\]
The last quantity is $2\pi i$ times the winding number of the curve $c \circ \eta$ in $\mC \setminus \{ 0 \}$, and hence belongs to $2\pi i \mZ$.

Finally, using that $M$ is connected, we may fix $x_0 \in \p M$ and define the function 
\[
\psi: M \to \mC/2\pi i \mZ, \ \ \psi(x) = \int_{\eta_{x_0,x}} h \quad \text{modulo $2\pi i \mZ$}
\]
where $\eta_{x_0,x}$ is any smooth curve from $x_0$ to $x$ in $M$. This is well defined by the condition on $h$, and $e^{\pm \psi}$ are well defined smooth functions $M \to \mC \setminus \{0\}$. Moreover, 
\[
d(c e^{-\varphi} e^{-\psi}) = (c^{-1} dc - d\varphi - h)c e^{-\varphi} e^{-\psi} = 0.
\]
Thus $c e^{-\varphi} e^{-\psi}$ is a constant. Evaluating at $x_0$ gives that $c = e^{\varphi} e^{\psi}$ and that $e^{\psi}|_{\p M} = c e^{-\varphi}|_{\p M} = 1$, i.e.\ $\psi(y) = 0$ modulo $2\pi i \mZ$ whenever $y \in \p M$. Now, if $\gamma: [0,T] \to M$, $\gamma(t) = (x(t),\xi(t))$ is a null bicharacteristic segment between boundary points, the fact that $\p_{\xi_a} p_m(x(t), \xi(t)) = \dot{x}_a(t)$ yields 
\[
\int_0^{T} (\p_{\xi_a} p_m \p_a \varphi)(\gamma(t)) \,dt = \int_0^{T} \frac{d}{dt} (\varphi(x(t))) \,dt = 0
\]
since $\varphi|_{\p M} = 0$. For the part involving $h$ we have 
\[
\int_0^{T} (\p_{\xi_a} p_m h_a)(\gamma(t)) \,dt = \int_0^{T} h(\dot{x}(t)) \,dt.
\]
Now we may write $\psi(x(t)) = \int_{\eta_{x_0,x(0)}} h + \int_0^t h(\dot{x}(s)) \,ds$ modulo $2\pi i \mZ$, leading to 
\[
\int_0^{T} (\p_{\xi_a} p_m h_a)(\gamma(t)) \,dt =  \psi(x(T)) - \psi(x(0)) = 0 \text{ modulo } 2\pi i \mZ
\]
since $\psi(y) = 0$ modulo $2\pi i \mZ$ for $y \in \p M$.
\end{proof}

Note also that the gauge invariances in Lemma \ref{lemma_cauchy_data_set_invariances} are formulated in a different way compared to the usual invariances in the Calder\'on or Gel'fand problems for second order operators. For instance, if $\bar{g}$ is a Riemannian metric, $A$ is a $1$-form and $q$ is a function on $M$, the corresponding second order operator (written in local coordinates) in the Calder\'on problem is 
\[
Pu = \abs{\bar{g}}^{-1/2}(D_j + A_j)(\abs{\bar{g}}^{1/2} \bar{g}^{jk} (D_k + A_k)u) + qu.
\]
In this case one typically defines a normal derivative with respect to $\bar{g}$ and $A$, leading to the Cauchy data set 
\[
\tilde{C}_P = \{(u|_{\p M}, (du + iAu)(\nu_{\bar{g}})|_{\p M} \,;\, Pu = 0 \text{ in $M$} \},
\]
see \cite{DKSaU}. The Cauchy data set in this article is instead defined in terms of a (known) reference Riemannian metric $g$ on $M$ as 
\[
C_P = \{(u|_{\p M}, \nabla_g u|_{\p M}) \,;\, Pu = 0 \text{ in $M$} \}.
\]
Of course, if $\bar{g}$ and $A$ are known on $\p M$, then $C_P$ determines $\tilde{C}_P$ and vice versa. However, the difference between $C_P$ and $\tilde{C}_P$ implies for instance that there is no analogue of the gauge invariance $A \to A + d\psi$ with $\psi|_{\p M} = 0$ in our boundary determination result for the subprincipal symbol.

Finally, we discuss a simple but fundamental integral identity. We first give a characterization of the inclusion $C_{P_1} \subset C_{P_2}$ in terms of an estimate for a certain inner product.

\begin{Lemma} \label{lemma_integral_identity}
Let $P_1$ and $P_2$ be differential operators of order $m$. One has $C_{P_1} \subset C_{P_2}$ if and only if for any $u_1 \in H^m(M)$ solving $P_1 u_1 = 0$ in $M$, there is $C_{u_1} > 0$ so that 
\begin{equation} \label{pone_ptwo_v}
\abs{((P_1-P_2)u_1, v)_{L^2(M)}} \leq C_{u_1} \norm{P_2^* v}_{H^{-m}(M)}, \qquad v \in H^m(M).
\end{equation}
\end{Lemma}
\begin{proof}
Assume that $C_{P_1} \subset C_{P_2}$ and that $u_1 \in H^m(M)$ solves $P_1 u_1 = 0$ in $M$. Then we can find $\tilde{u}_2 \in H^m(M)$ with $P_2 \tilde{u}_2 = 0$ in $M$ so that $u_1-\tilde{u}_2 \in H^m_0(M)$. For any $v \in H^m(M)$, one has in terms of $L^2(M)$ inner products 
\begin{align*}
( (P_1-P_2) u_1, v) = - (P_2 u_1, v) = - (P_2 (u_1 - \tilde{u}_2), v) = - (u_1 - \tilde{u}_2, P_2^* v).
\end{align*}
The inequality \eqref{pone_ptwo_v} follows from the duality of $H^m_0(M)$ and $H^{-m}(M)$.

Conversely, let $u_1 \in H^m(M)$ solve $P_1 u_1 =  0$, and assume that \eqref{pone_ptwo_v} holds. Then $f := -P_2 u_1$ satisfies 
\begin{equation} \label{f_v_inner_product_estimate}
\abs{(f,v)} \leq C \norm{P_2^* v}_{H^{-m}(M)}, \qquad v \in H^m(M).
\end{equation}
Define a linear functional 
\[
\ell_f: P_2^* H^m(M) \to \mC, \ \ \ell_f(P_2^* v) = (f,v).
\]
This is well defined and bounded with respect to the $H^{-m}(M)$ norm by \eqref{f_v_inner_product_estimate}. By the Hahn-Banach theorem $\ell_f$ extends continuously to $H^{-m}(M)$, and by duality there is $w \in H^m_0(M)$ so that 
\[
(f,v) = \ell_f(P_2^* v) = (w, P_2^* v), \qquad v \in H^m(M).
\]
It follows that $P_2 w = f = -P_2 u_1$. Thus $\tilde{u}_2 := u_1 + w$ solves $P_2 \tilde{u}_2 = 0$ with $u_1-\tilde{u}_2 \in H^m_0(M)$. This proves that $C_{P_1} \subset C_{P_2}$.
\end{proof}

%\tbl{[MS: in the Calder\'on or Gel'fand problems the condition $C_{P_1} = C_{P_2}$ is equivalent to the integral identity $((P_1-P_2)u_1, u_2) = 0$. I did not see how to prove the equivalence in the general case. Note that by Lemma \ref{lemma_integral_identity} it is enough to construct exact solutions for $P_1$ and quasimodes for $P_2$.]}

If we choose $v$ in \eqref{pone_ptwo_v} to be a solution of $P_2^* v = 0$, we immediately obtain an integral identity from the condition $C_{P_1} = C_{P_2}$. For our applications we indeed need $P_1-P_2$ to appear on the left hand side, and thus one of the functions needs to solve the adjoint equation.

\begin{Lemma} \label{lemma_integral_identity_potential}
If $C_{P_1} = C_{P_2}$, then 
\[
( (P_1 - P_2) u_1, u_2)_{L^2(M)} = 0
\]
whenever $u_j \in H^m(M)$ solve $P_1 u_1 = P_2^* u_2 = 0$ in $M$. \\[5pt]
%(b) If $C_{P+V_1} = C_{P+V_2}$ where $V_j \in C^{\infty}(M)$, then 
%\[
%( (V_1 - V_2) u_1, u_2)_{L^2(M)} = 0
%\]
%whenever $u_j \in H^m(M)$ solve $(P+V_1) u_1 = (P^*+V_2^*) u_2 = 0$ in $M$.
\end{Lemma}

\section{Quasimode construction} \label{sec_quasimode_construction}

We will now proceed to the construction of approximate solutions that concentrate near an injective segment of a null bicharacteristic curve. This can be done for any operator with real principal symbol.

\begin{Theorem} \label{thm_quasimode_direct}
Let $P$ be a differential operator of order $m$ on a smooth manifold $X$, with real valued principal symbol $p_m$. Let $\gamma: [0,T] \to T^* X \setminus 0$ be a segment of a null bicharacteristic curve. If $\gamma$ is injective on $[0,T]$, there is a family $u = u_h \in C^{\infty}_c(X)$ for $0 < h \leq 1$ such that 
\[
\WF_{\mathrm{scl}}(u) = \gamma([0,T]), \qquad \WF_{\mathrm{scl}}(Pu) = \gamma(0) \cup \gamma(T).
\]
Moreover, if $\tilde{P}$ is another differential operator on $X$ with the same principal symbol as $P$ and if $v = v_h$ is the analogous family for $\tilde{P}^*$, then for any differential operator $Q$ of order $\ell \geq 0$ on $X$ whose coefficients vanish near the end points of $\gamma$ one has 
\begin{equation} \label{concentration}
\lim_{h \to 0} h^{-\frac{n+1}{2}+\ell} \int_X (Q u) \bar{v} \,dV = c_0 \int_0^T q_{\ell}(\gamma(t)) \exp \left[-i \int_0^t (p_{m-1} - \tilde{p}_{m-1})(\gamma(s)) \,ds \right] \,dt
\end{equation}
where $c_0 \neq 0$ is the constant in \eqref{czero_constant_formula}, $q_{\ell}$ is the principal symbol of $Q$, and $\sum_{j=0}^m p_j$
and $\sum_{j=0}^m \tilde p_j$ are the polyhomogeneous full symbols of $P$ and $\tilde P$ in a local coordinate system. 
\end{Theorem}

Note that the statement about $\WF_{\mathrm{scl}}(Pu)$ implies that $Pu = O(h^{\infty})$ away from the end points of $\gamma$, and hence $u$ is an approximate solution for $P$. The formula \eqref{concentration} is related to the semiclassical limit measure for the family $u$. In fact, one could also prove that for any $a \in C^{\infty}_c(T^* X)$ vanishing near the end points of $\gamma$ one has 
\[
\lim_{h \to 0} h^{-\frac{n+1}{2}} (\mathrm{Op}_h(a) u_h, u_h)_{L^2(X)} = \int_0^T r_{\gamma}(t) a(\gamma(t)) \,dt
\]
where $\mathrm{Op}_h(a)$ is the semiclassical Weyl quantization of $a$, and for any nonvanishing half density $\mu$ 
\[
r_{\gamma}(t) = c_0 \exp \left[-i \int_0^t (\sigma_\subp[P^\mu] - \sigma_\subp[(P^*)^\mu])(\gamma(s)) \,ds \right].
\]
The formulation \eqref{concentration} will be convenient for our applications, and it does not involve semiclassical quantization. The expression on the right hand side of \eqref{concentration} must be coordinate invariant since the left hand side is, however, this can also be seen directly as follows: 

\begin{Lemma}\label{lem_coordinv_subdiff}
Let $P$ and $\tilde P$ be two differential operators of order $m$ on $X$ with the same principal symbol, and denote by $\sum_{j=0}^m p_j$
and $\sum_{j=0}^m \tilde p_j$ their polyhomogeneous full symbols in a local coordinate system. Then $p_{m-1} - \tilde p_{m-1}$ is an invariantly defined function on $T^* X \setminus 0$. 
\end{Lemma}
\begin{proof}
Let $\mu$ be a half density on $X$ and write $\sum_{j=0}^m q_j$
and $\sum_{j=0}^m \tilde q_j$ for the polyhomogeneous full symbols of $P^\mu$ and $\tilde P^\mu$ in a local coordinate system.
Then $q_m = p_m$ and
    \begin{align*}
q_{m-1} = p_{m-1} - i \mu \p_{\xi_j} p_m\, \p_{x_j} \mu^{-1}.
    \end{align*}
Analogous statements hold for $\tilde P^\mu$, 
and recalling the definition of subprincipal symbol in
\eqref{def_subprin_symb}, we see that the function
    \begin{align}\label{subprin_diff}
\sigma_\subp[P^\mu] - \sigma_\subp[\tilde P^\mu] = p_{m-1} - \tilde p_{m-1}
    \end{align} 
is defined invariantly on $T^* X \setminus 0$. 
\end{proof}

In the following, we will assume that $P$ has order $m \geq 1$ and its principal symbol is denoted by $p_m \in C^{\infty}(T^* X)$. We will write in local coordinates 
\[
\gamma(t) = (x(t), \xi(t))
\]
so that $(x(t), \xi(t))$ satisfies the Hamilton equations 
\[
\dot{x}(t) = d_{\xi} p_m(x(t),\xi(t)), \qquad \dot{\xi}(t) = -d_x p_m(x(t), \xi(t)).
\]
More invariantly, we define $x(t) = \pi(\gamma(t))$ where $\pi: T^*X \to X$ is the natural projection.

We will use the following standard result related to conjugating a differential operator by exponentials. This is the "fundamental asymptotic expansion lemma" stated in \cite[Section VIII.7]{Taylor1981} or \cite[Section VI.3]{Treves2} for pseudodifferential operators and real valued phase functions, and in \cite[Section X.4]{Treves2} for complex valued phase functions. In our case of differential operators, the proof is just an elementary computation in local coordinates.

\begin{Lemma} \label{lemma_conjugated_p}
Let $\Phi$ be a smooth real or complex valued function. Then 
\[
e^{-i \Phi/h} P(e^{i \Phi/h} u) = \sum_{j=0}^m h^{j-m} R_j u, \qquad h > 0,
\]
where each $R_j$ is a differential operator of order $j$. In particular, $R_1 u = \frac{1}{i} Lu + bu$ where $L$ is a (possibly complex) vector field and $b = b_P$ is a function. In local coordinates one has 
\begin{align*}
R_0 u &= p_m(x, \nabla \Phi(x)) u, \\
L u &= \p_{\xi_j} p_m(x, \nabla \Phi(x)) \p_j u, \\
b &= \frac{1}{2i} \p_{\xi_j \xi_k} p_m(x, \nabla \Phi(x)) \p_{jk} \Phi(x) + p_{m-1}(x, \nabla \Phi(x)), 
\end{align*}
where $p_{m-1}$ is the order $m-1$ term in the polyhomogeneous expansion of the full symbol of $P$ in these coordinates.
\end{Lemma}

There are several methods for constructing quasimodes concentrating near an injective null bicharacteristic segment $\gamma: [0,T] \to T^*X \setminus 0$ where $\gamma(t) = (x(t), \xi(t))$. We first give a brief discussion of such methods, to motivate the construction given in this paper.

\vspace{5pt}

\noindent {\bf Local case.} To construct a quasimode $u$ locally near $x(t_0)$ with $\dot{x}(t_0) \neq 0$ it is enough to use the geometrical optics ansatz $u = e^{i \varphi/h} a$ where $\varphi$ is a smooth real valued phase function solving the eikonal equation $p_m(x, d \varphi(x)) = 0$. If $\dot{x}(t_0) \neq 0$ this equation can always be solved locally near $x(t_0)$, but global solutions do not exist in general if there are caustic points (i.e.\ points where the projection $\Lambda \to X$ fails to have bijective differential where $\Lambda \subset T^*X \setminus 0$ is a Lagrangian manifold associated with the eikonal equation, see \cite[Section 6.4]{Hormander}).

\vspace{5pt}

\noindent {\bf Global case without cusps.} More generally, if $\gamma$ is injective on $[0,T]$ and has no cusps in the sense that $\dot{x}(t) \neq 0$ for $t \in [0,T]$, one can use a Gaussian beam type construction. If $x(t)$ is also injective on $[0,T]$, this amounts to looking for a quasimode $u = e^{i \Phi/h} a$ where $\Phi$ and $a$ are complex valued (more generally if  $x(t)$ is not injective the quasimode will be a finite sum of such functions). The phase function $\Phi$ is required to solve the eikonal equation only on the bicharacteristic in the sense that 
\begin{equation} \label{eikonal_equation_second}
p_m(x, d\Phi(x)) = 0 \text{ to infinite order on $x([0,T])$}.
\end{equation}
The construction of $\Phi$ and $a$ can be carried out very explicitly in global Fermi type coordinates near $x([0,T])$. This boils down to solving one nonlinear ODE (matrix Riccati equation) for the Hessian of $\Phi$ on $x([0,T])$, and linear ODEs for the derivatives of $\Phi$ and $a$ along $x([0,T])$. 

These Gaussian beam type constructions are very classical and go back at least to \cite{BabichLazutkin, Hormander1971} with further treatments e.g.\ in \cite{Ralston, Ralston_note, KKL}. A version of this construction is also given in \cite[Section 24.2]{Hormander}, where the existence of the required complex phase function is explained in terms of properties of complex Lagrangian planes in the complexification of $T^* X$. However, the construction always breaks down when the bicharacteristic has a cusp. In fact, if \eqref{eikonal_equation_second} holds near $x(t_0)$ where $\dot{x}(t_0) = 0$, then looking at first order derivatives in \eqref{eikonal_equation_second} and using the Hamilton equations implies that 
\[
0 = -\dot{\xi}(t_0) + \Phi''(x(t_0)) \dot{x}(t_0) = -\dot{\xi}(t_0).
\]
Hence one would have $dp_m(\gamma(t_0)) = 0$ and $\gamma(t) \equiv \gamma(t_0)$, which contradicts the assumption that $\gamma$ is injective on $[0,T]$.

We mention also the related construction of the propagator for hyperbolic equations as a global oscillatory integral with a complex-valued phase function \cite{LSV, SV, CLV}.

\vspace{5pt}

\noindent {\bf General case.} Let us now assume that $\gamma$ is injective on $[0,T]$ but $\dot{x}(t)$ may vanish. In the classical case (i.e.\ without a parameter), a construction of a quasimode $u$ associated with $\gamma$ such that $Pu \in C^{\infty}$ was given in \cite{DH72} and \cite[Section 26.1]{Hormander}. The argument proceeds by constructing a canonical transformation $\chi$ near $\gamma([0,T])$ that straightens $\gamma$ in phase space into the curve $\eta(x_1) = ((x_1,0), e_n)$ in $T^* \mR^n$, and by quantizing $\chi$ using suitable amplitudes to obtain Fourier integral operators $A$ and $B$ such that $B P A$ roughly corresponds to $D_{x_1}$ microlocally near $\gamma([0,T])$. One then constructs an explicit quasimode $U$ for $D_{x_1}$ in $\mR^n$ associated with $\eta$, and $u = AU$ will be the required quasimode for $P$. This phase space construction is not affected by the presence of cusps.

A semiclassical version of the above construction could be used to prove Theorem \ref{thm_quasimode_direct}. However, we will give a direct proof based on a modification of the Gaussian beam construction. This construction is motivated by the fact that a standard semiclassical quasimode $U$ for $D_{x_1}$ associated with $\eta$, 
\[
U(x_1,x') = e^{ix_n/h - \abs{x'}^2/h}
\]
can be thought of as a superposition of Gaussian wave packets along $\eta$. If $A$ is a semiclassical Fourier integral operator quantizing $\chi$, then $u = AU$ would be a superposition of Gaussian wave packets along the curve $x(t)$. We will thus look for the quasimode $u$ directly in the form 
\[
u = \int_0^T e^{i \Phi(x,t)/h} a(x,t) \,dt
\]
where $\Phi$ and $a$ are smooth complex valued functions in $M \times [0,T]$, and each $e^{i \Phi(\,\cdot\,,t)/h} a(\,\cdot\,,t)$ is a Gaussian wave packet at $x(t)$ oscillating in direction $\xi(t)$. The same idea appears in \cite{PaulUribe, KarasevVorobjev}. The phase function $\Phi$ will be chosen to satisfy 
\begin{equation} \label{eikonal_equation_third}
p_m(\,\cdot\,, d_x \Phi(\,\cdot\,,t)) + \p_t \Phi(\,\cdot\,,t) = 0 \text{ to infinite order at $x(t)$ for $t \in [0,T]$}.
\end{equation}
This generalizes \eqref{eikonal_equation_second} to the case where $\Phi$ may depend on $t$. The construction boils down to solving the same matrix Riccati equation as in the usual Gaussian beam construction, but it is not affected by the presence of cusps. By \eqref{eikonal_equation_third} we have that $p_m(x,d_x \Phi(x,t))$ is small near $x([0,T])$ modulo terms of the form $\p_t \Phi(x,t)$, but such terms can be dealt with using integration by parts in the formula for $u$.

After this motivating discussion, we will give a proof of the theorem. This will be done in detail since we will need to use the precise form of the quasimodes, including formulas within the proof, in the computations required for studying the inverse problems.

\begin{proof}[Proof of Theorem \ref{thm_quasimode_direct}]
The proof will be carried out in several steps. \\

\noindent {\it Step 1.} Setup. \\

We look for an approximate solution of $Pu = 0$ in the form 
\begin{equation} \label{quasimode_integral_ansatz}
u(x) = \int_0^T e^{i\Phi(x,t)/h} a(x,t) \,dt
\end{equation}
where $\Phi$ and $a$ are smooth complex valued functions in $X \times [0,T]$. By Lemma \ref{lemma_conjugated_p}, we have 
\begin{multline} \label{pu_conjugated_formula}
Pu(x) = \int_0^T e^{i\Phi(x,t)/h} \bigg[ h^{-m} p_m(x,d_x \Phi(x,t)) a(x,t) \\ 
+ h^{1-m}(\frac{1}{i} La + ba)(x,t) + \sum_{j=2}^m h^{j-m} R_j a(x,t) \bigg] \,dt.
\end{multline}
We wish to find $\Phi(x,t)$ so that 
\begin{equation} \label{phi_req1}
\Phi(x(t),t) = 0, \qquad d_x\Phi(x(t),t) = \xi(t),
\end{equation}
and so that 
\begin{equation} \label{eikonal_equation_fourth}
p_m(\,\cdot\,, d_x \Phi(\,\cdot\,,t)) + \p_t \Phi(\,\cdot\,,t) = 0 \text{ to infinite order at $x(t)$}.
\end{equation}
Moreover, we wish to find $a(x,t) \sim \sum_{j=0}^{\infty} h^j a_j(x,t)$ so that each $a_j$ is smooth and independent of $h$, and one has the transport equations 
\begin{align}
\frac{1}{i} (\p_t  + L) a_0 + ba_0 &= 0 \text{ to infinite order at $x(t)$}, \label{transport_eq1_timedependent} \\
\frac{1}{i} (\p_t + L) a_{1} + ba_{1} &= -R_2 a_0 \text{ to infinite order at $x(t)$}, \label{transport_eq2_timedependent} \\
\frac{1}{i} (\p_t + L ) a_{2} + ba_{2} &= -R_3 a_0 - R_2 a_{1} \text{ to infinite order at $x(t)$}, \label{transport_eq3_timedependent} \\
 &\vdots \notag
\end{align}

\noindent {\it Step 2.} Local construction of $\Phi$. \\

Fix $t_0 \in [0,T]$. We first construct $\Phi$ in a small neighborhood of $(x(t_0), t_0)$. Computing in local coordinates, we have (writing $\p_{x_{j}} p_m = \p_{x_{j}} p_m(x,\nabla_x \Phi(x,t))$ etc):
\begin{align}
\p_{x_{j}} (p_m(x,\nabla_x \Phi)) &= \p_{x_{j}} p_m + \p_{\xi_a} p_m \p_{x_a x_{j}} \Phi, \label{phi_req2} \\
\p_{x_{j} x_{k}} (p_m(x,\nabla_x \Phi)) &= \p_{x_{j} x_{k}} p_m + \p_{x_{j} \xi_a} p_m \p_{x_a x_{k}} \Phi + \p_{x_{k} \xi_a} p_m \p_{x_a x_{j}} \Phi \label{phi_req3} \\
 & \qquad + \p_{\xi_a \xi_b} p_m \p_{x_a x_{j}} \Phi \p_{x_b x_{k}} \Phi + \p_{\xi_a} p_m \p_{x_a x_{j} x_{k}} \Phi, \notag \\
\p_x^{\gamma}(p_m(x,\nabla_x \Phi)) &= \p_{\xi_a} p_m \p_{x_a} \p_x^{\gamma} \Phi + F^{\gamma}(\nabla^{\abs{\gamma}-1} p_m, \nabla^{\abs{\gamma}-1}_x \Phi) \nabla_x^{\abs{\gamma}} \Phi \label{phi_req4} \\
&\qquad + G^{\gamma}(\nabla^{\abs{\gamma}} p_m, \nabla^{\abs{\gamma}-1}_x \Phi), \qquad \abs{\gamma} \geq 3. \notag
\end{align}
In the last statement (which is easily proved by induction) $\nabla^k f$ collects all derivatives of $f$ up to order $k$, and $F^{\gamma}$ and $G^{\gamma}$ are polynomials in their arguments.

Motivated by \eqref{phi_req1}, we look for $\Phi$ locally having the form 
\[
\Phi(x,t) = \xi(t) \cdot (x-x(t)) + \frac{1}{2} H(t)(x-x(t)) \cdot (x-x(t)) + \Phi_3(x,t)
\]
where $H(t)$ is some complex symmetric matrix depending smoothly on $t$, and $\Phi_3(\,\cdot\,,t)$ vanishes to third order at $x(t)$. Then 
\begin{equation} \label{phi_first_formulas}
\Phi(x(t),t) = 0, \qquad \nabla_x \Phi(x(t),t) = \xi(t), \qquad \p_{x_j x_k} \Phi(x(t),t) = H_{jk}(t).
\end{equation}
One has $\xi(t) \cdot \dot{x}(t) = 0$ by the Hamilton equations and homogeneity. Thus it also follows that 
\begin{equation} \label{pt_phi_formula}
\p_t \Phi(x,t) = (\dot{\xi}(t) - H(t) \dot{x}(t)) \cdot (x-x(t)) + G(x,t)(x-x(t)) \cdot (x-x(t))
\end{equation}
for some matrix valued function $G(x,t)$.

We first show that \eqref{eikonal_equation_fourth} holds to first order. By \eqref{phi_first_formulas} and \eqref{pt_phi_formula}, one has 
\[
p_m(x,\nabla_x \Phi) + \p_t \Phi(x,t)\big|_{x=x(t)} = 0.
\]
We next use \eqref{phi_req2}, the Hamilton equations and \eqref{phi_first_formulas} to obtain 
\[
\p_{x_j}( p_m(x,\nabla_x \Phi) )\big|_{x=x(t)} = -\dot{\xi}_j(t) + H_{aj}(t) \dot{x}_a(t)
\]
By \eqref{pt_phi_formula} we also have 
\[
\p_{x_j} (\p_t \Phi(x,t) )\big|_{x=x(t)} = \dot{\xi}_j(t) - H_{ja}(t) \dot{x}_a(t)
\]
and therefore \eqref{eikonal_equation_fourth} holds to first order.
%\[
%\p_{x_j}( p_m(x,\nabla_x \Phi) + \p_t \Phi(x,t) )\big|_{x=x(t)} = 0.
%\]

To show that \eqref{eikonal_equation_fourth} holds to second order, we rewrite \eqref{phi_req3} as 
\[
\p_{x_{j} x_{k}} (p_m(x,\nabla_x \Phi)) \big|_{x=x(t)} = (D + B H + H B^t + H C H)_{jk}(t) + \dot{x}_a(t) \p_{x_a x_j x_k} \Phi(x(t),t)
\]
where $B(t)$, $C(t)$ and $D(t)$ are the matrices 
\begin{equation} \label{b_c_d_definition}
D_{jk}(t) := \p_{x_j x_k} p_m(\gamma(t)), \qquad B_{ja}(t) = \p_{x_j \xi_a} p_m(\gamma(t)), \qquad C_{ab}(t) = \p_{\xi_a \xi_b} p_m(\gamma(t)).
\end{equation}
Moreover, we note that 
\[
\p_{x_{j} x_{k}}( \p_t \Phi(x,t) ) \big|_{x=x(t)} = \p_t (\p_{x_{j} x_{k}} \Phi(x(t),t)) - \dot{x}_a(t) \p_{x_a x_j x_k} \Phi(x(t),t).
\]
Adding these two identities and using \eqref{phi_first_formulas}, we have 
\begin{equation} \label{pjk_riccati_equation}
\p_{x_{j} x_{k}} (p_m(x,\nabla_x \Phi) + \p_t \Phi(x,t) ) \big|_{x=x(t)} = (\dot{H} + H C H + B H + H B^t + D)_{jk}(t).
\end{equation}
Given any complex symmetric matrix $H_0$ with $\mathrm{Im}(H_0)$ positive definite, the matrix Riccati equation 
\begin{equation} \label{matrix_riccati_equation_general}
\dot{H} + H C H + B H + H B^t + D = 0, \qquad H(t_0) = H_0,
\end{equation}
has a smooth complex symmetric matrix solution $H(t)$ where $\mathrm{Im}(H(t))$ is positive definite, see \cite[Lemma 2.56]{KKL}. The solution exists globally in the set where $B$, $C$ and $D$ are smooth, i.e.\ in the whole coordinate patch. Choosing such a solution ensures that \eqref{eikonal_equation_fourth} holds to second order.

Finally, let $r \geq 3$ and assume that we have prescribed $\p_x^{\beta} \Phi(x(t),t)$ for $\abs{\beta} \leq r-1$ so that \eqref{eikonal_equation_fourth} holds to order $r-1$. Let $\abs{\gamma} = r$. Evaluating \eqref{phi_req4} at $x(t)$ gives 
\[
\p_x^{\gamma}(p_m(x,\nabla_x \Phi))\big|_{x(t)} = \dot{x}_a(t) \p_{x_a} \p_x^{\gamma} \Phi + F^{\gamma}(\nabla^{\abs{\gamma}-1} p_m, \nabla^{\abs{\gamma}-1}_x \Phi) \nabla_x^{\abs{\gamma}} \Phi + G^{\gamma}(\nabla^{\abs{\gamma}} p_m, \nabla^{\abs{\gamma}-1}_x \Phi)\big|_{x(t)}.
\]
Moreover, one has 
\[
\p_x^{\gamma} \p_t \Phi(x(t), t) = \p_t (\p_x^{\gamma} \Phi(x(t),t)) - \dot{x}_a(t) \p_{x_a} \p_x^{\gamma} \Phi(x(t),t).
\]
Adding the above two equations gives 
\begin{align}
 &\p_x^{\gamma}(p_m(x,\nabla_x \Phi) + \p_t \Phi)\big|_{x(t)} \label{pxgamma_formula}\\
 &\qquad = \p_t (\p_x^{\gamma} \Phi(x(t),t)) + F^{\gamma}(\nabla^{\abs{\gamma}-1} p_m, \nabla^{\abs{\gamma}-1}_x \Phi) \nabla_x^{\abs{\gamma}} \Phi + G^{\gamma}(\nabla^{\abs{\gamma}} p_m, \nabla^{\abs{\gamma}-1}_x \Phi)\big|_{x(t)}. \notag
\end{align}
Thus \eqref{eikonal_equation_fourth} will hold to order $r$ provided that 
\[
\p_t (\p_x^{\gamma} \Phi(x(t),t)) + F^{\gamma}(\nabla^{\abs{\gamma}-1} p_m, \nabla^{\abs{\gamma}-1}_x \Phi) \nabla_x^{\abs{\gamma}} \Phi + G^{\gamma}(\nabla^{\abs{\gamma}} p_m, \nabla^{\abs{\gamma}-1}_x \Phi) = 0, \qquad \abs{\gamma} = r.
\]
These equations over all $\gamma$ with $\abs{\gamma}=r$ can be understood as a linear system of ODEs for the vector $(\p_x^{\gamma} \Phi(x(t),t))_{\abs{\gamma} = r}$, and given any initial data at $t_0$ this system has a smooth solution. Thus we have constructed the formal Taylor series of $\Phi$ at $(x(t),t)$, and applying Borel summation gives the required function $\Phi$ locally.

\vspace{15pt}

\noindent {\it Step 3.} Global construction of $\Phi$. \\

We wish to glue together the local constructions in Step 2. To do this, we cover $[0,T]$ by finitely many open intervals $I_1, \ldots, I_R$ so that $0 \in I_1$, $T \in I_R$, $I_j \cap I_{j+1} \neq \emptyset$, and each $x(I_j)$ is contained in some coordinate patch $(U_j,y_j)$. We first construct $\Phi_1$ near $x(I_1)$ as in Step 2 with $\mathrm{Im}(H(0))$ positive definite. Choosing some $t_1 \in I_1 \cap I_2$, we then construct $\Phi_2$ near $x(I_2)$ as in Step 2 so that the initial data for the ODEs for $\Phi_2$ at $(x(t_1),t_1)$ matches the Taylor series of $\Phi_1$ at $(x(t_1),t_1)$ when $\Phi_1$ is written in the $y_2$ coordinates.

We claim that one has 
\begin{equation}
\Phi_1 = \Phi_2 \text{ to infinite order at any $(x(t),t)$ where $t \in I_1 \cap I_2$.}
\end{equation}
This holds to first order since $\Phi_j(x(t),t) = 0$ and $d_x \Phi_j(x(t),t) = \xi(t)$. For the higher order case, it is enough to use the invariant statements 
\[
p_m(\,\cdot\,, d_x \Phi_j(\,\cdot\,,t)) + \p_t \Phi_j(\,\cdot\,,t)|_{x(t)} = 0 \text{ to high order for $t \in I_1 \cap I_2$ and $j=1,2$}
\]
and the formulas \eqref{pjk_riccati_equation} and \eqref{pxgamma_formula}, which imply that in some local coordinates $y$, both $\p_y^{\gamma} \Phi_1(x(t),t)$ and $\p_y^{\gamma} \Phi_2(x(t),t)$ satisfy the same first order ODEs and have the same initial data when $t=t_1$.

Repeating the above process finitely many times, we see that the formal Taylor series of $\Phi(\,\cdot\,,t)|_{x(t)}$ varies smoothly with $t$ for $t \in [0,T]$. We can obtain a global smooth function $\Phi$ with this Taylor series by a Borel summation scheme that is compatible across coordinate patches. To do this, we use the auxiliary Riemannian metric on $X$ with covariant derivative $\nabla$. Then the tensor fields 
\[
T_j(t) = \nabla_x^j \Phi(\,\cdot\,,t)|_{x(t)},
\]
defined in terms of the formal Taylor series of $\Phi$, are invariantly defined and depend smoothly on $t \in [0,T]$. The map $E(t,v) = (\exp_{x(t)}(v), t)$ for $t \in [0,T]$ and $v \in T_{x(t)} M$ is a local diffeomorphism near each $(t,0)$ and it is injective on $[0,T] \times \{0\}$. Hence it is a diffeomorphism onto a neighborhood of $\{ (x(t),t) \,;\, t \in [0,T] \}$ (see e.g.\ \cite[Lemma 7.3]{KenigSalo}). We may define a smooth function $\Phi$ as the Borel sum 
\[
\Phi(x,t) = \sum_{j=0}^{\infty} \frac{T_j(t)(V(x,t), \ldots, V(x,t))}{j!} \chi(\abs{V(x,t)}/\eps_j)
\]
where $V(x,t)$ is the projection of $E^{-1}(x,t)$ to the $v$ component, $\chi \in C^{\infty}_c(\mR)$ satisfies $\chi=1$ near $0$, and $\eps_j$ are suitable small numbers (see \cite[Theorem 1.2.6]{Hormander}). To check that $\Phi$ indeed has the right Taylor series, observe that for small $s$ 
\[
\Phi(\exp_{x(t)}(sv), t) = \sum_{j=0}^{\infty} \frac{T_j(t)(v, \ldots, v)}{j!} s^j \chi(s \abs{v}/\eps_j).
\]
If $\eta(s)$ is a geodesic, the formula $\p_s^j (f(\eta(s)) = \nabla_{\dot{\eta}(s)}( \nabla_{\dot{\eta}(s)} ( \cdots (\nabla_{\dot{\eta}(s)} f))) = \nabla^j f(\dot{\eta}(s), \ldots, \dot{\eta}(s))$ shows that $\Phi$ defined as above satisfies $\nabla_x^j \Phi(\,\cdot\,,t)|_{x(t)} = T_j(t)$. \\

\noindent {\it Step 4.} Construction of $a$. \\

Recall that we wish to find $a(x,t) \sim \sum_{j=0}^{\infty} h^j a_{j}(x,t)$ so that $a_{j}$ solve the transport equations \eqref{transport_eq1_timedependent}--\eqref{transport_eq3_timedependent}. Evaluating the first equation at $x(t)$ gives that 
\[
\frac{1}{i} (\p_t  + L) a_0 + ba_0 \big|_{x=x(t)} = \frac{1}{i} \p_t (a_0(x(t),t)) + b(x(t)) a_0(x(t),t).
\]
This will vanish provided that 
\begin{equation} \label{azero_exponential_value}
a_0(x(t),t) = \exp \left[ -i \int_0^t b(x(s)) \,ds \right]
\end{equation}
where, for the sake of definiteness, we have chosen the initial value 
\[
a(x(0),0) = 1.
\]
Taking first order derivatives in \eqref{transport_eq1_timedependent} gives 
\[
\p_{x_j} \left[ \frac{1}{i} (\p_t  + L) a_0 + ba_0 \right] \big|_{x=x(t)}  = \frac{1}{i} \p_t (\p_{x_j} a_0(x(t),t)) + q_j^k(t) \p_{x_k} a_0(x(t),t))+ r_j(t) a_0(x(t),t) 
\]
where $q_j^k, r_j$ are independent of $a_0$. Requiring this to vanish leads to a linear ODE system for $(\p_{x_j} a_0(x(t),t))_{j=1}^n$, which can be solved as above. Continuing this process gives the formal Taylor series of $a_0(\,\cdot\,,t)$ at $x(t)$ within a coordinate patch. As in Step 4, covering $x([0,T])$ by finitely many coordinate patches and solving the ODEs for $a_0$ in the next coordinate patch with initial data obtained from the previous coordinate patch, we see that the formal Taylor series of $a_0(\,\cdot\,,t)|_{x(t)}$ depends smoothly on $t \in [0,T]$. Applying Borel summation as in Step 4 yields a smooth function $a_0$ satisfying \eqref{transport_eq1_timedependent}. The amplitudes $a_{1}, a_{2}, \ldots$ are constructed analogously. The support of $a$ can be chosen to be contained in any fixed small neighborhood of $\{ (x(t),t) \,;\, t \in [0,T] \}$. \\

\noindent {\it Step 5.} Further properties of $\Phi$. \\

Here we collect some further properties of the phase function $\Phi(x,t)$ for later use. Again, in order to have coordinate invariant statements, we use the auxiliary Riemannian metric $g$ on $X$. We also equip $X \times [0,T]$ with the product metric $g \oplus e$, where $e$ is the Euclidean metric on $\mR$.

First we fix the choice for the initial data for $H(t)$ by assuming that the quasimode near $x(0)$ is constructed using Riemannian normal coordinates at $x(0)$, and when $H(t)$ is written in these coordinates one chooses in \eqref{matrix_riccati_equation_general} the initial value 
\[
H(0) = i \mathrm{Id}.
\]
Next recall from \eqref{phi_req1} that
\begin{equation} \label{phi_req1_second}
\Phi(x(t), t) = 0, \qquad d_x \Phi(x(t), t) = \xi(t), \qquad \p_{x_j x_k} \Phi(x(t), t) = H_{jk}(t).
\end{equation}
In particular, $\mathrm{Im}(\Phi(\,\cdot\,,t))$ vanishes to first order at $x(t)$. Consequently the Riemannian Hessian $\nabla_x^2 \,\mathrm{Im}(\Phi)|_{(x(t),t)} = (\p_{x_j x_k} \mathrm{Im}(\Phi) - \Gamma_{jk}^l \p_{x_l} \mathrm{Im}(\Phi)) \,dx^j \otimes dx^k$ is equal to $\mathrm{Im}(H_{jk}(t)) \,dx^j \otimes dx^k$. This shows that $\mathrm{Im}(H(t))$ is an invariantly defined $2$-tensor field along $x(t)$, and 
\[
\nabla_x^2 \,\mathrm{Im}(\Phi)|_{(x(t),t)} = \mathrm{Im}(H(t)).
\]
Now $\mathrm{Im}(H(0))$ is positive definite, and this property is preserved for solutions of the matrix Riccati equation. This shows that there is $c > 0$ with 
\[
\nabla_x^2 \,\mathrm{Im}(\Phi)|_{(x(t),t)}(v, v) \geq c \abs{v}^2, \qquad \text{uniformly over }t \in [0,T].
\]
Thus if $\supp(a)$ is chosen sufficiently small, we will have 
\begin{equation} \label{im_phi_positive}
\mathrm{Im}(\Phi(x,t)) \geq c d(x, x(t))^2, \qquad (x,t) \in \supp(a).
\end{equation}

We will next consider derivatives of $\Phi$ with respect to $t$. By \eqref{pt_phi_formula} one has 
\begin{equation} \label{dtphi_formula}
\p_t \Phi(x(t), t) = 0.
\end{equation}
Differentiating \eqref{pt_phi_formula} further, we have 
\begin{equation} \label{dxtphi_formula}
d_x \p_t \Phi(x(t),t) = (\dot{\xi}_j(t) - H_{jk}(t) \dot{x}^k(t)) \,dx^j
\end{equation}
which is an invariantly defined $1$-form along $x(t)$ since the left hand side is invariant. Similarly 
\begin{equation} \label{dttphi_formula}
\p_t^2 \Phi(x(t),t) = -(\dot{\xi}_j(t) - H_{jk}(t) \dot{x}^k(t)) \dot{x}^j(t).
\end{equation}

\vspace{5pt}

\noindent {\it Step 6.} Wave front set of $Pu$. \\

We now study the wave front set of $Pu$ using the formula \eqref{pu_conjugated_formula}. By \eqref{eikonal_equation_fourth} we have 
\[
p_m(x,d_x \Phi(x,t)) = -\p_t \Phi(x,t) + r(x,t)
\]
where $r(\,\cdot\,,t)$ vanishes to infinite order at $x(t)$. We observe that 
\begin{align*}
 &\int_0^T e^{i\Phi(x,t)/h} h^{-m} (-\p_t \Phi(x,t)) a(x,t)  \,dt = -h^{-m} \int_0^T \frac{h}{i} \p_t(e^{i\Phi(x,t)/h}) a(x,t)  \,dt \\
 &= h^{1-m} \int_0^T e^{i\Phi(x,t)/h} \frac{1}{i} \p_t a(x,t)  \,dt + \frac{1}{i} h^{1-m} (e^{i \Phi(x,0)/h} a(x,0) - e^{i \Phi(x,T)/h} a(x,T)).
\end{align*}
Thus, using the formula for $Pu$ in \eqref{pu_conjugated_formula}, we have 
\begin{align*}
 &Pu(x) = \frac{1}{i} h^{1-m} (e^{i \Phi(x,0)/h} a(x,0) - e^{i \Phi(x,T)/h} a(x,T)) \\
 &+ \int_0^T e^{i\Phi(x,t)/h} \left[ h^{-m} (r a)(x,t) + h^{1-m}(\frac{1}{i} (\p_t a + La )+ ba)(x,t) + \sum_{j=2}^m h^{j-m} R_j a(x,t) \right] \,dt.
\end{align*}
By \eqref{transport_eq1_timedependent}--\eqref{transport_eq3_timedependent} we further have 
\begin{align} \label{pux_wavefrontset_formula}
Pu(x) = \frac{1}{i} h^{1-m} (e^{i \Phi(x,0)/h} a(x,0) - e^{i \Phi(x,T)/h} a(x,T)) + \int_0^T e^{i\Phi(x,t)/h} \left[ \sum_{j=0}^m h^{j-m} r_j(x,t) \right] \,dt
\end{align}
where each $r_j$ is a smooth function with $\supp(r_j) \subset \supp(a)$ and $r_j(\,\cdot\,,t)$ vanishes to infinite order at $x(t)$. By \eqref{im_phi_positive}, for any $N$ there is $C_N$ so that 
\[
\Big\lvert \int_0^T e^{i\Phi(x,t)/h} \left[ \sum_{j=0}^m h^{j-m} r_j(x,t) \right] \,dt \Big\rvert \leq C_N \int_0^T e^{-c d(x,x(t))^2/h} \left[ \sum_{j=0}^m h^{j-m} d(x,x(t))^N \right] \chi(x,t) \,dt
\]
where $\chi$ is the characteristic function of $\supp(a)$. Thus the $L^2(X)$ norm of the last term in \eqref{pux_wavefrontset_formula} is $O(h^{\infty})$. Since the first two terms on the right of \eqref{pux_wavefrontset_formula} are Gaussian wave packets at $\gamma(0)$ and $\gamma(T)$ with nonvanishing amplitudes, it follows (see \cite[Section 8.4.2]{Zworski}) that 
\[
\WF_{\mathrm{scl}}(Pu) = \gamma(0) \cup \gamma(T).
\]

\vspace{5pt}

\noindent {\it Step 7.} Wave front set of $u$. \\

Let $(x_0,\xi_0) \in T^* X$. We will test whether $(x_0,\xi_0)$ is in $\WF_{\mathrm{scl}}(u)$ by integrating $u$ against the wave packet $e^{i\Psi/h} b$, where $\Psi = \Psi(\,\cdot\,;x_0,\xi_0)$ and $b$ are as in Proposition \ref{prop_wavefront_fbi}. Note in particular that 
\[
\Psi(x_0) = 0, \qquad d \Psi(x_0) = \xi_0, \qquad \nabla^2 \,\mathrm{Im}(\Psi)(x_0) > 0, \qquad b(x_0) \neq 0.
\]
Moreover, since $b$ is supported close to $x_0$, we have 
\[
d\Psi \neq 0 \text{ in $\supp(b) \setminus \{x_0\}$}, \qquad \mathrm{Im}(\Psi)(x) \geq c d(x,x_0)^2 \text{ in $\supp(b)$}.
\]
We thus need to study the integral 
\[
I(h) := \int_X u \ol{e^{i\Psi/h} b} \,dV = \int_0^T \int_X e^{i \Theta(x,t)/h} r(x,t) \,dV \,dt
\]
where 
\[
\Theta(x,t) := \Phi(x,t) - \ol{\Psi(x)}, \qquad r(x,t) := a(x,t) \ol{b(x)}.
\]
We will show that $I(h) = O(h^{\infty})$ if $(x_0, \xi_0) \notin \gamma([0,T])$ (see cases 7a and 7b below), and $I(h) \sim h^{\frac{n+1}{2}}$ if $(x_0,\xi_0) \in \gamma([0,T])$ (see case 7c below). Then Proposition \ref{prop_wavefront_fbi} proves that $\WF_{\mathrm{scl}}(u) = \gamma([0,T])$.

{\it Case 7a.} 
First assume that $x_0 \notin x([0,T])$. By the properties of $\Phi$ and $\Psi$ and by the triangle inequality, 
\begin{equation} \label{imtheta_wfu_estimate}
\mathrm{Im}(\Theta(x,t)) \geq c (d(x,x(t))^2 + d(x,x_0)^2) \geq \frac{c}{2} d(x(t),x_0)^2 \geq c' > 0
\end{equation}
uniformly over $\supp(r)$. Thus in this case $I(h) = O(e^{-C/h})$.

{\it Case 7b.} 
Let $x_0 \in x([0,T])$, but $\xi_0 \neq \xi(t)$ for any $t \in I_0$ where $I_0 = \{Êt \in [0,T] \,;\, x(t) = x_0 \}$. Note that 
\[
d_{x,t} \Theta(x(t),t) = (\xi(t) - d_x \Psi(x(t)), 0).
\]
It follows that $d_{x,t} \Theta(x(t),t)$ is nonvanishing for any $t \in I_0$. Thus if we choose $\chi \in C^{\infty}_c([0,T])$ supported close to $I_0$ with $\chi=1$ near $I_0$, then nonstationary phase \cite[Theorem 7.7.1]{Hormander}, together with its boundary version \cite[Theorem 7.7.17(i)]{Hormander} if $0 \in I_0$ or $T \in I_0$, implies that  
\[
\int_0^T \int_X e^{i \Theta(x,t)/h} r(x,t) \chi(t) \,dV \,dt = O(h^{\infty}).
\]

If we replace $\chi$ by $1-\chi$ in the integrand, we are in the region where $d(x(t),x_0) \geq c > 0$ and by \eqref{imtheta_wfu_estimate} the resulting integral is $O(e^{-C/h})$.

{\it Case 7c.}
Let now $(x_0, \xi_0) = (x(t_0), \xi(t_0))$ for some $t_0 \in [0,T]$. The phase function satisfies 
\[
\im(\Theta) \geq 0, \qquad \Theta(x(t_0),t_0) = 0, \qquad d_{x,t} \Theta(x(t_0), t_0) = 0.
\]
In order to use stationary phase, we need to show that the Hessian $\nabla_{x,t}^2 \Theta|_{(x(t_0),t_0)}$ is invertible. Computing in Riemannian normal coordinates at $(x(t_0), t_0)$ and using \eqref{dxtphi_formula}--\eqref{dttphi_formula}, we have 
\begin{equation} \label{theta_hessian_first}
\nabla_{x,t}^2 \Theta|_{(x(t_0),t_0)} = \left( \begin{array}{cc} H + G & \dot{\xi} - H \dot{x} \\ (\dot{\xi} - H \dot{x})^t & (H \dot{x} - \dot{\xi}) \cdot \dot{x} \end{array} \right)
\end{equation}
where everything is evaluated at $t=t_0$ and $G := -\nabla^2 \bar{\Psi}(x_0)$ satisfies $\im(G) > 0$. If $(\nabla_{x,t}^2 \Theta) \zeta = 0$ where $\zeta = \left(\begin{array}{c} v \\ z \end{array} \right) \in \mC^{n+1}$ with $v \in \mC^n$ and $z \in \mC$, looking at imaginary parts gives 
\[
0 = \im(\nabla_{x,t}^2 \Theta) \zeta \cdot \bar{\zeta} = \abs{\im(G)^{1/2} v}^2 + \abs{\im(H)^{1/2}(v - \dot{x}z)}^2.
\]
Thus $v=0$ and $\dot{x} z = 0$. If $\dot{x}(t_0) \neq 0$ this implies that $\zeta = 0$, showing that $\nabla_{x,t}^2 \Theta|_{(x(t_0),t_0)}$ is invertible. If $\dot{x}(t_0) = 0$ we additionally use that $0 = \nabla_{x,t}^2 \Theta \left(\begin{array}{c} 0 \\ z \end{array} \right) = \left(\begin{array}{c} \dot{\xi} z \\ 0  \end{array} \right)$, so $\dot{\xi}(t_0) z = 0$. But the condition $\dot{x}(t_0) = 0$ implies that $\dot{\xi}(t_0) \neq 0$ (otherwise $dp$ would vanish at $\gamma(t_0)$ so $\gamma(t) \equiv \gamma(t_0)$, contradicting the assumption that $\gamma$ is injective). Thus $z = 0$, showing that $\nabla_{x,t}^2 \Theta|_{(x(t_0),t_0)}$ is invertible also when $\dot{x}(t_0) = 0$.

Now if $\chi \in C^{\infty}_c([0,T])$ is supported near $t_0$ and satisfies $\chi=1$ near $t_0$, then stationary phase \cite[Theorem 7.7.5]{Hormander}, or its boundary version \cite[Theorem 7.7.17(iii)]{Hormander} if $t_0=0$ or $t_0=T$, implies that 
\[
\int_0^T \int_X e^{i \Theta(x,t)/h} r(x,t) \chi(t) \,dV \,dt = c_0 h^{\frac{n+1}{2}} + O(h^{\frac{n+1}{2}+1})
\]
where $c_0 \neq 0$ since $r(x_0,t_0) \neq 0$. When we replace $\chi$ by $1-\chi$ in the integrand, then we are either in the region where $x(t)$ is away from $x_0$ (case 7a) or $x(t)$ is close to $x_0$ but $\xi(t)$ is away from $\xi_0$ (case 7b), and the resulting integral is $O(h^{\infty})$. Combining these facts gives that $I(h) \sim h^{\frac{n+1}{2}}$. \\

\noindent {\it Step 8.} Semiclassical measure. \\

Let $u$ be as above, and let $v$ be the corresponding quasimode for $\tilde{P}^*$. Then $u$ and $v$ have the form 
\[
u(x) = \int_0^T e^{i\Phi(x,t)/h} a(x,t) \,dt, \qquad v(x) = \int_0^T e^{i \Phi(x,s)/h} \beta(x,s) \,ds.
\]
Note that $u$ and $v$ have the same phase function $\Phi$ since both $P$ and $\tilde{P}$ have principal symbol $p_m$. Moreover, $\beta$ is an amplitude analogous to $a$ satisfying 
\[
\beta(x(t),t) = \exp \left[ -i \int_0^t b_{\tilde{P}^*}(x(s)) \,ds \right] + O(h).
\]
Using Lemma \ref{lemma_conjugated_p}, we have 
\[
Qu(x) = \int_0^T e^{i\Phi(x,t)/h} \bigg[ h^{-{\ell}} q_{\ell}(x,d_x \Phi(x,t)) a(x,t) + \sum_{j=1}^{\ell} h^{j-{\ell}} R_j a(x,t) \bigg] \,dt
\]
where $R_j$ is now defined in terms of $Q$. We only need to consider the $h^{-{\ell}}$ term, since the argument below will show that the other terms will be lower order as $h \to 0$. Then 
\begin{multline*}
\int_X Qu \bar{v} \,dV =  h^{-{\ell}} \int_0^T \left[ \int_0^T \int_X e^{i\Phi(x,t)/h} q_{\ell}(x,d_x \Phi(x,t)) a(x,t) \ol{e^{i\Phi(x,s)/h} \beta(x,s)} \,dV \,dt \right] \,ds \\
 + \text{lower order terms}.
\end{multline*}

For any fixed $s \in [0,T]$ we will study the integral inside the brackets, i.e.\ the integral 
\[
I_s(h) := \int_0^T \int_X e^{i \Theta_s(x,t)/h} r_s(x,t) \,dV \,dt
\]
where 
\[
\Theta_s(x,t) := \Phi(x,t) - \ol{\Phi(x,s)}, \qquad r_s(x,t) = a(x,t) \ol{\beta(x,s)} q_{\ell}(x, d_x \Phi(x,t)).
\]

Using the conditions for $\Phi$, one has 
\[
\im(\Theta_s) \geq 0 \text{ on $\supp(r_s)$}, \qquad \Theta_s(x(s),s) = 0, \qquad d_{x,t} \Theta_s(x(s), s) = 0.
\]
The formula \eqref{theta_hessian_first} with the choices $\Psi(x) = \Phi(x,s)$ and $t_0=s$ shows that 
\begin{equation} \label{def_Ms}
M(s) := \nabla_{x,t}^2 \Theta_s|_{(x(s),s)} = \left( \begin{array}{cc} 2i\im(H) & \dot{\xi} - H \dot{x} \\ (\dot{\xi} - H \dot{x})^t & (H \dot{x} - \dot{\xi}) \cdot \dot{x} \end{array} \right).
\end{equation}
Again, this is the representation of $\nabla_{x,t}^2 \Theta_s|_{(x(s),s)}$ in Riemannian normal coordinates at $(x(s),s)$. The argument after \eqref{theta_hessian_first} shows that $M(s)$ is invertible and hence we can use stationary phase to evaluate $I_s(h)$. However, since there is an additional integration over $s$, we need to ensure that all constants are uniform over $s \in [0,T]$. To do this, first we show that there is $c > 0$ with 
\begin{equation} \label{ms_estimate_below}
\abs{M(s) \zeta}Ê\geq c \abs{\zeta}, \qquad \text{uniformly over $s \in [0,T]$}.
\end{equation}
In fact, it is enough to prove that 
\begin{equation} \label{ms_determinant_computation}
\det(M(s)) = \alpha_0 \det(\im(H(s)))
\end{equation}
for some constant $\alpha_0 \neq 0$. The computation below will show that 
\[
\alpha_0 = \frac{i}{2} (2i)^n (\abs{\dot{x}(0)}^2 + \abs{\dot{\xi}(0)}^2).
\]
Since $\im(H(s))$ is positive definite for $s \in [0,T]$, \eqref{ms_estimate_below} is a consequence of \eqref{ms_determinant_computation} and the Cramer rule for computing $M(s)^{-1}$.

For proving \eqref{ms_determinant_computation}, we write 
\[
\rho(s) := H(s) \dot{x}(s) - \dot{\xi}(s), \qquad \mathcal{R}(s) := \mathrm{Re}(H(s)), \qquad \mathcal{I}(s) := \mathrm{Im}(H(s)).
\]
Here $\mathcal{R}(s)$ and $\mathcal{I}(s)$ are real symmetric matrices. The Schur complement formula for the determinant of a block matrix yields 
\[
\det(M(s)) = \det(2i \mathcal{I}(s)) f(s) 
\]
where 
\[
f(s) := \rho(s) \cdot (\dot{x}(s) - (2i\mathcal{I}(s))^{-1} \rho(s)) = \frac{i}{2} \mathcal{I}(s)^{-1} \rho(s) \cdot \ol{\rho(s)}.
\]
In the last equality we used the fact that $\dot{x} = \im(\mathcal{I}^{-1} \rho)$. We thus have 
\begin{equation} \label{f_dot_formula}
\dot{f}(s) = \frac{i}{2} \left[ -\mathcal{I}^{-1} \dot{\mathcal{I}} \mathcal{I}^{-1} \rho \cdot \bar{\rho} + \mathcal{I}^{-1} \dot{\rho} \cdot \bar{\rho} + \mathcal{I}^{-1} \rho \cdot \bar{\dot{\rho}} \right].
\end{equation}
To compute $\dot{\rho}$, note that from the Hamilton equations and \eqref{b_c_d_definition} we obtain the formulas 
\[
\ddot{x} = B^t \dot{x} + C \dot{\xi}, \qquad \ddot{\xi} = -D \dot{x} - B \dot{\xi}.
\]
Using these formulas together with the matrix Riccati equation \eqref{matrix_riccati_equation_general}, we obtain 
\begin{equation} \label{rho_dot_formula}
\dot{\rho} = -(HC+B)\rho.
\end{equation}
Finally, taking the imaginary part of the matrix Riccati equation \eqref{matrix_riccati_equation_general} we have 
\begin{equation} \label{i_dot_formula}
\dot{\I}  + \I C \R + \R C \I + B \I + \I B^t = 0.
\end{equation}
Inserting the formulas \eqref{rho_dot_formula} and \eqref{i_dot_formula} into \eqref{f_dot_formula} proves that $\dot{f}(s) \equiv 0$. This shows \eqref{ms_determinant_computation}, where $\alpha_0 = (2i)^n f(0)$ and $f(0) = \frac{i}{2} \rho(0) \cdot \ol{\rho(0)} = \frac{i}{2} (\abs{\dot{x}(0)}^2 + \abs{\dot{\xi}(0)}^2)$.

Having proved \eqref{ms_estimate_below}, the Taylor formula shows that there is $\delta > 0$ with 
\[
\abs{d_{x,t} \Theta_s(\exp_{x(s)}(v), s+z)} \geq \frac{c}{2} \left\lvert \left( \begin{array}{c} v \\Êz \end{array} \right) \right\rvert,  \qquad \left\lvert \left( \begin{array}{c} v \\Êz \end{array} \right) \right\rvert \leq \delta, \ \ s \in [0,T].
\]
In normal coordinates at $(x(s),s)$ this means that 
\[
\frac{\abs{(x-x(s), t-s)}}{\abs{\Theta_s'(x,t)}} \leq C, \qquad \abs{(x-x(s), t-s)} \leq \delta, \ \ s \in [0,T].
\]
This is precisely the condition in the stationary phase theorem \cite[Theorem 7.7.5]{Hormander} that will make the constants uniform over $s \in [0,T]$. Note that since $Q$ vanishes near the end points of $\gamma$, we do not need to worry about boundary contributions. Thus, if $\chi \in C^{\infty}_c(\mR)$ satisfies $\chi = 1$ for $\abs{t} \leq \delta/2$ and $\supp(\chi) \subset (-\delta,\delta)$, and if the support of $\beta(x,s)$ is chosen sufficiently small depending on $\delta$, then the stationary phase theorem implies that 
\[
\left\lvert \int_0^T \int_X e^{i \Theta_s(x,t)/h} r_s(x,t) \chi(t-s) \,dV \,dt - h^{\frac{n+1}{2}} (\det(M(s)/2\pi i))^{-1/2} r_s(x(s),s) \right \rvert \leq C  h^{\frac{n+3}{2}}
\]
uniformly over $s \in [0,T]$. There is no $\abs{g}^{1/2}$ term coming from $dV$ since the computation is done in Riemannian normal coordinates. Moreover, if we replace $\chi(t-s)$ by $1-\chi(t-s)$ in the integrand, then one is either in the region where $x(t)$ is close to $x(s)$ but $\abs{t-s} \geq \delta/2$, so that the integral will be $O(h^{\infty})$ uniformly over $s \in [0,T]$ by nonstationary phase as in Case 7b above, or one is in the region where $x(t)$ is away from $x(s)$ and one can use Gaussian decay as in Case 7a above. It follows that 
\[
\abs{I_h(s) - h^{\frac{n+1}{2}} (\det(M(s)/2\pi i))^{-1/2} r_s(x(s),s)} \leq C  h^{\frac{n+3}{2}}
\]
uniformly over $s \in [0,T]$. By \eqref{ms_determinant_computation} this may be rewritten as 
\[
I_h(s) = c_0 h^{\frac{n+1}{2}} \det(\mathcal{I}(s))^{-1/2} \exp\left[ -i \int_0^s \left[ b_P(x(\sigma)) - \ol{b_{\tilde{P}^*}(x(\sigma)}) \right] \,d\sigma \right] q_{\ell}(\gamma(s)) + O(h^{\frac{n+3}{2}})
\]
uniformly over $s \in [0,T]$, where 
\begin{equation} \label{czero_constant_formula}
c_0 = (2\pi i)^{\frac{n+1}{2}} \alpha_0^{-\frac{1}{2}} = \frac{2 \pi^{\frac{n+1}{2}} }{(\abs{\dot{x}(0)}^2 + \abs{\dot{\xi}(0)}^2)^{1/2}}.
\end{equation}
This proves that, as $h \to 0$,  
\[
h^{{\ell}-\frac{n+1}{2}} \int_X Qu \bar{v} \,dV \to c_0 \int_0^T \det(\mathcal{I}(s))^{-1/2} \exp\left[ -i \int_0^s \left[ b_P(x(\sigma)) - \ol{b_{\tilde{P}^*}(x(\sigma))}) \right] \,d\sigma \right] q_{\ell}(\gamma(s)) \,ds.
\]

Using the properties of $\Phi$ and the notation \eqref{b_c_d_definition}, the function $b = b_P$ in Lemma \ref{lemma_conjugated_p} satisfies 
\[
b_P(x(t)) = \frac{1}{2i} \tr(C(t)H(t)) + p_{m-1}(\gamma(t)).
\]
As $\tilde p_m = p_m$ is real, the adjoint $\tilde P^*$, defined with respect the $L^2(M)$ inner product associated to $g$, has in some local coordinates the full symbol
\[
p_m + \frac{1}{i} \sum_{j=1}^n (\p_{x_j \xi_j} p_m + |g|^{-1/2} \p_{\xi_j} p_m \p_{x_j} |g|^{1/2})  + \overline{\tilde p_{m-1}} + r,
\]
where $r$ contains the terms of order $\le m-2$. Thus, in normal coordinates at $x(t)$,  
\[
b_{\tilde{P}^*}(x(t)) = \frac{1}{2i} \tr(C(t)H(t)) + \frac{1}{i} \tr(B(t)) + \ol{\tilde{p}_{m-1}(\gamma(t))}.
\]
It follows that 
\[
b_P(x(t)) - \ol{b_{\tilde{P}^*}(x(t))} = \frac{1}{i} \tr(C(t) \mathcal{R}(t) + B(t)) + p_{m-1}(\gamma(t)) - \tilde{p}_{m-1}(\gamma(t)).
\]
Moreover, using \eqref{i_dot_formula} and symmetries of trace, 
\[
\p_t (\log \det \I) = \tr (\I^{-1} \dot{\I})
=
- \tr(C \R + \I^{-1} \R C \I + \I^{-1 }B \I + B^t)
= - 2 \tr(C \R + B)
\]
which gives, using that $\I(0) = \mathrm{Id}$, 
\[
\det \I(t) = \exp \left[ -2 \int_0^t \tr(C(s)\R(s) + B(s)) \,ds \right].
\]
These facts imply that 
\[
\det(\mathcal{I}(t))^{-1/2} \exp\left[ -i \int_0^t \left[ b_P(x(s)) - \ol{b_{\tilde{P}^*}(x(s))}) \right] ds \right] = \exp \left[ -i \int_0^t (p_{m-1}(\gamma(s)) - \tilde{p}_{m-1}(\gamma(s))) \,ds \right].
\]
This finally proves that 
\[
h^{{\ell}-\frac{n+1}{2}} \int_X Qu \bar{v} \,dV \to c_0 \int_0^T q_{\ell}(\gamma(t)) \exp \left[ -i \int_0^t (p_{m-1}(\gamma(s)) - \tilde{p}_{m-1}(\gamma(s))) \,ds \right]\,dt
\]
where $c_0 \neq 0$ is given in \eqref{czero_constant_formula}.
\end{proof}

\begin{Remark}
For later purposes, we observe that if $M$ is a compact subset of $X$ disjoint from the end points of $\gamma$, then for $k \geq 0$ one has 
\begin{equation} \label{u_hsm_norm}
\norm{u}_{H^k(M)} = O(h^{\frac{n+1}{4}-k}).
\end{equation}
This follows from
\eqref{concentration} upon taking $\tilde{P} = P^*$, $u = v$ and $Q = R^* \chi^2 R$ where $R$ is elliptic of order $k$ (one can take $k$ even) and $\chi \in C^{\infty}_c(X)$ satisfies $\chi = 1$ near $M$ and $\chi = 0$ near the end points of $\gamma$.
\end{Remark}

\begin{Remark}
It is possible to view the construction in Theorem \ref{thm_quasimode_direct} formally in terms of a quasimode construction in $X \times \mR$ for the operator 
\[
Q(x,t,D_x,D_t) := h^{1-m} D_t + P(x, D_x).
\]
If $\gamma(s) = (x(s),\xi(s))$ is an injective null bicharacteristic segment for $P$, then $(x(s), s, \xi(s), 0)$ is a similar bicharacteristic segment for $Q$ and its spatial projection $\eta(s) = (x(s), s)$ satisfies $\dot{\eta}(s) \neq 0$ everywhere. One can then think of $e^{i \Phi(x,t)/h} a(x,t)$ as a quasimode for $Q$ near the curve $\eta$, and integrating in $t$ produces a quasimode for $P$.
\end{Remark}

We will also formulate a result related to computing the values of $u$ at $x(t)$ away from cusp points. This involves stationary phase in the $t$ variable. Near cusp points the phase function will have a degenerate critical point in $t$ and the behaviour is more complicated. For example, if $P$ is the Tricomi operator, one can check from the construction that the phase function will have cubic behaviour near cusps. For such phase functions, at least if they are real valued, the asymptotics involve Airy functions \cite[Theorem 7.7.18]{Hormander}.

\begin{Lemma} \label{lemma_bicharacteristic_solution_values}
If $u$ is the function in Theorem \ref{thm_quasimode_direct} and if $t_0 \in (0,T)$ is such that $\dot{x}(t_0) \neq 0$ and $x(t) \neq x(t_0)$ for $t \neq t_0$, then 
\[
u(x(t_0)) = h^{1/2} \frac{1}{c_{\gamma}(t_0)} \exp \left[ -i \int_0^{t_0} b_P(x(s)) \,ds \right] + O(h^{3/2})
\]
where $c_{\gamma}$ is a smooth function depending on $\gamma$ that is nonvanishing at any $t_0$ satisfying the assumption above.
\end{Lemma}
\begin{proof}
We will give the proof under the more general assumption that $\dot{x}(t) \neq 0$ for any $t$ with $x(t) = x(t_0)$, and $x(t_0)$ is not one of $x(0)$ or $x(T)$. Write $x_0 = x(t_0)$, and recall that 
\[
u(x_0) = \int_0^T e^{i\Phi(x_0,t)/h} a(x_0,t) \,dt.
\]
Let $I_{t_0} = \{ t \in [0,T] \,;\, x(t) = x_0 \}$. Since $\dot{x}(t) \neq 0$ when $t \in I_{t_0}$, points of $I_{t_0}$ are isolated and hence $I_{t_0} = \{ s_1, \ldots, s_N \}$ where $s_1 < s_2 < \ldots < s_N$ and one of the $s_j$ is $t_0$.

Note first that if $t$ is away from $I_{t_0}$, then $x(t)$ is bounded away from $x_0$ and thus by \eqref{im_phi_positive} the contribution of this region to $u(x_0)$ is $O(e^{-C/h})$. At points of $I_{t_0}$, since $x(s_j) = x_0$, the formulas \eqref{phi_req1_second}--\eqref{dttphi_formula} imply that 
\[
\Phi(x_0, s_j) = 0, \qquad \p_t \Phi(x_0, s_j) = 0, \qquad \p_t^2 \Phi(x_0, s_j) = (H\dot{x} - \dot{\xi}) \cdot \dot{x}|_{s_j}
\]
where $H, \dot{x}, \dot{\xi}$ are written in terms of Riemannian normal coordinates at $x_0$. Since $\im(H(s_j)) > 0$ and since $\dot{x}(s_j) \neq 0$ by assumption, we have $\p_t^2 \Phi(x_0, s_j) \neq 0$ and hence we can use stationary phase \cite[Theorem 7.7.5]{Hormander} near each $s_j$. This implies that 
\[
u(x_0) = h^{1/2} \sum_{j=1}^N \frac{1}{c_{\gamma}(s_j)} a(x_0, s_j) + O(h^{3/2})
\]
where 
\[
c_{\gamma}(s_j) = \left( \frac{2\pi i}{H(s_j) \dot{x}(s_j) \cdot \dot{x}(s_j)} \right)^{1/2}
\]
and, since $x_0 = x(s_j)$, by \eqref{azero_exponential_value} one has 
\[
a(x_0, s_j) = \exp \left[ -i \int_0^{s_j} b_P(x(s)) \,ds \right].
\]
Now if $t_0$ satisfies the assumption of the theorem, then $N=1$ and the result follows.
\end{proof}

\section{Determining the scattering relation and ray transforms} \label{sec_scattering_relation}

We will next prove Theorem \ref{thm_main1}. The main idea for recovering the scattering relation is the following: if $\gamma_1$ is a maximal null bicharacteristic for $P_1$ in $M$ starting at $(x_0,\xi_0)$, one constructs a solution $u_1$ of $P_1 u_1 = 0$ in a neighborhood $X$ of $M$ so that $\WF_{\mathrm{scl}}(u_1)$ is on the extension $\breve{\gamma}_1$ of $\gamma_1$ to $X$. Using the assumption $C_{P_1} = C_{P_2}$, there is a solution $\tilde{u}_2$ of $P_2 \tilde{u}_2 = 0$ in $M$ having the same Cauchy data as $u_1$ on $\p M$. In fact one can extend $\tilde{u}_2$ so that $P_2 \tilde{u}_2 = 0$ near $M$ and $u_1 = \tilde{u}_2$ outside $M$ (this is the mix-and-match construction mentioned in the introduction). One then considers the maximal null bicharacteristic $\gamma_2$ of $P_2$ in $M$ going through $(x_0,\xi_0)$.

The above construction ensures that $\WF_{\mathrm{scl}}(u_1|_{X \setminus M})$ contains a part near $(x_0,\xi_0)$ and also a part near $\alpha_{P_1}(x_0,\xi_0)$. Since $u_1 = \tilde{u}_2$ outside $M$, also $\WF_{\mathrm{scl}}(\tilde{u}_2|_{X \setminus M})$ contains these two parts. On the other hand, since $P_2 \tilde{u}_2 = 0$ near $M$ and since $\gamma_2$ contains $(x_0,\xi_0)$, by semiclassical propagation of singularities $\WF_{\mathrm{scl}}(\tilde{u}_2|_{X \setminus M})$ must also contain parts near the start and end points of $\gamma_2$. Combining these conditions eventually leads to the fact that $\alpha_{P_1}(x_0,\xi_0) = \alpha_{P_2}(x_0,\xi_0)$. In the proof we need to deal with the possibility of tangential contacts with $\p M$, which adds some technicalities.

If the principal parts of $P_1$ and $P_2$ agree, the recovery of subprincipal information in Theorem \ref{thm_main1} is also based on constructing solutions concentrating near a maximal null bicharacteristic. However, instead of using the mix-and-match construction, it is more convenient to employ the integral identity in Lemma \ref{lemma_integral_identity_potential} and the formula for the semiclassical limit measure in \eqref{concentration}.

\begin{proof}[Proof of Theorem \ref{thm_main1}]
The proof will be done in several steps.

\vspace{10pt}

\noindent {\it Step 1.} Preparation.

\vspace{10pt}

We begin by extending both $P_1$ and $P_2$ smoothly to a slightly larger manifold $X$, in such a way that $P_1 = P_2$ in $X \setminus M$ (this is possible since $P_1 = P_2$ to infinite order at $\p M$). The principal symbols $p_j = p_{j,m}$ of $P_j$ then satisfy 
\begin{equation} \label{pone_ptwo_exterior_condition}
p_1 = p_2 \text{ in $T^*(X \setminus M^{\mathrm{int}})$}.
\end{equation}
In particular, the null directions on $\p(T^* M)$ are the same for $P_1$ and $P_2$. Note also that if $X$ is chosen sufficiently small, both $P_1$ and $P_2$ will be real principal type in $X$ (i.e.\ the nontrapping condition will hold in $X$), see \cite[proof of Theorem 26.1.7]{Hormander}.

%Let $(x_0, \xi_0) \in \p_{\null}^{-,P_1}(T^* M)$, and let $\gamma_1: [0,T_1] \to T^* M \setminus 0$ be the maximal null $P_1$-bicharacteristic with $\gamma_1(0) = (x_0,\xi_0)$. We denote by $\breve{\gamma}_1$ the maximal continuation of $\gamma_1$ as a null bicharacteristic of $P_1$ in $X$. The fact that $\gamma_1$ is maximal means that there are sequences $(\delta_j), (\eps_j)$ of positive numbers with $\delta_j \to 0, \eps_j \to 0$ and 
%\[ \breve{\gamma}_1(-\delta_j),  \breve{\gamma}_1(T_1+\eps_j) \in T^*(X \setminus M). \]
%Let $\gamma_2: [-S_2,T_2] \to T^* M \setminus 0$ be the maximal null $P_2$-bicharacteristic with $\gamma_2(0) = (x_0,\xi_0)$ (which is a null direction also for $P_2$ by \eqref{p_infinite_order}).

%We wish to show that $(x_0, \xi_0) \in \p_{\null}^{-,P_2}(T^* M)$ (i.e.\ that $S_2 = 0$), and that $\gamma_1(T_1) = \gamma_2(T_2)$. This will prove that $\p_{\null}^{-,P_1}(T^* M) \subset \p_{\null}^{-,P_2}(T^* M)$ and $\alpha_{P_1}|_{\p_{\null}^{-,P_1}(T^* M)} = \alpha_{P_2}|_{\p_{\null}^{-,P_2}(T^* M)}$. Interchanging the roles of $P_1$ and $P_2$ and repeating the same argument for the outgoing boundary will then show that 
%\[ \p_{\null}^{\pm,P_1}(T^* M) = \p_{\null}^{\pm,P_2}(T^* M) \text{ and } \alpha_{P_1} = \alpha_{P_2}. \]

%Since 
%\[ \p_{\pm, P_j}(T^* M) = \{ (x,\xi) \in \mathrm{Char}(P_j) \,;\, \mp d_{\xi} p_j(x,\xi)(\nu) > 0 \} \]
%where $d_{\xi} p_j = \pi_*(H_{p_j})$, it follows from \eqref{p_infinite_order} that $\p_{\pm,P_1}(T^* M) = \p_{\pm,P_2}(T^*M)$.

Let $(x_0, \xi_0) \in \p'_{\mathrm{null},P_1}(T^* M)$, and let $\gamma_1: [0,T_1] \to T^* M \setminus 0$ be the maximal null bicharacteristic for $P_1$ with $\gamma_1(0) = (x_0,\xi_0)$ (the case where $\gamma_1(T_1) = (x_0,\xi_0)$ is analogous). We denote by $\breve{\gamma}_1$ the maximal continuation of $\gamma_1$ as a null bicharacteristic of $P_1$ in $X$. Since $\gamma_1$ is maximal, there is a strictly increasing sequence $(t_j^-)_{j=1}^{\infty}$ with $t_j^- \to 0$ and a strictly decreasing sequence $(t_j^+)_{j=1}^{\infty}$ with $t_j^+ \to T_1$ such that 
\[
\gamma_1(t_j^{\pm}) \in T^*(X \setminus M).
\]
Since $d(x_1(t_j^{\pm}), M) \to 0$ as $j \to \infty$, where $x_1(t) = \pi(\gamma_1(t))$, after passing to subsequences we may choose a strictly decreasing sequence $(\kappa_j)$ with 
\begin{equation} \label{kappaj_conditions}
d(x_1(t_{j+1}^{\pm}), M) < \kappa_j < d(x_1(t_j^{\pm}), M), \qquad \kappa_j \to 0.
\end{equation}
Then $X_j := \{ x \in X \,;\, d(x, M) < \kappa_j \}$ is an open manifold so that $M \subset X_j \subset \subset X$ and 
\begin{equation} \label{gammaone_description}
\gamma_1(t_j^{\pm}) \in T^*(X \setminus \ol{X}_j).
\end{equation}
Let $\gamma_2: [-S_2,T_2] \to T^* M \setminus 0$ be the maximal null bicharacteristic for $P_2$ in $M$ with $\gamma_2(0) = (x_0,\xi_0)$ and $S_2, T_2 \geq 0$.% Since $(x_0,\xi_0) \in \p_{+,P_2}(T^* M)$, one actually has $S_2 = 0$ and $\gamma_2: [0,T_2] \to T^* M \setminus 0$.

\vspace{10pt}

\noindent {\it Step 2.} Strategy.

\vspace{10pt}

We wish to show that $S_2=0$ or $T_2=0$, i.e.\ $(x_0,\xi_0) \in \p'_{\mathrm{null},P_2}(T^* M)$, and that 
\[
\{ \gamma_1(0), \gamma_1(T_1) \} = \{ \gamma_2(-S_2), \gamma_2(T_2) \}.
\]
This will imply that $\p'_{\mathrm{null},P_1}(T^* M) \subset \p'_{\mathrm{null},P_2}(T^* M)$ and $\alpha_{P_1} = \alpha_{P_2}$ on $\p'_{\mathrm{null},P_1}(T^* M)$. Changing the roles of $P_1$ and $P_2$ then implies that $\p'_{\mathrm{null},P_1}(T^* M) = \p'_{\mathrm{null},P_2}(T^* M)$ and $\alpha_{P_1} = \alpha_{P_2}$.

Below we will work with $j$ fixed, and we will suppress $j$ from the notation of $u_1$ etc.

\vspace{10pt}

\noindent {\it Step 3.} There is $u_1 \in H^m(X_j)$ with $P_1 u_1 = 0$ in $X_j$ and $\WF_{\text{scl}}(u_1) = \breve{\gamma}_1([t_j^-,t_j^+]) \cap T^* X_j$.

\vspace{10pt}

Let $v_1$ be the quasimode associated to $\breve{\gamma}_1|_{[t_j^-,t_j^+]}$ in $X$ constructed in Theorem \ref{thm_quasimode_direct}, satisfying 
\[
\WF_{\mathrm{scl}}(v_1) = \breve{\gamma}_1([t_j^-,t_j^+]), \qquad \WF_{\mathrm{scl}}(P_1 v_1) = \breve{\gamma}_1(t_j^-) \cup \breve{\gamma}_1(t_j^+).
\]
Then $\norm{P_1 v_1}_{H^s(X_j)} = O(h^{\infty})$ for any fixed $s$ since $\WF_{\mathrm{scl}}(P_1 v_1)$ is away from $\ol{X}_j$ using \eqref{gammaone_description}. Let $u_1 = v_1|_{X_j} + r_1$, where $r_1 \in H^{s+m-1}(X_j)$ is the solution given in Proposition \ref{prop_real_principal_type_solvability} of the equation 
\[
P_1 r_1 = -P_1 v_1 \text{ in $X_j$}
\]
with 
\[
\norm{r_1}_{H^{s+m-1}(X_j)} \lesssim \norm{P_1 v_1}_{H^s(X_j)} = O(h^{\infty}).
\]
Then $\WF_{\text{scl}}(u_1) = \WF_{\text{scl}}(v_1|_{X_j}) = \breve{\gamma}_1([t_j^-,t_j^+]) \cap T^* X_j$, and $P_1 u_1 = 0$ in $X_j$.

\vspace{10pt}

\noindent {\it Step 4.} There is $\tilde{u}_2 \in H^m(X_j)$ so that $P_2 \tilde{u}_2 = 0$ in $X_j$, $\tilde{u}_2 = u_1$ in $X_j \setminus M$, and $\tilde{u}_2$ is $L^2$-tempered.

\vspace{10pt}

To prove this we invoke the assumption $C_{P_1} = C_{P_2}$, which implies that there is $w_2 \in H^m(M)$ such that 
\[
P_2 w_2 = 0 \text{ in $M$}, \qquad \nabla^k w_2|_{\p M} = \nabla^k u_1|_{\p M} \text{ for $k \leq m-1$.}
\]
Define the function 
\[
\tilde{u}_2 := \left\{ \begin{array}{cl} w_2, & \text{ in $M$}, \\ u_1, & \text{ in $X_j \setminus M$}. \end{array} \right.
\]
Then $\tilde{u}_2 \in H^m(X_j)$, and $P_2 \tilde{u}_2 = 0$ in $X_j$ since this holds separately in $M$ and $X_j \setminus M$. It remains to show that $\tilde{u}_2$ is $L^2$-tempered. First note that $u_1 - \tilde{u}_2$ is in $H^m_M(X_j)$ and solves 
\[
P_2(u_1-\tilde{u}_2) = -P_2 u_1.
\]
We may add a function in $N(P_2)$ to $\tilde{u}_2$ so that one still has $P_2 \tilde{u}_2 = 0$ in $X_j$ and $\tilde{u}_2 = u_1$ in $X_j \setminus M$, and additionally $u_1-\tilde{u}_2 \in N(P_2)^{\perp}$. Then \eqref{v_htm_equation} implies that 
\begin{equation} \label{u1_tildeu2_estimate}
\norm{u_1-\tilde{u}_2}_{L^2_M} \lesssim \norm{P_2(u_1-\tilde{u}_2)}_{L^2_M} = \norm{P_2 u_1}_{L^2_M} \lesssim \norm{u_1}_{H^m(M)}.
\end{equation}
Thus $\norm{\tilde{u}_2}_{L^2_M} \lesssim \norm{u_1}_{H^m(M)} = O(h^{\frac{n+1}{4}-m})$ by \eqref{u_hsm_norm}. It follows that $\tilde{u}_2$ is $L^2$-tempered.

\vspace{10pt}

\noindent {\it Step 5.} $\gamma_1(\sigma_-) = \gamma_2(-S_2)$ and $\gamma_1(\sigma_+) = \gamma_2(T_2)$ for some $\sigma_{\pm} \in \{ 0, T_1 \}$.

\vspace{10pt}

By Step 3 we have $\WF_{\text{scl}}(u_1)= \breve{\gamma}_1([t_j^-,t_j^+]) \cap T^* X_j$. Since $\tilde{u}_2 = u_1$ in $X_j \setminus M$, we further have 
\begin{equation} \label{wfscl_tildeutwo}
\WF_{\text{scl}}(\tilde{u}_2|_{X_j \setminus M}) = (\breve{\gamma}_1([t_j^-,0)) \cup \breve{\gamma}_1((T_1,t_j^+])) \cap T^*(X_j \setminus M).
\end{equation}
Using \eqref{kappaj_conditions} we have $\breve{\gamma}_1(t_k^{\pm}) \in \WF_{\text{scl}}(\tilde{u}_2|_{X_j \setminus M})$ for $k > j$, and since the wave front set is closed this implies that $\gamma_1(0), \gamma_1(T_1) \in \WF_{\text{scl}}(\tilde{u}_2)$. However, since $\gamma_1(0) = \gamma_2(0)$, semiclassical propagation of singularities for the solution $\tilde{u}_2$ of $P_2 \tilde{u}_2 = 0$ in $X_j$, together with the maximality of $\gamma_2$, implies that $\breve{\gamma}_2(\tau_k^{j,\pm}) \in \WF_{\text{scl}}(\tilde{u}_2|_{X_j \setminus M})$ for some increasing sequence $\tau_k^{j,-} \to -S_2$ and decreasing sequence $\tau_k^{j,+} \to T_2$ where $\breve{\gamma}_2$ is the continuation of $\gamma_2$ to $X$. After passing to subsequences, we may assume that $-S_2-1/k < \tau_k^{j,-} < -S_2$ and $T_2 < \tau_k^{j,+} < T_2 + 1/k$.

Up to now we have been working with a fixed $j$. Next we consider the diagonal sequences $\tau_j^{j,\pm}$ as $j \to \infty$, and note that by \eqref{wfscl_tildeutwo} one has 
\begin{equation} \label{breve_gammaone_gammatwo_points}
\breve{\gamma}_2(\tau_j^{j,\pm}) = \breve{\gamma}_1(\sigma_j^{\pm})
\end{equation}
for some $\sigma_j^{\pm}$ with $\sigma_j^{\pm} \in [t_j^-,0] \cup [T_1,t_j^+]$.

It follows that $\tau_j^{j,-} \to -S_2$ and $\tau_j^{j,+} \to T_2$, as well as (possibly after passing to subsequences) $\sigma_j^{\pm} \to \sigma^{\pm}$ for some $\sigma_{\pm} \in \{0, T_1 \}$. Thus by \eqref{breve_gammaone_gammatwo_points} one has 
\[
\gamma_1(\sigma_-) = \gamma_2(-S_2), \qquad \gamma_1(\sigma_+) = \gamma_2(T_2).
\]

\vspace{10pt}

\noindent {\it Step 7.} $\p'_{\mathrm{null},P_1}(T^* M) = \p'_{\mathrm{null},P_2}(T^* M)$ and $\alpha_{P_1} = \alpha_{P_2}$.

\vspace{10pt}

Now if $\sigma_- \neq \sigma_+$, then necessarily $T_1 > 0$ and $\{ \sigma_-, \sigma_+ \} = \{ 0, T_1 \}$. Thus 
\[
\{ \gamma_1(0), \gamma_1(T_1) \} = \{ \gamma_2(-S_2), \gamma_2(T_2) \}.
\]
In particular, $(x_0,\xi_0) = \gamma_1(0) = \gamma_2(0)$ is one of the endpoints of $\gamma_2$, so either $S_2=0$ or $T_2 = 0$.

On the other hand, assume that $\sigma_- = \sigma_+ = 0$ (the case $\sigma_{\pm} = T_1$ is analogous). Then $\gamma_2(-S_2) = \gamma_2(T_2)$ and consequently $S_2 = T_2 = 0$ (since $\gamma_2|_{[-S_2,T_2]}$ is injective by the real principal type assumption). We wish to prove that also $T_1 = 0$. If this is not true, then $T_1 > 0$ and $\gamma_1(0) \neq \gamma_1(T_1)$ using that $\gamma_1|_{[0,T_1]}$ is injective. Returning to \eqref{breve_gammaone_gammatwo_points}, after passing to subsequences if necessary, it follows that $\sigma_j^{\pm} \leq 0$ for all $j$. Now in local coordinates in $T^* M$, the curves $\breve{\gamma}_j(t)$ for small $t$ are very close to the line $(x_0,\xi_0) + t\dot{\gamma}_1(0)$ where $\dot{\gamma}_1(0) = \dot{\gamma}_2(0) \neq 0$ by \eqref{pone_ptwo_exterior_condition} and the real principal type assumption. Since $\tau_j^{j,-} < 0$, $\tau_j^{j,+} > 0$ and $\sigma_j^{\pm} \leq 0$, this leads to a contradiction with \eqref{breve_gammaone_gammatwo_points}. This proves that $T_1 = 0$ and 
\[
\{ \gamma_1(0), \gamma_1(T_1) \} = \{ \gamma_2(-S_2), \gamma_2(T_2) \}
\]
also when $\sigma_+ = \sigma_-$.

We have now proved the claims stated in Step 2, and thus Step 7 is complete.

\vspace{10pt}

\noindent {\it Step 8.} If $P_1$ and $P_2$ have the same principal symbol and if $\mu$ is a half density on $M$, one has 
\[
\exp \left[ i\int_0^T \sigma_\subp[P_1^\mu](\gamma(t)) \,dt \right] 
= 
\exp \left[ i\int_0^T \sigma_\subp[P_2^\mu](\gamma(t)) \,dt \right]
\]
whenever $\gamma: [0,T] \to T^* M \setminus 0$ is a maximal null bicharacteristic in $M$.

\vspace{10pt}

By the Remark after Theorem \ref{thm_main1}, it is enough to prove that 
\[
\exp \left[ i\int_0^T (p_{m-1,1} - p_{m-1,2})(\gamma(t)) \,dt \right] 
= 0.
\]
Write $Q = P_1 - P_2$. Since $P_1 = P_2$ in $X \setminus M$, the operator $Q$ is a differential operator with smooth coefficients in $X$ that vanishes in $X \setminus M$ and has principal symbol 
\[
q_{m-1} = p_{m-1,1} - p_{m-1,2}.
\]
We denote by $\breve{\gamma}$ the continuation of $\gamma$ as a null bicharacteristic curve in $X$.

We now use Theorem \ref{thm_quasimode_direct} with $P = P_1$ and $\tilde{P} = P_2$ and construct quasimodes $v_1, v_2$ associated with $\breve{\gamma}|_{[-\eps_1,T+\eps_2]}$ with $\eps_j$ chosen so that the end points are outside $M$. Then $\norm{P_1 v_1}_{H^s(M)} = \norm{P_2^* v_2}_{H^s(M)} = O(h^{\infty})$. As in Step 3, we construct solutions $u_j = v_j + r_j$ of $P_1 u_1 = 0$ and $P_2^* u_2 = 0$ in $M$ with $\norm{r_j}_{H^{s+m-1}(M)} = O(h^{\infty})$.

Since $C_{P_1} = C_{P_2}$ and $Q = P_1-P_2$ vanishes outside $M$, the integral identity in Lemma \ref{lemma_integral_identity_potential} implies that 
\[
0 = (Q u_1, u_2)_{L^2(M)} = (Q v_1, v_2)_{L^2(X)} + O(h^{\infty}).
\]
Multiplying this identity by $h^{-\frac{n+1}{2}+m-1}$ and using \eqref{concentration}, we obtain 
\[
\int_0^T q_{m-1}(\gamma(t)) \exp \left[ -i \int_0^t q_{m-1}(\gamma(s)) \,ds \right] \,dt = 0.
\]
The integrand is equal to $i \p_t (\exp \left[ -i \int_0^t q_{m-1}(\gamma(s)) \,ds \right])$, hence it follows that 
\[
\exp \left[ -i \int_0^T q_{m-1}(\gamma(t)) \,dt \right] = 1.
\]
This is the required statement.
\end{proof}

\begin{Remark}
If the maximal null bicharacteristic $\gamma: [0,T] \to T^*M \setminus 0$ is sufficiently nice, e.g.\ $\dot{x}(t) \neq 0$ and $x(t)$ is injective on $[0,T]$, one can give an alternative proof of Step 8 by using the mix-and-match construction and a similar argument as in Steps 3--5. That is, one constructs solutions of $P_j u_j = 0$ associated with $\gamma$, and defines another solution of $P_2 \tilde{u}_2 = 0$ in $X_1$ so that $\tilde{u}_2 = u_1$ outside $M$. By semiclassical propagation of singularities, $u_2 - \tilde{u}_2$ is $O(h^{\infty})$ at the end point $x(T)$. Comparing the values of $u_1$ and $u_2$ at $x(T)$ using Lemma \ref{lemma_bicharacteristic_solution_values} proves Step 8.
\end{Remark}

\begin{proof}[Proof of Theorem \ref{thm_main2}]
The proof is the same as in Step 8 in the proof of Theorem \ref{thm_main1}. By Lemma \ref{lemma_integral_identity_potential}, the assumption $C_{P+Q_1} = C_{P+Q_2}$ implies the integral identity  
\[
( (Q_1-Q_2) u_1, u_2)_{L^2(M)} = 0
\]
whenever $u_j \in H^m(M)$ solve $(P + Q_1) u_1 = (P^* + Q_2^*) u_2 = 0$ in $M$. We wish to find such solutions $u_j$ that are close to the quasimodes constructed in Theorem \ref{thm_quasimode_direct}.

Let $\gamma: [0,T] \to T^* M \setminus 0$ be a maximal null bicharacteristic in $M$. We extend $P$ as a smooth differential operator to some slightly larger open manifold $X$ and let $\breve{\gamma}$ be the continuation of $\gamma$ to $X$. We also extend $Q_j$ smoothly to $X$ so that the coefficients of $Q_1-Q_2$ vanish outside $X$. Choose some $\eps_j > 0$ so that $\breve{\gamma}(-\eps_1)$ and $\breve{\gamma}(T+\eps_2)$ are in $X \setminus M$. Let $v_1 = v_{1,h} \in C^{\infty}_c(X)$ be the quasimode given by Theorem \ref{thm_quasimode_direct} related to $\breve{\gamma}|_{[-\eps_1,T+\eps_2]}$ satisfying $\norm{(P+Q_1) v_1}_{H^s(M)} = O(h^{\infty})$, and let $v_2 = v_{2,h}$ be the corresponding quasimode satisfying $\norm{(P^* + Q_2^*) v_2}_{H^s(M)} = O(h^{\infty})$.

We define $u_j = v_j + r_j$, where $r_j$ are correction terms satisfying 
\[
(P+Q_1)r_1 = f_1 \text{ in $M$}, \qquad (P^* + Q_2^*)r_2 = f_2 \text{ in $M$},
\]
where $f_1 = -(P+Q_1)v_1|_M$ and $f_2 = -(P^*+Q_2^*)v_2|_M$. Note that $f_j \in C^{\infty}(M)$ with $f_1 \in N(P_1^*+Q_1^*)^{\perp}$ and $f_2 \in N(P+Q_2)^{\perp}$. By Proposition \ref{prop_real_principal_type_solvability}, there exist solutions $r_j \in H^{s+m-1}(M)$ of the above equations with the norm estimates 
\[
\norm{r_j}_{H^{s+m-1}(M)} \lesssim \norm{f_j}_{H^s(M)} = O(h^{\infty}).
\]
Then $u_j \in H^{s+m-1}(M)$ also solve $(P + Q_1) u_1 = (P^* + Q_2^*) u_2 = 0$ in $M$.

Inserting the solutions $u_j$ in the integral identity, we obtain that 
\[
\int_M (Q_1-Q_2) v_1 \bar{v}_2 \,dV = O(h^{\infty}).
\]
Using Theorem \ref{thm_quasimode_direct} and the fact that the coefficients of $Q_1-Q_2$ vanish outside $M$, this implies that 
\[
\int_0^T \sigma_{\pr}[(Q_1-Q_2)](\gamma(t)) \,dt = 0.
\]
Here we also used that the subprincipal symbols of $P+Q_1$ and $P+Q_2$ agree, so that the exponential factor in \eqref{concentration} goes away.
\end{proof}

We next consider the case of strictly hyperbolic operators and prove Theorem \ref{thm_main_hyper_intro}. First we verify that strictly hyperbolic operators are indeed real principal type (this should be well known, but we give a proof for completeness in Appendix \ref{sec_appendix}).

\begin{Lemma} \label{lemma_hyperbolic_real_principal_type}
Let $X$ be an open manifold, let $\phi \in C^{\infty}(X,\mR)$, and let $P$ be a strictly hyperbolic differential operator of order $m$ in $X$ with respect to the level surfaces of $\phi$. Then $P$ is of real principal type in $X$.
\end{Lemma}

\begin{proof}[Proof of Theorem \ref{thm_main_hyper_intro}]
We begin by establishing suitable analogues of
Proposition \ref{prop_real_principal_type_solvability} (see Step 1 below) and
Lemma \ref{lemma_integral_identity_potential} (Step 2 below),
and then indicate the minor modifications required in the proofs of Theorems \ref{thm_main1} and  \ref{thm_main2} by giving a proof of \eqref{hyper_parallel_tr} (Step 3 below). The other claims can be proven using analogous modifications. 

Choose an open manifold $Y_1 \subset X$ so that $M_1 \subset Y_1$ and $\overline{Y_1}$ is compact.  
Define     
    \begin{align*}
X_1 = \{x \in Y_1 \,;\, 0 < \phi(x) < T \}.
    \end{align*}

\vspace{10pt}

\noindent {\it Step 1.} Given any $f \in H_0^{s}(X_1)$, with $s \ge 1$, there is a solution $u \in H^{s+m-1}(X_1)$ of $P_1 u = f$ satisfying 
$\tr_{\Gamma_-}^{m-1} u = 0$ and 
    \begin{align}\label{hyper_cont}
\norm{u}_{H^{s+m-1}(X_1)} \le C \norm{f}_{H^{s}(X_1)}.
    \end{align}

\vspace{10pt}

We define the space
    \begin{align*}
\mathcal H &= \{u \in H^{s+m-1}(X_1) \,;\, P_1 u \in H^{s}_0(X_1),\ \tr_{\Gamma_-}^{m-1} u = 0\},
    \end{align*}
equipped with the norm defined by
    \begin{align*}
\norm{u}_{\mathcal H}^2 = \norm{u}_{H^{s+m-1}(X_1)}^2 + \norm{P_1 u}_{H^{s}(X_1)}^2.
    \end{align*}
Now $\mathcal H$ is a Hilbert space
and writing $P = P_1|_{\mathcal H}$, $P : \mathcal H \to H_0^{s}(X_1)$ is clearly continuous.
Let us show that $P$ is surjective. 
Let $f \in H_0^{s}(X_1)$. We may extend $f$ by zero to obtain a function in $H^{s}(X)$, still denoted by $f$. By \cite[Theorem 23.2.4]{Hormander} there is $u \in H^{s+m-1}_{loc}(X)$ such that $\supp(u) \subset \phi^{-1}([0,\infty))$ and $P_1 u=f$ on $X_1$. 
Hence $P$ is surjective. 

It follows now (see \cite[Proposition 2.3 and 2.4]{EnglHankeNeubauer}) that the pseudoinverse $P^\dagger$ of $P$ is continuous 
$P^\dagger : H_0^{s}(X_1) \to \mathcal H$
and $PP^\dagger f = f$. Thus $u = P^\dagger f$ solves $Pu = f$ with the continuous dependence (\ref{hyper_cont}).

\vspace{10pt}

\noindent {\it Step 2.} If $C_{P_1}^\lat  = C_{P_2}^\lat $, then 
    \begin{align}\label{hyper_alessandrini}
( (P_1 - P_2) u_1, u_2)_{L^2(M)} = 0
    \end{align}
whenever $u_j \in H^m(M)$ solve $P_1 u_1 = P_2^* u_2 = 0$ in $M$,
 $\tr_{\Gamma_-}^{m-1} u_1 = 0$ and $\tr_{\Gamma_+}^{m-1} u_2 = 0$. 
\vspace{10pt}

As $C_{P_1}^\lat \subset C_{P_2}^\lat$
we can find $\tilde u_2 \in H^m(M)$ such that $P_2 \tilde u_2 = 0$ in $M$ and $\tr_{\Gamma \cup \Gamma_-}^{m-1} \tilde u_2 = \tr_{\Gamma \cup \Gamma_-}^{m-1} u_1$. Then, using also $\tr_{\Gamma_+}^{m-1} u_2 = 0$,
    \begin{align*}
((P_1 - P_2) u_1, u_2) = -(P_2 u_1, u_2) = - (P_2(u_1-\tilde u_2),u_2) 
= -(u_1-\tilde u_2,P_2^* u_2) = 0.
    \end{align*}

\vspace{10pt}

\noindent {\it Step 3.} If $P_1$ and $P_2$ have the same principal symbol and if $\mu$ is a half density on $M$, one has 
\[
\exp \left[ i\int_0^T \sigma_\subp[P_1^\mu](\gamma(t)) \,dt \right] 
= 
\exp \left[ i\int_0^T \sigma_\subp[P_2^\mu](\gamma(t)) \,dt \right]
\]
whenever $\gamma: [0,T] \to T^* M \setminus 0$ is a maximal null bicharacteristic in $M$ that does not meet $\p(T^* M)$ away from $\Gamma$.

\vspace{10pt}

It is again enough to prove 
    \begin{align}\label{hyper_subp}
\exp \left[ i\int_0^T q_{m-1}(\gamma(t)) \,dt \right]
= 1,
    \end{align}
where $q_{m-1}$ is the principal symbol of $Q = P_1 - P_2$. 
We denote by $\breve{\gamma}$ the continuation of $\gamma$ as a null bicharacteristic curve in $X$.

We use Theorem \ref{thm_quasimode_direct} with $P = P_1$ and $\tilde{P} = P_2$ to obtain quasimodes $v_1, v_2$ associated with $\breve{\gamma}|_{[-\eps_1,T+\eps_2]}$ with $\eps_j$ chosen so that the end points are outside $X_1$. 
Observe that, when $X_1$ is chosen small enough, $v_j$ can be constructed so that, writing $\Sigma =  \phi^{-1}(0) \cup \phi^{-1}(T)$, one has $\supp(v_j) \cap  \Sigma = \emptyset$.
We choose $\chi \in C^\infty(X_1)$ so that $\chi = 1$ in $M$
and that $\chi = 0$ near $\p X_1 \setminus \Sigma$.
Finally we choose  $r_1 \in H^m(X_1)$ so that $P_1 r_1 = - \chi P_1 v_1$,
$\tr_{\Gamma_-}^{m-1} r_1 = 0$ and that (\ref{hyper_cont}) holds with $s=1$.
Then
    \begin{align*}
\norm{r_1}_{H^m(X_1)} \le C \norm{\chi P_1 v_1}_{H^1(X_1)} = O(h^{\infty}),
    \end{align*}
and setting $u_1 = v_1 + r_1$, it follows that $P_1 u_1 = 0$ in $M$.

Note that $P_2$ is strictly hyperbolic also with respect to the level surfaces of $-\phi$. Repeating the argument in Step 1, with $P_1$ replaced by $P_2$ and $\phi$ by $-\phi$, we see that  
there is a solution $r_2 \in H^{s+m-1}(X_1)$ of $P_2 r_2 = -\chi P_2 v_2$ satisfying 
$\tr_{\Gamma_+}^{m-1} r_2 = 0$ and 
$\norm{r_2}_{H^{1}(X_1)} = O(h^\infty)$.
We set $u_2 = v_2 + r_2$.

Since $C_{P_1} = C_{P_2}$ and $Q = P_1-P_2$ vanishes outside $M$, the integral identity (\ref{hyper_alessandrini}) implies that 
\[
0 = (Q u_1, u_2)_{L^2(M)} = (Q v_1, v_2)_{L^2(X)} + O(h^{\infty}).
\]
Multiplying this identity by $h^{-\frac{n+1}{2}+m-1}$ and using \eqref{concentration}, we obtain (\ref{hyper_subp}) as in the proof of Theorem~\ref{thm_main1}.
\end{proof}

We will show next that by small modifications of the above proofs, one can obtain partial data results in the corresponding inverse problems.
Let $\Gamma \subset \p M$ be a nonempty open set.
We use the notation
    \begin{align*}
\norm{\tr_{\Gamma}^{m-1} u}_{\mathcal H^{m-1}(\Gamma)}^2
= \sum_{j=0}^{m-1} \norm{\nabla^j u|_\Gamma}_{L^2(\Gamma)}^2,
    \end{align*}
and consider the partial Cauchy data set 
    \begin{align*}
C_{P,\Gamma} = \{ (u|_{\Gamma}, \nabla u|_{\Gamma}, \ldots, \nabla^{m-1} u|_{\Gamma}) \,;\, 
&u \in H^m(M) \text{ solves } Pu = 0 \text{ in $M$
and } 
\\&
\norm{\tr_{\p M \setminus \Gamma}^{m-1} u}_{\mathcal H^{m-1}(\p M \setminus \Gamma)} \le \norm{\tr_{\Gamma}^{m-1} u}_{\mathcal H^{m-1}(\Gamma)} \}.
    \end{align*}
The last condition means that we are only using solutions of $Pu = 0$ whose Cauchy data on $\p M \setminus \Gamma$ is not much larger than the Cauchy data on $\Gamma$ (the constant $1$ in the inequality is quite arbitrary and could be replaced by any other fixed constant). For simplicity we do not consider the scattering relation in the following result, whose proof is given in Appendix \ref{sec_appendix}.

\begin{Theorem} \label{thm_partial_data}
Let $M$ be a compact manifold with smooth boundary, let $P_1, P_2$ be real principal type differential operators of order $m \geq 1$ in $M$, and let $\Gamma \subset \p M$ be a nonempty open set. 
Assume that $P_1 = P_2$ to infinite order on $\p M$. 
%If 
%\[
%C_{P_1,\Gamma} = C_{P_2,\Gamma},
%\]
%then 
%\begin{align*}
%\alpha_{P_1}(x,\xi) = \alpha_{P_2}(x,\xi)
%\end{align*}
%for any $(x,\xi) \in \p'_{\mathrm{null},P_1}(T^* M)$ such that both for $j=1,2$, the maximal null $P_j$-bicharacteristic curve through $(x,\xi)$ does not meet $\p(T^* M)$ away from $\Gamma$.
If the principal symbols of $P_1$ and $P_2$ coincide, then for any nonvanishing half density $\mu$ on $M$ one has 
    \begin{align}\label{subprin_part_data}
\exp \left[ i\int_0^T \sigma_\subp[P_1^\mu](\gamma(t)) \,dt \right] 
= 
\exp \left[ i\int_0^T \sigma_\subp[P_2^\mu](\gamma(t)) \,dt \right]
    \end{align}
whenever $\gamma : [0,T] \to T^* M$ is a maximal null bicharacteristic curve for $P_1$ 
whose spatial projection 
intersects $\Gamma$ transversally and 
does not meet $\p M$ away from $\Gamma$.

Finally, if $m \geq 2$ and if $P_j = P + Q_j$ where $Q_j$ has order $\leq m-2$ for $j=1,2$, then 
then 
    \begin{align}\label{lot_part_data}
\int_{0}^T \sigma_{\pr}[Q_1](\gamma(t)) \,dt = \int_{0}^T \sigma_{\pr}[Q_2](\gamma(t)) \,dt 
    \end{align}
whenever $\gamma : [0,T] \to T^* M$ is a maximal null bicharacteristic curve for $P$ 
whose spatial projection 
intersects $\Gamma$ transversally and 
does not meet $\p M$ away from $\Gamma$.
\end{Theorem}
%\LO{The analogous result for the scattering relation seems to require stronger notion of Cauchy data. For example the following could work:}
%    \begin{align*}
%\tilde C_{P,\Gamma} = \{ (u|_{\Gamma}, \nabla u|_{\Gamma}, \ldots, \nabla^{m-1} u|_{\Gamma}, \epsilon) \,;\, &u \in H^m(M) \text{ solves } Pu = 0 \text{ in $M$ and }  \\&
%\norm{\tr_{\Gamma}^{m-1} u}_{\mathcal H^{m-1}(\p M \setminus \Gamma)} \le \epsilon\}.
%    \end{align*}

Finally, we show that for first order operators a converse of Theorem \ref{thm_main1} holds: the scattering relation and integrals of the subprincipal symbol modulo $2\pi \mZ$ determine the Cauchy data set, at least under a strict convexity assumption. The result is stated for $C^{\infty}$ Cauchy data sets since we use a regularity result from \cite{PestovUhlmann, PSU2} stated in the $C^{\infty}$ case. Note that any first order differential operator with real principal symbol in $M$ is of the form $P = \frac{1}{i} L + V$, where $L$ is a real vector field in $M$ and $V \in C^{\infty}(M)$. Let $X$ be an open manifold containing $M$, and extend $L$ smoothly to $X$. We say that $M$ is \emph{strictly convex for $L$} if 
\[
\frac{d^2}{dt^2} (\rho(x(t)) \bigg|_{t=0} < 0
\]
whenever $\rho \in C^{\infty}(X)$ is a boundary defining function for $M$ (i.e.\ $M = \{ \rho \geq 0 \}$ and $d\rho \neq 0$ on $\p M$), and $x(t)$ is an integral curve of $L$ so that $x(0) \in \p M$ and $\dot{x}(0) \in T(\p M)$. A local coordinate computation shows that this condition does not depend on the extension or on the choice of $\rho$. The proof of the following theorem is in Appendix \ref{sec_appendix}.

\begin{Theorem} \label{thm_first_order_equivalence}
Let $M$ be compact with smooth boundary, and let $P_j = \frac{1}{i} L_j + V_j$ be real principal type differential operators of order $1$ in $M$. Assume that is $M$ is strictly convex for $L_1$ and $L_2$. Suppose that 
\[
\alpha_{P_1} = \alpha_{P_2}, \qquad \exp \left[ -i \int_0^{T_1} V_1(x_1(t)) \,dt \right] = \exp \left[ -i \int_0^{T_2} V_2(x_2(t)) \,dt \right]
\]
whenever $x_j: [0,T_j] \to M$ are maximal integral curves of $L_j$ with $x_1(0) = x_2(0)$. Then the $C^{\infty}$ Cauchy data sets of $P_1$ and $P_2$ are equal:
\[
\{ u|_{\p M} \,;\, u \in C^{\infty}(M), \ P_1 u = 0 \text{ in $M$} \} = \{ u|_{\p M} \,;\, u \in C^{\infty}(M), \ P_2 u = 0 \text{ in $M$} \}.
\]
\end{Theorem}

\section{Boundary determination} \label{sec_boundary_determination}

In this section we assume that $M$ is a compact manifold with smooth boundary and that $P$ is a real principal type differential operator of order $m \geq 2$. We consider the problem of determining the boundary values of coefficients of $P$ and their derivatives (possibly up to gauge) from the knowledge of $C_P$. For $(x_0, \xi_0) \in T^*(\p M)$, we give two arguments for determining boundary values at $x_0$ (these arguments were already described in the introduction, recall also that $\nu$ is the \emph{inner} unit conormal vector to $\p M$ with respect to the auxiliary Riemannian metric on $M$):

\begin{itemize}
\item 
(Elliptic region) If $t \mapsto p_m(x_0,\xi_0 + t\nu)$ has a simple non-real root, we use exponentially decaying solutions that concentrate near $x_0$ to give an analogue of boundary determination results for second order elliptic equations.
\item 
(Hyperbolic region) If $t \mapsto p_m(x_0,\xi_0 + t\nu)$ has at least two distinct real roots, we use solutions concentrating near two null bicharacteristics through $x_0$ and obtain an analogue of boundary determination results for the wave equation.
\end{itemize}

In fact the regions could be mixed, and we will use a combination of both methods. Here are some examples:

\begin{itemize}
\item 
If $P = \p_t^2 - \Delta_{g_0}$ is the wave operator in $M := M_0 \times (0,T)$, then the elliptic region (resp.\ hyperbolic region) at $(x_0, t_0) \in \p M_0 \times (0,T)$ is the set of spacelike covectors (resp.\ timelike covectors) in $T_{(x_0,t_0)}^*(\p M)$.
\item 
If $P$ is elliptic, then all roots of $t \mapsto p_m(x_0,\xi_0+t\nu)$ are non-real.
\item 
If $m$ is odd, the map $t \mapsto p_m(x_0,\xi_0+t\nu)$ always has a real root unless $x_0$ is a characteristic boundary point.
\item 
If $P$ is strictly hyperbolic with respect to the level surfaces of $\phi \in C^{\infty}(X, \mR)$ (see Section \ref{sec_introduction}), and if $x_0$ in an interior point of $\p M \cap \{ \phi = c \}$ for some $c$, then by definition $t \mapsto p_m(x_0,\xi_0+t\nu)$ has $m$ distinct real roots.
\item 
If $P = (\p_t^2 - \Delta_{g_0})(\p_t^2 + \Delta_{g_0})$ in $M_0 \times (0,T)$, there can be both real and non-real roots.
\end{itemize}

The special solutions in the hyperbolic region will be the ones constructed in Theorem \ref{thm_quasimode_direct}. In the elliptic region, new special solutions will be required. Their construction is based on the fact that whenever $x \in \p M$ and $t \mapsto p_m(x,\xi+t\nu)$ has a simple non-real root, there is an approximate solution of $Pu = 0$ in $M$ with complex phase that concentrates near $x$, decays exponentially in the interior, and whose boundary value oscillates in direction $\xi$. The argument is similar to the construction of a boundary parametrix for the wave equation in the elliptic region, see e.g.\ \cite[Chapter IX]{Taylor1981}.

\begin{Theorem} \label{prop_exponentially_localized_solution}
Let $P$ be a differential operator of order $m \geq 2$ in $M$ (here the principal symbol $p_m$ is allowed to be complex valued). Let $x \in \p M$, and assume that for some $\xi \in T_x^* \p M$ the map $z \mapsto p_m(x, \xi + z \nu)$ has a simple root $z$ with positive imaginary part. Fix $s > 0$. For $0 < h \leq 1$ there exists $u = u_{h} \in C^{\infty}(M)$ having the form 
\[
u = e^{i\Phi/h} a
\]
where $d\Phi|_x = \xi + z\nu$, $a$ is supported near $x$, $a|_{\p M}$ can be prescribed arbitrarily near $x$ independently of $h$, $\mathrm{Im}(\Phi) \geq 0$ on $\mathrm{supp}(a)$, and 
\[
\norm{u}_{L^2(M)} \sim h^{1/2}, \qquad \norm{u}_{H^k(M)} \lesssim h^{-k+1/2}, \qquad \norm{Pu}_{H^s(M)} = O(h^{\infty}).
\]
\end{Theorem}
\begin{proof}
Choose local coordinates near $x$ so that $x$ corresponds to $0$, $M$ is given by $\{ x_n \geq 0 \}$ near $0$, the unit conormal $\nu$ corresponds to $(0,1)$, and the cotangent vector $\xi$ corresponds to $(\xi', 0)$ with $\xi' \in \mR^{n-1}$. Write $u = e^{i\Phi/h} a$. By Lemma \ref{lemma_conjugated_p}, we have 
\[
Pu = e^{i\Phi/h} \left[ h^{-m} p_m(x, d\Phi) a + h^{1-m}(\frac{1}{i} La + ba) + \sum_{j=2}^m h^{j-m} R_j a \right].
\]
We first construct a smooth complex function $\Phi$ so that near $0$, 
\begin{align} \label{phi_localized_conditions}
\left\{ \begin{array}{rcl}
\Phi(x',0) \!\!\! &=& \!\!\! x' \cdot \xi', \\[3pt]
\mathrm{Im}(\p_n \Phi(x',0)) \!\!\! &>& \!\!\! 0, \\[3pt]
p_m(x,d\Phi) \!\!\! &=& \!\!\! 0 \text{ to infinite order on $\{ x_n = 0 \}$.}
\end{array} \right.
\end{align}
In fact, if we define $\Phi(x',0) = x' \cdot \xi'$, then 
\[
p_m(x,d\Phi)|_{x_n = 0} = p_m(x',0,\xi',\p_n \Phi(x',0)).
\]
This vanishes if $\p_n \Phi(x',0)$ is a root of $z \mapsto p_m(x',0, \xi', z)$. By assumption there exists a simple root with positive imaginary part when $x'=0$, hence also for $x'$ near $0$, and we denote this root by $z(x')$. Then $z(x')$ depends smoothly on $x'$ since it is a simple root, and we may define 
\begin{equation} \label{phi_localized_normal_derivative}
\p_n \Phi(x',0) = z(x').
\end{equation}
Next we compute 
\[
\p_{x_n}(p_m(x,d\Phi))|_{x_n = 0} = \p_{x_n} p_m(x',0,\xi',z(x')) + \p_{\xi_j} p_m(x',0,\xi',z(x')) \p_{jn} \Phi(x',0).
\]
We wish to choose $\p_n^2 \Phi(x',0)$ so that this vanishes. However, writing 
\[
p_m(x',0,\xi',\xi_n) = (\xi_n - z(x')) q(x',\xi',\xi_n)
\]
where $q(x',\xi',z(x')) \neq 0$, we see that $\p_{\xi_n} p_m(x',0,\xi',z(x')) \neq 0$, and hence we may define 
\begin{equation*}
\p_n^2 \Phi(x',0) = -\frac{1}{\p_{\xi_n} p_m} ( \p_{x_n} p_m + \sum_{\alpha=1}^{n-1} \p_{\xi_{\alpha}} p_m \p_{\alpha n} \Phi)\Big|_{(x',0,\xi',z(x'))}.
\end{equation*}
With this choice we have $\p_{x_n}(p_m(x,d\Phi))|_{x_n = 0} = 0$. Continuing this process allows us to prescribe $\p_n^3 \Phi(x',0)$, $\p_n^4 \Phi(x',0)$, $\ldots$ so that $\p_{x_n}^2(p_m(x,d\Phi))$, $\p_{x_n}^3(p_m(x,d\Phi))$, $\ldots$ vanish on $\{ x_n = 0 \}$. Using Borel summation we obtain a smooth function $\Phi$ with the required Taylor series at $x_n = 0$, so that $\Phi$ satisfies \eqref{phi_localized_conditions}.

We now construct the amplitude $a$ by Borel summation in the form 
\[
a \sim \sum_{j=0}^{\infty} h^{j} a_{j}
\]
where the $a_{k}$ are smooth functions independent of $h$, and they satisfy the following transport equations to infinite order on $\{ x_n = 0 \}$:
\begin{align*}
\frac{1}{i} La_0 + ba_0 &= 0, \\
\frac{1}{i} La_{1} + ba_{1} &= -R_2 a_0, \\
\frac{1}{i} La_{2} + ba_{2} &= -R_3 a_0 - R_2 a_{1}, \\
 & \ \,\vdots \notag
\end{align*}
The first equation on $\{ x_n = 0 \}$ reads as 
\[
\frac{1}{i} \p_{\xi_j} p_m(x',0,\xi',z(x')) \p_j a_0(x',0) + b(x',0) a_0(x',0) = 0.
\]
We define $a_0(x',0) = \eta(x')$ where $\eta$ is any given function in $C^{\infty}_c(\mR^{n-1})$ supported close enough to $0$. We have seen that $\p_{\xi_n} p_m(x',0,\xi',z(x')) \neq 0$, hence we may solve $\p_n a(x',0)$ from the above equation as 
\[
\p_n a_0(x',0) = -\frac{i}{\p_{\xi_n} p_m} (\frac{1}{i} \sum_{\alpha=1}^{n-1} \p_{\xi_{\alpha}} p_m \p_{\alpha} a_0 + b a_0) \big|_{(x',0,\xi',z(x'))}.
\]
Continuing in this way, we may define $\p_n^j a_0(x',0)$ for $j \geq 2$ and apply Borel summation to obtain a smooth function $a_0$ with the required Taylor series at $\{ x_n = 0 \}$, so that $\frac{1}{i} La_0 + ba_0 = 0$ to infinite order on $\{ x_n = 0 \}$. The functions $a_{j}$ for $j \geq 1$ are constructed in a similar way (one may set $a_{j}(x',0) = 0$ for $j \geq 1$), so that all the transport equations are satisfied to infinite order on $\{ x_n = 0 \}$. This completes the construction of the amplitude.

We have now constructed a smooth function 
\[
u = e^{i\Phi/h} a
\]
so that 
\[
Pu \sim e^{i\Phi/h} \sum_{j=0}^{\infty} h^{j-m} f_j
\]
where each $f_j$ vanishes to infinite order on $\{ x_n = 0 \}$ and $\abs{f_j} \lesssim 1$ uniformly over $h \leq 1$. Near $0$ one has 
\[
c x_n \leq \mathrm{Im}(\Phi(x',x_n)) \leq C x_n
\]
for some $c, C > 0$. Thus if $\supp(a)$ is chosen to be sufficiently close to $0$, one has 
\begin{align*}
\norm{u}_{L^2(M)}^2 &\sim \int_{ \{ \abs{x'} \leq 1 \}} \int_0^1 e^{-x_n/h} \,dx_n \,dx' \sim h, \\
\norm{u}_{H^k(M)}^2 &\lesssim h^{-2k} \int_{ \{ \abs{x'} \leq 1 \}} \int_0^1 e^{-x_n/h} \,dx_n \,dx' \sim h^{1-2k}, \\
\end{align*}
and 
\[
\norm{Pu}_{H^N(M)} \lesssim h^{-m-N} \norm{e^{-c x_n/h} O(x_n^{\infty})}_{L^2(\supp(a))} = O(h^{\infty}).
\]
The result follows.
\end{proof}

The next example gives a useful interpretation for a boundary determination result for the wave equation. This will motivate the corresponding proof for general real principal type operators.

\begin{Example}
Let $(M_0,g_0)$ be compact with smooth boundary, let $V \in C^{\infty}(M_0 \times (0,T))$, and consider the Dirichlet problem for the wave equation 
\[
(\p_t^2 - \Delta_{g_0} + V) u = 0 \text{ in $M_0 \times (0,T)$}, \qquad u|_{\p M_0 \times (0,T)} = f, \qquad u|_{t=0} = \p_t u|_{t=0} = 0.
\]
The Dirichlet-to-Neumann map is given by 
\[
\Lambda_V: f \mapsto \p_{\nu} u|_{\p M_0 \times (0,T)}.
\]
The boundary determination result in \cite{StefanovYang} shows that for a small neighborhood $\Gamma \subset \p M_0 \times (0,T)$ of a boundary point $(x_0, t_0)$, the localized map $\Lambda_V^{\Gamma}: C^{\infty}_c(\Gamma) \to C^{\infty}(\Gamma), \ f \mapsto \Lambda_V f|_{\Gamma}$ is a pseudodifferential operator and the Taylor series of $V$ at $(x_0,t_0)$ can be computed from the symbol of $\Lambda_V^{\Gamma}$.

Consider coordinates $(x_1,x',t)$ and assume that $\Gamma$ corresponds to $\{x_1 = 0\}$. The symbol of $\Lambda_V^{\Gamma}$ can be computed by testing against highly oscillatory functions $f(x',t) = \chi(x',t) e^{i(x'\cdot\xi' + t)/h}$ where $\chi$ is a smooth cutoff and $\xi'$ is in the hyperbolic region (i.e.\ there are two null directions at $(x_0,t_0)$ whose projection to $T^* (\p M_0 \times (0,T))$ is $(0,\xi',1)$). The argument in \cite{StefanovYang} gives roughly (when the cutoff $\chi$ is made to depend on $h$ in a suitable way) that 
\[
\lim_{h \to 0} h^{\alpha} ( (\Lambda_V-\Lambda_0) f, f)_{L^2(\Gamma)} = c V(x_0,t_0)
\]
where $c \neq 0$ and $\alpha$ is a suitable number depending on $n$.

On the other hand, one has the integral identity 
\[
( (\Lambda_V-\Lambda_0) f, f)_{L^2(\Gamma)} = \int_{M_0} \int_0^T V u \bar{u}_0 \,dt \,dV
\]
where $u$ is the solution given above, and $u_0$ solves $(\p_t^2 - \Delta_{g_0}) u_0 = 0$ in $M_0 \times (0,T)$ with Dirichlet data $f$ and vanishing Cauchy data on $\{t=T\}$. Since $u$ has vanishing Cauchy data on $\{t=0\}$ and highly oscillatory boundary data, it is related to a null bicharacteristic (possibly with reflections) that starts at $(x_0,t_0)$ and moves forward in time. Similarly, since $u_0$ has vanishing Cauchy data on $\{t=T\}$, it is related to a null bicharacteristic starting at $(x_0,t_0)$ and moving backward in time. Thus the product $u \bar{u}_0$ concentrates near the intersection of the projections of these two bicharacteristics, i.e. near the point $(x_0,t_0)$. Analyzing the integral over $M_0 \times (0,T)$ when $h \to 0$ then leads to a proof for recovering $V(x_0,t_0)$.
\end{Example}

\begin{proof}[Proof of Theorem \ref{thm_boundary_determination_principal_intro}]
Let $P_1, P_2$ be real principal type differential operators of order $m \geq 2$. Assume that  
\[
C_{P_1} = C_{P_2}.
\]
The integral identity in Lemma \ref{lemma_integral_identity_potential} implies that 
\begin{equation} \label{integral_identity_boundary_determination_first}
((P_1-P_2) u_1, u_2)_{L^2(M)} = 0
\end{equation}
whenever $u_j \in H^m(M)$ are solutions of $P_1 u_1 = 0$ and $P_2^* u_2 = 0$ in $M$. Depending on the roots of the characteristic polynomial, we will use different special solutions $u_j$ to prove the theorem. The  strategy however will be the same: the main point is to prove that $p_{m,2}(x,\xi+\cdot\,\nu)$ vanishes at the roots of $p_{m,1}(x,\xi+\cdot\,\nu)$ whenever $(x,\xi) \in T^*(\p M)$ is close to $(x_0,\xi_0)$, i.e.\ that 
\begin{equation} \label{pm2_tau_roots}
\left\{ \begin{array}{c} \text{$p_{m,2}(x,\xi+\tau \nu) = 0$ for any $\tau \in \mC$ for which $p_{m,1}(x,\xi+\tau \nu) = 0$,} \\[3pt]
\text{whenever $(x,\xi) \in T^*(\p M)$ is close to $(x_0,\xi_0)$.} \end{array} \right.
\end{equation}
Here and below, we will use the fact that the assumptions of the theorem remain true for any $(x,\xi) \in T^* (\p M)$ that is close enough to $(x_0,\xi_0)$.

Let us show how \eqref{pm2_tau_roots} implies the claim \eqref{boundary_determination_principal_claim} for $j=0$ (the case $j \geq 1$ will be done later). By \eqref{pm2_tau_roots} the polynomial $\tau \mapsto p_{m,2}(x,\xi+\tau \nu)$ vanishes at each of the $m$ distinct roots of $\tau \mapsto p_{m,1}(x,\xi+\tau \nu)$. The fact that $x$ is a noncharacteristic boundary point for $P_j$ implies that both maps are polynomials of degree exactly $m$ in $\tau$. Thus one must have, for some $c(x,\xi) \neq 0$, that  
\[
p_{m,2}(x,\xi+\tau \nu) = c(x,\xi) p_{m,1}(x,\xi+\tau \nu), \qquad \tau \in \mR.
\]
Multiplying by $\tau^{-m}$ and letting $\tau \to \infty$ gives that 
\[
c(x,\xi) = \frac{p_{m,2}(x,\nu_x)}{p_{m,1}(x,\nu_x)}.
\]
Thus in fact $c(x,\xi) = c(x)$. For $(x,\xi) \in T^*(\p M)$ close to $(x_0,\xi_0)$ we obtain 
\[
p_{m,2}(x,\xi+\tau \nu_x) = c(x) p_{m,1}(x,\xi+\tau \nu_x), \qquad \tau \in \mR.
\]
For any $x \in \p M$ near $x_0$, the set $\{ \xi + \tau \nu_x \,;\, \text{$\tau \in \mR$ and $(x,\xi)$ is near $(x_0,\xi_0)$} \}$ is open in $T_x^* M$. By real-analyticity in the fiber variable one then has 
\[
p_{m,2}(x,\xi) = c(x) p_{m,1}(x,\xi), \qquad \text{$x \in \p M$ near $x_0$, \ $\xi \in T_x^* M$},
\]
where $c$ is a smooth nonvanishing function near $x_0$ on $\p M$  given by 
\[
c(x) = \frac{p_{m,2}(x,\nu_x)}{p_{m,1}(x,\nu_x)}.
\]
This proves the desired claim \eqref{boundary_determination_principal_claim} for $j=0$.

The proof of \eqref{pm2_tau_roots} will be divided in several cases.

\vspace{10pt}

\noindent {\it Case (1).} We first assume condition (1) in the theorem, i.e.\ that $p_{m,2}(x,\xi + \sigma \nu) = 0$ for some non-real $\sigma \in \mC$ (if this holds for $p_{m,1}$ instead we obtain the result by interchanging $P_1$ and $P_2$; note that the conclusion of the theorem remains invariant under this change). Since $P_2$ has real principal symbol, the non-real roots come in complex conjugate pairs and hence we may assume that $\im(\sigma) > 0$.

Using Theorem \ref{prop_exponentially_localized_solution} for $P_2^*$, there are smooth functions 
\[
v_2 = v_{2,h} = h^{-1/2} e^{i\Phi_2/h} a_2
\]
where $d\Phi_2|_x = \xi + \sigma \nu$, $a_2|_{\p M}$ is independent of $h$ with $a_2(x) \neq 0$, $\norm{v_2}_{L^2} \sim 1$, and for any fixed but suitably large $s$ 
\[
\norm{P_2 v_2}_{H^s(M)} = O(h^{\infty}).
\]
By the solvability result for real principal type operators in Proposition \ref{prop_real_principal_type_solvability}(d), there exist smooth solutions $u_2 = u_{2,h} = v_2 + r_2$ of $P_2 u_2 = 0$ in $M$ such that 
\[
\norm{r_2}_{H^{s+m-1}(M)} = O(h^{\infty}).
\]

Let $\{ \tau_1, \ldots, \tau_m \}$ be the complex roots of $\tau \mapsto p_{m,1}(x,\xi+\tau \nu)$. By assumption these are all distinct. Let $\tau$ be one of these roots. We wish to prove that $p_{m,2}(x,\xi+\tau \nu) = 0$, which establishes \eqref{pm2_tau_roots}. We will consider two subcases.

\vspace{10pt}

\noindent {\it Subcase (1a).} 
Suppose that $\tau$ is non-real. We may assume that $\im(\tau) > 0$ (since if $\im(\tau) < 0$, then $\bar{\tau}$ is also a root and the argument below shows that $p_{m,2}(x,\xi+\bar{\tau} \nu) = 0$, and hence $p_{m,2}(x,\xi+\tau \nu) = 0$ since the roots come in complex conjugate pairs).

Using Theorem \ref{prop_exponentially_localized_solution} for $P_1$, we obtain solutions $u_1 = u_{1,h}$ of $P_1 u_1 = 0$ in $M$ of the form $u_1 = v_1 + r_1$, and 
\[
v_1 = h^{-1/2} e^{i\Phi_1/h} a_1
\]
where $d\Phi_1|_x = \xi + \tau \nu$, $a_1|_{\p M}$ is independent of $h$ with $a_1(x) \neq 0$, $\norm{v_1}_{L^2} \sim 1$, and 
\[
\norm{r_1}_{H^{s+m-1}(M)} = O(h^{\infty}).
\]
Inserting the solutions $u_1$ and $u_2$ into \eqref{integral_identity_boundary_determination_first}, we obtain  that 
\[
\int_M (P_1-P_2) (v_1) \bar{v}_2 \,dV = O(h^{\infty}).
\]
The explicit form of $v_1$ and $v_2$ together with Lemma \ref{lemma_conjugated_p} gives that 
\begin{align}
0 &= \lim_{h \to 0} h^{m} \int_M (P_1-P_2)(v_1) \bar{v}_2 \,dV \notag \\
 &= \lim_{h \to 0} h^{-1} \int_M (p_{m,1}(x,d\Phi_1)-p_{m,2}(x,d\Phi_1)) e^{i (\Phi_1-\bar{\Phi}_2)/h} a_1 \bar{a}_2 \,dV. \label{subcaseonea_first_identity}
\end{align}
Write $q(x) := p_{m,1}(x,d\Phi_1(x))-p_{m,2}(x,d\Phi_1(x))$ and $\Theta := \Phi_1-\bar{\Phi}_2$. By the construction of the functions $\Phi_j$, see \eqref{phi_localized_conditions}--\eqref{phi_localized_normal_derivative}, in boundary normal coordinates (for $g$) near $0$ one has 
\[
\Theta(x',x_n) =(\tau-\bar{\sigma}) x_n + O(x_n^2).
\]
Note that $\tau$ and $\sigma$ depend on $x'$, and $\im(\tau-\bar{\sigma}) > 0$. If the support of $a_j$ is chosen small enough, one has as $h \to 0$ 
\begin{align*}
 &h^{-1} \int_M e^{i\Theta/h} q a_1 \bar{a}_2 \,dV \\
 &= h^{-1} \int_{\{ x_n > 0 \}} e^{i(\tau-\bar{\sigma}) x_n/h + O(x_n^2)/h} q a_1 \bar{a}_2 \abs{g}^{1/2} \,dx \\
 &= \int_{\{ x_n > 0 \}}  e^{i(\tau-\bar{\sigma})x_n + h O(x_n^2)} (q a_1 \bar{a}_2 \abs{g}^{1/2})(x',h x_n) \,dx \\
 &\to \int_{\{ x_n > 0 \}}e^{i(\tau-\bar{\sigma})x_n}  (q a_1 \bar{a}_2 \abs{g}^{1/2})(x',0) \,dx \\
 &= \int_{\mR^{n-1}} \frac{i}{\tau-\bar{\sigma}} (q a_1 \bar{a}_2 \abs{g}^{1/2})(x',0) \,dx'.
\end{align*}
In the last step we used that $\int_0^{\infty} e^{iwx_n} \,dx_n = \frac{i}{w}$ when $\im(w) > 0$. Since $a_j(x',0)$ can be chosen arbitrarily near $0$, we obtain that $q(x',0) = 0$ near $0$. This means that 
\[
p_{m,1}(x',0,\xi',\tau(x')) = p_{m,2}(x',0,\xi',\tau(x')).
\]
But $\tau(x')$ was a root of $t \mapsto p_{m,1}(x',0,\xi',t)$, so it is also a root for $p_{m,2}$, i.e.\ $p_{m,2}(x,\xi+\tau \nu) = 0$. This proves \eqref{pm2_tau_roots} whenever $\tau$ is non-real.

\vspace{10pt}

\noindent {\it Subcase (1b).} 
Now suppose that $\tau$ is real. Let $\gamma: [0,T] \to T^*M$ be the maximal null $P_1$-bicharacteristic curve with $\gamma(0) = (x_0,\xi_0+\tau \nu_{x_0})$. Extend $M$ to a slightly larger open manifold $X$, extend $P_1$ smoothly as a real principal type operator to $X$, and let $\breve{\gamma}(t) = (x(t), \xi(t))$ be the extension of $\gamma$. Choose $t_- < 0$ and $t_+ > T$ so that $\breve{\gamma}(t_{\pm})$ are outside $M$, and let $v_1$ be the quasimode in Theorem \ref{thm_quasimode_direct} associated with $\breve{\gamma}|_{[t_-,t_+]}$ with $\norm{P_1 v_1}_{H^s(M)} = O(h^{\infty})$ for some fixed large $s > 0$. Use Proposition \ref{prop_real_principal_type_solvability}(d) to find $r_1$ with $\norm{r_1}_{H^{s+m-1}(M)} = O(h^{\infty})$ so that $u_1 = v_1 + r_1$ solves $P_1 u_1 = 0$ in $M$. Recall that $v_1$ has the form 
\[
v_1(x) = \int_{t_-}^{t_+} e^{i\Phi_1(x,t)/h} a_1(x,t) \,dt
\]
Recall also from \eqref{phi_req1_second}--\eqref{dtphi_formula} that 
\begin{equation} \label{phione_first_facts}
\Phi_1(x(t), t) = 0, \qquad d_x \Phi_1(x(t),t) = \xi(t), \qquad \p_t \Phi_1(x(t),t) = 0,
\end{equation}
and 
\begin{equation} \label{phione_second_fact}
\im(\Phi_1(x,t)) \geq c d(x,x(t))^2, \qquad (x,t) \in \supp(a_1).
\end{equation}
As discussed in \eqref{phi_req1_second}--\eqref{dttphi_formula}, the quantity $\I = \im((\p_{x_j x_k} \Phi_1))|_{(x_0,0)}$ is invariantly defined and we may write 
\begin{equation} \label{phione_third_fact}
\nabla_{x,t}^2 \im(\Phi_1)|_{(x_0,0)} = \left( \begin{array}{cc} \I & - \I \dot{x} \\ -(\I \dot{x})^t & \I \dot{x} \cdot \dot{x} \end{array} \right)
\end{equation}

Inserting the solutions $u_1$ and $u_2$ in \eqref{integral_identity_boundary_determination_first}, we obtain that 
\[
\int_M ((P_1-P_2)v_1) \bar{v}_2 \,dV = O(h^{\infty}).
\]
Inserting the expressions for $v_1$ and $v_2$ and using Lemma \ref{lemma_conjugated_p}, it follows that 
\begin{multline*}
\int_M \int_{t-}^{t_+} (p_{m,1}(x,d_x \Phi_1(x,t)) - p_{m,2}(x,d_x \Phi_1(x,t))) e^{i(\Phi_1(x,t)-\bar{\Phi}_2(x))/h} a_1(x,t) \bar{a}_2(x) \,dt \,dV \\ + \text{lower order terms} = O(h^{\infty}).
\end{multline*}
Let $x'$ be normal coordinates at $x_0$ on $\p M$ (for the metric induced by $g$), and let $(x',x_n)$ be corresponding boundary normal coordinates so that $x_0$ corresponds to $0$. Since $a_2$ is supported near $x_0$, the $M$-integral can be written in these coordinates and we have 
\begin{equation} \label{boundary_subcaseoneb_integral}
\int_{\{x_n > 0\}} \int_{t_-}^{t_+}  e^{i \Theta(x,t)/h} q(x,t) \,dt \,dx + \text{lower order terms} = O(h^{\infty})
\end{equation}
where $q(x,t) = (p_{m,1}(x,d_x \Phi_1(x,t)) - p_{m,2}(x,d_x \Phi_1(x,t))) a_1(x,t) \bar{a}_2(x) \abs{g(x)}^{1/2}$ and the phase $\Theta$ is given by $\Theta(x,t) = \Phi_1(x,t)-\bar{\Phi}_2(x)$.

We wish to use stationary phase to show that the main contribution to the integral \eqref{boundary_subcaseoneb_integral} comes from the region near $t=0$. Note first that 
\[
d_x \Theta(0,0) = \xi(0) - \xi_0 - \bar{\sigma} \nu = (\tau-\bar{\sigma})\nu, \qquad \p_t \Theta(0,0) = 0.
\]
We next study the Hessian of $\im(\Theta)$ in the $(x',x_n,t)$ coordinates. By \eqref{phione_third_fact} we have 
\[
\nabla_{x,t}^2 \im(\Theta)|_{(0,0)} = \left( \begin{array}{cc} \I & - \I \dot{x} \\ -(\I \dot{x})^t & \I \dot{x} \cdot \dot{x} \end{array} \right) + \nabla_{x,t}^2 \im(\Phi_2)|_{(0,0)}
\]
By \eqref{phi_localized_conditions} we have $\nabla_{x',t}^2 \im(\Phi_2)|_{(0,0)} = 0$. If $\zeta' = (v', z)^t$ with $v' \in \mC^{n-1}$, and if $\zeta = (v',0,z)^t$, it follows that 
\begin{equation} \label{nablaxprimet_im_theta_computation}
\nabla_{x',t}^2 \im(\Theta)|_{(0,0)} \zeta' \cdot \bar{\zeta}' = \left( \begin{array}{cc} \I & - \I \dot{x} \\ -(\I \dot{x})^t & \I \dot{x} \cdot \dot{x} \end{array} \right) \zeta \cdot \bar{\zeta} = \abs{\I^{1/2} (\left( \begin{array}{c} v' \\ 0 \end{array} \right) - \dot{x} z)}^2.
\end{equation}
If the right hand side vanishes, then $\dot{x}_n z = 0$ and $v' = \dot{x}' z$. But the fact that $\tau$ is a simple root of $t \mapsto p_{m,1}(x_0,\xi_0+t\nu)$ implies that 
\begin{equation} \label{xndot_simple_condition}
\dot{x}_n(0) = \p_{\xi_n} p_{m,1}(x_0,\xi_0+\tau \nu) \neq 0.
\end{equation}
Thus we get $z = 0$ and $v' = 0$. This proves that $\nabla_{x',t}^2 \Theta|_{(0,0)}$ is invertible.

If $\chi(t)$ is a smooth cutoff supported near $0$ with $\chi = 1$ near $0$, the above discussion together with stationary phase on manifolds with boundary \cite[Theorem 7.7.17(ii)]{Hormander} implies that 
\[
\int_{\{x_n > 0\}} \int_{t_-}^{t_+}  e^{i \Theta(x,t)/h} q(x,t) \chi(t) \,dt \,dx = c_0 h^{\frac{n+2}{2}} q(0,0) + O(h^{\frac{n+4}{2}}), \qquad c_0 \neq 0.
\]
If we replace $\chi(t)$ by $1-\chi(t)$ in the integrand, then the fact that the $x$-integral is over a small neighborhood of $x_0$ and the assumption that $\gamma$ never returns to $x_0$ after $t=0$ implies that 
the corresponding integral is over the region where $x(t)$ is bounded away from $x_0$, and hence the corresponding integral is $O(e^{-C/h})$. Multiplying \eqref{boundary_subcaseoneb_integral} by $h^{-\frac{n+2}{2}}$ and letting $h \to 0$ implies that $q(0,0) = 0$, i.e. 
\[
p_{m,1}(x,d_x \Phi_1(x,t)) - p_{m,2}(x,d_x \Phi_1(x,t))|_{(0,0)} = p_{m,1}(x_0,\xi_0+\tau \nu) - p_{m,2}(x_0,\xi_0+\tau \nu) = 0.
\]
Since $p_{m,1}(x_0,\xi_0+\tau \nu) = 0$, we obtain $p_{m,2}(x_0,\xi_0+\tau \nu) = 0$. This also holds for $(x,\xi) \in T^*(\p M)$ close to $(x_0,\xi_0)$, which proves \eqref{pm2_tau_roots} for real $\tau$.

\vspace{10pt}

\noindent {\it Case (2).} We now assume condition (2) in the theorem, i.e.\ that all roots for $\tau \mapsto p_{m,j}(x_0,\xi_0+\tau \nu)$ are real and simple and have the stated properties. Let $\sigma$ be such that $p_{m,2}(x_0,\xi_0+\sigma \nu) = 0$, and let $\tau$ be any root of $t \mapsto p_{m,1}(x_0,\xi_0+t\nu) = 0$. We want to prove that $p_{m,2}(x_0,\xi_0+\tau \nu) = 0$, which would imply \eqref{pm2_tau_roots}.

We again use the integral identity \eqref{integral_identity_boundary_determination_first} with suitable special solutions $u_j$. Let $\gamma_j: [0,T_j] \to T^* M$ be the maximal null $P_j$-bicharacteristic curves with $\gamma_1(0) = (x_0,\xi_0+\tau \nu)$ and $\gamma_2(0) = (x_0,\xi_0+\sigma \nu)$, and let $\breve{\gamma}_j(t) = (x_j(t), \xi_j(t))$ be extensions to a slightly larger manifold $X$ so that the end points of $\breve{\gamma}_j|_{[t_j^-, t_j^+]}$ are outside $M$. As in Subcase (1b) above, let $u_j = v_j + r_j$ solve $P_1 u_1 = 0$ and $P_2^* u_2 = 0$ in $M$, where $v_j$ is the quasimode provided by Theorem \ref{thm_quasimode_direct} associated with $\breve{\gamma}_j|_{[t_j^-,t_j^+]}$ and $\norm{r_j}_{H^{s+m-1}(M)} = O(h^{\infty})$ for some fixed large $s$. Then $v_j$ has the form 
\[
v_j(x) = \int_{t_j^-}^{t_j^+} e^{i\Phi_j(x,t)/h} a_j(x,t) \,dt
\]
where the phase functions $\Phi_j$ satisfy the counterparts of \eqref{phione_first_facts}--\eqref{phione_third_fact} with $x(t)$, $\xi(t),$ and $\I$ replaced by $x_j(t)$, $\xi_j(t)$, and $\I_j$. Moreover, $a_j$ is supported in a small neighborhood of the curve $(x_j(t),t)$.

Inserting the solutions $u_j$ in \eqref{integral_identity_boundary_determination_first}, we obtain 
\begin{equation} \label{boundary_subcasetwo_integral}
\int_M \int_{t_1^-}^{t_1^+} \int_{t_2^-}^{t_2^+} e^{i\Theta(x,t,s)/h} q(x,t,s) \,ds \,dt \,dV(x) + \text{lower order terms} = O(h^{\infty})
\end{equation}
where $\Theta(x,t,s) = \Phi_1(x,t) - \ol{\Phi_2(x,s)}$ and 
\[
q(x,t,s) =  (p_{m,1}(x,d_x \Phi_1(x,t)) - p_{m,2}(x,d_x \Phi_1(x,t))) a_1(x,t) \ol{a_2(x,s)}.
\]
The integral in \eqref{boundary_subcasetwo_integral} is over a fixed small neighborhood of the set 
\[
E := \{ (\bar{x},t,s) \,;\, \bar{x} \in M, \ x_1(t) = x_2(s) = \bar{x} \}.
\]
Again, we want to use stationary phase to show that the main contribution comes from $t=s=0$. First note that if $x_1(t) = x_2(s) = \bar{x}$, then 
\[
\Theta(\bar{x},t,s) = 0, \qquad d_x \Theta(\bar{x},t,s) = \xi_1(t) - \xi_2(s), \qquad \p_{t,s} \Theta(\bar{x},t,s) = 0.
\]
By the assumption that the bicharacteristics intersect nicely, one has $E = \{ (x_0,0,0) \} \cup K$ where $K$ is a compact subset of $M^{\mathrm{int}} \times [\eps,T_1] \times [\eps,T_2]$ for some $\eps > 0$, and in $K$ one always has $\xi_1(t) \neq \xi_2(s)$. This shows that when $(t,s)$ is away from $(0,0)$, and if the supports of $a_j$ were chosen small enough, the integral in \eqref{boundary_subcasetwo_integral} is $O(h^{\infty})$ by nonstationary phase.

It remains to evaluate the integral 
\[
\int_M \int_{\mR^2} e^{i\Theta(x,t,s)/h} q(x,t,s) \chi(t,s) \,ds \,dt \,dV(x)
\]
where $\chi(t,s)$ is a cutoff function with small support and with $\chi = 1$ near $(0,0)$. Note that the $M$-integral is over a small neighborhood of $x_0$. Now 
\[
\Theta(x_0,0,0) = 0, \qquad d_x \Theta(x_0,0,0) = \xi_1(0) - \xi_2(0) = (\tau-\sigma) \nu, \qquad \p_{t,s} \Theta(x_0,0,0) = 0.
\]
Let $x'$ be normal coordinates on $\p M$ near $x_0$, and let $(x',x_n)$ be corresponding boundary normal coordinates so that $x_0$ corresponds to $0$. Let $v' \in \mC^{n-1}$ and write $\zeta' = (v', z, w)^t$ and $\zeta = (v', 0, z, w)^t$. Arguing as in \eqref{nablaxprimet_im_theta_computation}, we compute 
\begin{align*}
\nabla_{x',t,s}^2 \im(\Theta)|_{(0,0,0)} \zeta' \cdot \bar{\zeta}' &= \nabla_{x,t,s}^2 \im(\Theta)|_{(0,0,0)} \zeta \cdot \bar{\zeta} \\
 &= \abs{\I_1^{1/2} (\left( \begin{array}{c} v' \\ 0 \end{array} \right) - \dot{x}_1 z)}^2 + \abs{\I_2^{1/2} (\left( \begin{array}{c} v' \\ 0 \end{array} \right) - \dot{x}_2 w)}^2.
\end{align*}
Assume that the last quantity vanishes. Using that $\tau$ and $\sigma$ are simple roots, we have $\dot{x}_{1,n}(0) \neq 0$ and $\dot{x}_{2,n}(0) \neq 0$ as in \eqref{xndot_simple_condition}. Since $\I_1$ and $\I_2$ are positive definite, we obtain $z=w=0$ and $v' = 0$. This proves that $\nabla_{x',t,s}^2 \Theta|_{(0,0,0)}$ is invertible. Thus stationary phase \cite[Theorem 7.7.17(ii)]{Hormander} yields that 
\[
\int_M \int_{\mR^2} e^{i\Theta(x,t,s)/h} q(x,t,s) \chi(t,s) \,ds \,dt \,dV(x) = c_0 h^{\frac{n+3}{2}} q(x_0,0,0) + O(h^{\frac{n+5}{2}}), \qquad c_0 \neq 0.
\]
Multiplying \eqref{boundary_subcasetwo_integral} by $h^{-\frac{n+3}{2}}$ and letting $h \to 0$ we obtain that $q(x_0,0,0) = 0$, i.e.\ 
\[
p_{m,1}(x_0,\xi_0+\tau \nu) = p_{m,2}(x_0,\xi_0+\tau \nu).
\]
Thus $p_{m,2}(x_0,\xi_0+\tau \nu) = 0$. This remains true for $(x,\xi) \in T^* (\p M)$ close to $(x_0,\xi_0)$, which proves \eqref{pm2_tau_roots}.

\vspace{10pt}

\noindent {\it Concluding the proof.} We have now proved \eqref{pm2_tau_roots} in all cases. As discussed after \eqref{pm2_tau_roots}, this gives the claim \eqref{boundary_determination_principal_claim} for $j=0$. To prove this for all $j$, we assume that the claim has already been proved for $j \leq k-1$. Extend the function $c$ as a smooth nonvanishing function to $M$, and replace $P_1$ by $c^{-1} P_1$ (so $p_{m,1}$ is replaced by $c^{-1} p_{m,1}$). By Lemma \ref{pm2_tau_roots} we still have $C_{P_1} = C_{P_2}$, and 
\begin{equation} \label{pxnj_condition_first}
\p_{x_n}^j (p_{m,1} - p_{m,2})(x',0,\eta) = 0 \text{ when $j \leq k-1$, $x'$ is close to $0$, $\eta \in \mR^n$}.
\end{equation}
Thus by the Taylor formula, for $x$ near $0$ one has  
\begin{equation} \label{pxnj_condition_second}
p_{m,1}(x,\eta) - p_{m,2}(x,\eta) = x_n^k f(x,\eta), \qquad f(x',0,\eta) = \frac{1}{k!} \p_{x_n}^k(p_{m,1} - p_{m,2})(x',0,\eta).
\end{equation}

Let us assume that we are in Subcase (1a), with $p_{m,1}(x,\xi+\tau \nu) = 0$ and $\im(\tau) > 0$ where $\tau = \tau(x')$. Multiplying \eqref{integral_identity_boundary_determination_first} by $h^{m-k}$, we have 
\begin{align*}
0 &= \lim_{h \to 0} h^{m-k} \int_M (P_1-P_2)(v_1) \bar{v}_2 \,dV \\
 &= \lim_{h \to 0} h^{-k-1} \int_M (p_{m,1}(x,d\Phi_1)-p_{m,2}(x,d\Phi_1)) e^{i (\Phi_1-\bar{\Phi}_2)/h} a_1 \bar{a}_2 \,dV.
\end{align*}
Note that $p_{m,1}(x,d\Phi_1(x))-p_{m,2}(x,d\Phi_1(x)) = x_n^k r(x)$ where $r(x) := f(x,d\Phi_1(x))$. Arguing as in Subcase (1a), we get 
\begin{align*}
 &h^{-k-1} \int_M e^{i\Theta/h} x_n^k r a_1 \bar{a}_2 \,dV \\
 &= h^{-k-1} \int_{\{ x_n > 0 \}} e^{i(\tau-\bar{\sigma}) x_n/h + O(x_n^2)/h} x_n^k r a_1 \bar{a}_2 \abs{g}^{1/2} \,dx \\
 &= \int_{\{ x_n > 0 \}}  x_n^k e^{i(\tau-\bar{\sigma})x_n + h O(x_n^2)} (r a_1 \bar{a}_2 \abs{g}^{1/2})(x',h x_n) \,dx \\
 &\to \int_{\{ x_n > 0 \}} x_n^k e^{i(\tau-\bar{\sigma})x_n}  (r a_1 \bar{a}_2 \abs{g}^{1/2})(x',0) \,dx' \\
 &= \int_{\mR^{n-1}} \frac{\Gamma(k+1)}{(-i(\tau-\bar{\sigma}))^{k+1}} (r a_1 \bar{a}_2 \abs{g}^{1/2})(x',0) \,dx'
\end{align*}
using that $\int_0^{\infty} t^k e^{iwt} \,dt = \frac{\Gamma(k+1)}{(-iw)^{k+1}}$ when $\im(w) > 0$. Since $a_j(x',0)$ can be chosen arbitrarily near $0$, we obtain that $r(x',0) = 0$ near $0$, i.e.\ that for $(x,\xi) \in T^*(\p M)$ close to $(x_0,\xi_0)$ 
\begin{equation} \label{pxnk_difference_tau}
\p_{x_n}^k(p_{m,1} - p_{m,2})(x,\xi+\tau\nu) = 0.
\end{equation}
This is true for each of the $m$ distinct zeros $\tau$ of $t \mapsto p_{m,1}(x',0,\xi',t)$, hence by real-analyticity in the fiber variable it follows that 
\[
\p_{x_n}^k(p_{m,1} - p_{m,2})(x',0,\eta) = \tilde{c}(x') p_{m,1}(x',0,\eta)
\]
where $\tilde{c}(x')$ is smooth near $0'$ and has the invariant expression 
\[
\tilde{c}(y) = \frac{\p_{x_n}^k(p_{m,1} - p_{m,2})(y,\nu_y)}{p_{m,1}(y,\nu_y)}.
\]
We now replace $P_1$ by $(1-\chi(x) \tilde{c}(x') x_n^k/(k!)) P_1$ where $\chi$ is a cutoff to a small neighborhood of $0$ with $\chi=1$ near $0$. Then $C_{P_1}$ and condition \eqref{pxnj_condition_first} remain unchanged, and we have 
\[
\p_{x_n}^k(p_{m,1} - p_{m,2})(x',0,\eta) = 0.
\]

The argument above shows that if \eqref{boundary_determination_principal_claim} holds for $j \leq k-1$, then possibly after replacing $c$ by some $c_k$, \eqref{boundary_determination_principal_claim} holds for $j \leq k$. The Taylor coefficients of $c_k$ at $x_n=0$ up to order $k$ are uniquely determined by \eqref{boundary_determination_principal_claim}. Thus we may use Borel summation to construct a smooth nonvanishing function $c \in C^{\infty}(M)$ with this Taylor series at $x_n=0$ near $0$, and with this choice of $c$ the condition \eqref{boundary_determination_principal_claim} holds for all $j$. This concludes the proof if we are in Subcase (1a).

The arguments for Subcase (1b) and Case 2 are quite analogous, and we will only give a sketch for Subcase (1b). Again we arrange so that \eqref{pxnj_condition_second} holds. Then \eqref{boundary_subcaseoneb_integral} is replaced by 
\[
\int_{\{x_n > 0\}} \int_{t_-}^{t_+}  e^{i \Theta(x,t)/h} x_n^k r(x,t) \,dt \,dx + \text{lower order terms} = O(h^{\infty})
\]
where $r(x,t) = f(x,d_x \Phi_1(x,t)) a_1(x,t) \bar{a}_2(x) \abs{g(x)}^{1/2}$. As discussed in Subcase (1b), the main contribution comes from the integral 
\[
\int_{\{x_n > 0\}} \int_{\mR}  e^{i \Theta(x,t)/h} x_n^k r(x,t) \chi(t) \,dt \,dx
\]
where $\chi(t)$ is a cutoff to the region near $t=0$. We write 
\[
e^{i\Theta/h} = \left( \frac{h}{i \p_{x_n} \Theta} \p_{x_n} \right)^k (e^{i\Theta/h}).
\]
Integrating by parts and noting that the boundary terms always vanish due to the $x_n^k$ factor, we see that the largest contribution with respect to $h$ comes when all the $\p_{x_n}$ derivatives hit the $x_n^k$ factor. This term has the form 
\[
(ih)^k (k!) \int_{\{x_n > 0\}} \int_{\mR}  e^{i \Theta(x,t)/h} (\p_{x_n} \Theta)^{-k} r(x,t) \chi(t) \,dt \,dx.
\]
The stationary phase argument in Subcase (1b) gives that $r(0,0) = 0$, i.e.\ that $f(x_0,d\Phi_1(x_0,0)) = 0$, which means that 
\[
\p_{x_n}^k(p_{m,1} - p_{m,2})(x_0,\xi_0+\tau \nu) = 0.
\]
This is the counterpart of \eqref{pxnk_difference_tau}. The argument after \eqref{pxnk_difference_tau} can now be repeated to conclude the proof.
\end{proof}

\begin{Remark}
The proof shows that even when the principal parts of $P_j$ have complex coefficients (and if the equations $P_j u = f$ are solvable in a suitable sense), it is possible to obtain boundary determination results when $t \mapsto p_m(x,\xi+t\nu)$ has sufficiently many simple non-real roots whose imaginary parts have suitable signs.
\end{Remark}

Finally, we give the proof of Theorem \ref{thm_boundary_determination_lower_order_intro} on boundary determination for lower order terms. The proof is very similar to that of Theorem \ref{thm_boundary_determination_principal_intro}, so we will only sketch the required modifications.

\begin{proof}[Proof of Theorem \ref{thm_boundary_determination_lower_order_intro}]
Let $P_j := P + Q_j$ and $Q := Q_1-Q_2$, and let $q(x,\xi)$ be the principal symbol of $Q$. We need to show that 
\[
\p_{x_n}^j q(x_0,\eta) = 0, \qquad j \geq 0, \ \eta \in \mR^n.
\]
The integral identity \eqref{integral_identity_boundary_determination_first} becomes 
\[
(Qu_1, u_2)_{L^2(M)} = 0
\]
for solutions of $P_1 u_1 = 0$ and $P_2^* u_2 = 0$ in $M$.

Now $p_{m,1} = p_{m,2} = p_m$. By assumption $t \mapsto p_m(x_0,\xi_0 + t \nu)$ has $s$ distinct simple roots $\{ \tau_1, \ldots, \tau_s \}$, and these roots have nonnegative imaginary parts. We want to prove that 
\begin{equation} \label{qxzeroxizero_vanishing}
q(x_0,\xi_0 + \tau_j \nu) = 0, \qquad 1 \leq j \leq s.
\end{equation}

Assume first that some root $\sigma \in \{ \tau_1, \ldots, \tau_s \}$ is non-real, with $\im(\sigma) > 0$. This is the counterpart of Case (1). Let $\tau$ be one of the roots $\tau_1, \ldots, \tau_s$. If $\tau$ is non-real, then  $\im(\tau) > 0$, and we argue as in Subcase (1a) above. The counterpart of \eqref{subcaseonea_first_identity} is  
\[
0 = \lim_{h \to 0} h^{-1} \int_M q(x, d\Phi_1(x)) e^{i (\Phi_1-\bar{\Phi}_2)/h} a_1 \bar{a}_2 \,dV.
\]
Evaluating the limit as in Subcase (1a), we obtain that 
\[
q(x_0, \xi_0 + \tau \nu) = 0.
\]
Thus $t \mapsto q(x_0, \xi_0 + t \nu)$ vanishes at each non-real root $\tau$ in the set $\{ \tau_1, \ldots, \tau_s \}$. However, if $\tau$ is a real root, the argument in Subcase (1b) implies that $q(x_0,\xi_0 + \tau \nu) = 0$ also in this case. This proves \eqref{qxzeroxizero_vanishing} if one of the roots is non-real.

Next assume that all roots in $\{ \tau_1, \ldots, \tau_s \}$ are real, and $s \geq \max(r+1, 2)$. Let $\sigma$ be one of these roots, and let $\tau \neq \sigma$ be another one of these roots. The argument in Case (2) implies that 
\[
q(x_0, \xi_0 + \tau \nu) = 0.
\]
Now $q(x_0, \xi_0 + \tau \nu) = 0$ at each root $\tau \neq \sigma$, and choosing $\sigma$ to be a different root (this is where we need $s \geq 2$) implies \eqref{qxzeroxizero_vanishing}.

Finally, if $q(x,\xi)$ is real valued and if some $\tau = \tau_j$ is such that $\im(\tau) < 0$, then the argument above shows that $q(x_0,\xi_0+\bar{\tau} \nu) = 0$. Taking complex conjugates gives $q(x_0,\xi_0+\tau \nu) = 0$. Thus if $Q_1-Q_2$ has real principal symbol, then \eqref{qxzeroxizero_vanishing} holds without the assumption that the roots $\tau_j$ have nonnegative imaginary parts.

Now, the polynomial $t \mapsto q(x_0, \xi_0 + t \nu)$ has degree $r$, and \eqref{qxzeroxizero_vanishing} implies that it has at least $s \geq r+1$ distinct roots. Thus 
\[
q(x_0, \xi_0 + t\nu) = 0, \qquad t \in \mR.
\]
The same result holds when $(x_0,\xi_0)$ is varied slightly, and real-analyticity in $\eta$ implies that 
\[
q(x_0,\eta) = 0, \qquad \eta \in T_{x_0}^* M.
\]
Arguing as in the end of proof of Theorem \ref{thm_boundary_determination_principal_intro} gives that 
\[
\p_{x_n}^j q(x_0,\eta) = 0, \qquad j \geq 0, \ \eta \in T_{x_0}^* M.
\]
This finishes the proof.
\end{proof}

\section{Semilinear equations} \label{sec_semilinear}

In this section we prove Theorem \ref{thm_semilinear_uniqueness_intro} related to semilinear equations of the form 
\[
Pu + a(x,u) = 0 \text{ in $M$}
\]
where $a(x,z)$ is a nonlinearity satisfying the following conditions. Fix an integer $s > \max(m, n/2)$, and assume initially that 
\begin{gather}
\text{$a(x,z)$ is analytic in $z$ near $0$ as a $H^s(M)$-valued function}, \label{semilinear_condition_one} \\
a(x,0) = \p_z a(x,0) = 0. \label{semilinear_condition_two}
\end{gather}

The proof is based on constructing solutions $u_{\eps_1, \ldots, \eps_r}$ to the semilinear equation that are close to $\eps_1 v_1 + \ldots + \eps_r v_r$, where $v_j$ are suitable solutions of the linearized equation $P v_j = 0$ concentrating near null bicharacteristics, and on higher order linearization with respect to the parameters $\eps_j$. The following result will allow us to construct such solutions under the weak uniqueness assumption $N(P^*) = \{0\}$.

\begin{Lemma} \label{lemma_semilinear_solution_existence}
Let $M$ be a compact manifold with smooth boundary, and let $P$ be a real principal type differential operator on $M$ with $N(P^*) = \{0\}$. Assume that $a(x,z)$ satisfies \eqref{semilinear_condition_one}--\eqref{semilinear_condition_two}. There exist $\delta, C > 0$ such that for any $v$ in the set 
\[
X_{\delta} = \{v \in H^s(M) \,;\, \norm{v}_{H^s(M)} < \delta \},
\]
there is a solution $u = S(v) \in H^s(M)$ of 
\[
Pu + a(x,u) = Pv \text{ in $M$}
\]
which satisfies 
\begin{equation} \label{u_v_difference_quadratic_estimate}
\norm{u-v}_{H^s(M)} \leq C \norm{v}_{H^s(M)}^2.
\end{equation}
The map $S: X_{\delta} \to H^s(M), \ v \mapsto u$ is $C^{\infty}$. In particular, if we define 
\[
U_{\delta} = \{v \in X_{\delta} \,;\, Pv = 0 \},
\]
then for any $v \in U_{\delta}$ the function $u = S(v)$ solves $Pu + a(x,u) = 0$ in $M$.
\end{Lemma}
\begin{proof}
Since $N(P^*) = \{0 \}$, Proposition \ref{prop_real_principal_type_solvability} shows that there is a bounded linear map $E: H^s(M) \to H^{s+m-1}(M)$ with $PE = \mathrm{Id}$. Given $v \in X_{\delta}$ we wish to find a solution $u = v + r$ of $Pu + a(x,u) = Pv$. It is enough to find $r$ solving the fixed point equation 
\begin{equation} \label{fixed_point_equation}
r = T_v(r)
\end{equation}
where $T_v(r) = -E(a(x,v+r))$.

We now study \eqref{fixed_point_equation} for $v \in X_{\delta}$, and show that $T_v$ is a contraction on $X_{\delta}$ for $\delta$ small enough. All constants below will be uniform over $v \in X_{\delta}$. Note first that by \eqref{semilinear_condition_one}, for some $\delta_0, R > 0$ one has 
\begin{equation} \label{semilinear_condition_one_sum}
\sum_{j=0}^{\infty} \frac{\norm{\p_z^j a(x,0)}_{H^s(M)}}{j!} \delta_0^j \leq R.
\end{equation}
Then \eqref{semilinear_condition_two} and the fact that $H^s(M)$ for $s > n/2$ is an algebra imply that 
\begin{equation} \label{axw_sobolev_regularity}
\norm{a(x,w)}_{H^s(M)} \leq \sum_{j=2}^{\infty} \norm{\frac{\p_u^j a(x,0)}{j!} w^j}_{H^s(M)} \leq C_s \sum_{j=2}^{\infty} \frac{\norm{\p_u^j a(x,0)}_{H^s(M)}}{j!} C_s^j \norm{w}_{H^s(M)}^j.
\end{equation}
If $\delta$ is chosen small enough, \eqref{semilinear_condition_one_sum} yields 
\begin{equation} \label{axw_quadratic_bound}
\norm{a(x,w)}_{H^s(M)} \leq C \norm{w}_{H^s(M)}^2, \qquad w \in X_{\delta},
\end{equation}
and similarly 
\begin{equation} \label{puaxw_linear_bound}
\norm{\p_u a(x,w)}_{H^s(M)} \leq C \norm{w}_{H^s(M)}, \qquad w \in X_{\delta}.
\end{equation}
Now 
\begin{equation} \label{tvw_bound}
\norm{T_v(w)}_{H^s} \leq \norm{E(a(x,v+w))}_{H^{s+m-1}} \leq C \norm{a(x,v+w)}_{H^s} \leq C \norm{v+w}_{H^s}^2, \qquad w \in X_{\delta},
\end{equation}
showing that $T$ maps $X_{\delta}$ to itself when $\delta$ is small. Similarly, 
\begin{align*}
\norm{T_v(u)-T_v(w)}_{H^s} &\leq \norm{E(a(x,v+u)-a(x,v+w))}_{H^{s+m-1}} \leq C \norm{a(x,v+u)-a(x,v+w)}_{H^s} \\
 &\leq C \left( \int_0^1 \norm{\p_u a(x,v + (1-t)u + tw)} _{H^s} \,dt \right) \norm{u-w}_{H^s}.
\end{align*}
By \eqref{puaxw_linear_bound}, if $\delta$ is small enough one has 
\begin{equation} \label{tvu_tvw_bound}
\norm{T_v(u)-T_v(w)}_{H^s} \leq \frac{1}{2} \norm{u-w}_{H^s}, \qquad u, w \in X_{\delta}.
\end{equation}

We have proved that there is $\delta > 0$ so that for any $v \in X_{\delta}$, the map $T_v: X_{\delta} \to X_{\delta}$ satisfies \eqref{tvw_bound}--\eqref{tvu_tvw_bound}. By the Banach fixed point theorem, there is a unique $r = r_v \in X_{\delta}$ solving $r = T_v(r)$. Writing $u = S(v) = v + r$ we obtain a solution of $Pu + a(x,u) = Pv$. The estimate \eqref{tvw_bound} yields 
\[
\norm{r}_{H^s} = \norm{T_v(r)}_{H^s} \leq C \norm{v+r}_{H^s}^2 \leq C (\norm{v}_{H^s}^2 + \norm{r}_{H^s}^2),
\]
and since $r \in X_{\delta}$ it follows that for $\delta$ small enough 
\[
\norm{u-v}_{H^s} = \norm{r}_{H^s} \leq C \norm{v}_{H^s}^2.
\]
To show that $S$ is $C^{\infty}$, consider the map 
\[
F: X_{\delta} \times X_{\delta} \to X_{\delta}, \ \ F(v,r) = r - T_v(r) = r + E(a(x,v+r)).
\]
Now $F$ is $C^{\infty}$ by \eqref{semilinear_condition_one}. It satisfies $F(0,0) = 0$ and 
\[
D_r F_{(0,0)}(h) = h + E(\p_u a(x,0)h) = h.
\]
Thus $D_r F_{(0,0)}$ is a homeomorphism, and the implicit function theorem \cite[Theorem 10.6]{RenardyRogers} shows that there is a smooth map $G$ near $0$ in $X_{\delta}$ so that $F(v,r) = 0$ near $0$ iff $r = G(v)$. But $F(v,r) = 0$ iff $r = r_v$, which shows that $G(v) = r_v$ and that $v \mapsto r_v$ is $C^{\infty}$. It follows that also  $S$ is $C^{\infty}$.
\end{proof}

The next technical lemma shows that if $u$ is a small solution to $Pu + a(x,u) = 0$ depending smoothly on a solution $v$ of $Pv = 0$, and if the Cauchy data sets for nonlinearities $a$ and $\tilde{a}$ coincide, then the corresponding small solution of $P \tilde{u} + \tilde{a}(x,\tilde{u}) = 0$ with $u - \tilde{u} \in H^m_0(M)$ also depends smoothly on $v$. The argument requires the uniqueness assumption $N(P) = \{0\}$.

\begin{Lemma} \label{lemma_utilde_differentiability}
Let $M$ be a compact manifold with smooth boundary, and let $P$ be a real principal type differential operator on $M$ with $N(P) = N(P^*) = \{0\}$. Suppose that $a, \tilde{a}$ satisfy \eqref{semilinear_condition_one}--\eqref{semilinear_condition_two} and they agree to high order on $\p M$ in the sense that 
\begin{equation} \label{a_atilde_agree_boundary}
a(\,\cdot\,,z)-\tilde{a}(\,\cdot\,,z) \in H^s_0(M) \text{ for $z$ near $0$.}
\end{equation}
Assume that  for any sufficiently small $\delta > 0$ there is $\delta_1 < \delta$ so that $C_{a,\delta_1} \subset C_{\tilde{a},\delta}$.

If $\delta$ is small enough, given any $v \in U_{\delta}$ and solution $u = S(v)$ of $Pu + a(x,u) = 0$ there is a small solution $\tilde{u} \in H^s(M)$ of $P \tilde{u} + \tilde{a}(x,\tilde{u}) = 0$ such that $u-\tilde{u} \in H^m_0(M)$, the map $v \mapsto \tilde{u}$ is $C^{\infty}$ from $U_{\delta}$ to $H^s(M)$, and 
\begin{equation} \label{utilde_v_difference_estimate}
\norm{\tilde{u} - v}_{H^s(M)} \leq C \norm{v}_{H^s(M)}^2.
\end{equation}
\end{Lemma}
\begin{proof}
Fix any small $\delta > 0$. By assumption there is $\delta_1 < \delta$ with $C_{a,\delta_1} \subset C_{\tilde{a},\delta}$. Moreover, by \eqref{u_v_difference_quadratic_estimate} there is $\delta_2 < \delta_1$ so that $v \in U_{\delta_2}$ implies that $u = S(v) \in X_{\delta_1}$. Thus for any $v \in U_{\delta_2}$ there is $\tilde{u} \in X_{\delta}$ solving $P \tilde{u} + \tilde{a}(x,\tilde{u}) = 0$ such that $u-\tilde{u} \in H^m_0(M)$. We need to show that the map $\tilde{S}: U_{\delta_2} \to H^s(M)$, $v \mapsto \tilde{u}$ is smooth. In order to do this, write $w = w_v := u-\tilde{u} \in H^s(M) \cap H^m_0(M)$, and note that $w$ solves 
\[
Pw = f, \qquad f := \tilde{a}(x,u-w) - a(x,u).
\]
We wish to express $w$ as the solution of a fixed point equation as in Lemma \ref{lemma_semilinear_solution_existence}, and show that the map $v \mapsto w_v$ is smooth.

We first prove that in fact $w \in H^{s+m-1}_0(M)$ and $Pw = f \in H^s_0(M)$. Let $M$ be contained in some larger manifold $X$, let $\breve{u}$ be some $H^s_{\mathrm{comp}}(X)$ extension of $u$, and let $\breve{w}$ be the zero extension of $w$ outside $M$. Then $\breve{w} \in H^m_M(X)$ solves 
\[
P \breve{w} = \tilde{a}(x,\breve{u}-\breve{w}) - a(x,\breve{u}) \text{ a.e.\ in $X$}
\]
since the right hand side vanishes outside $M$ by the assumption \eqref{a_atilde_agree_boundary}. By the argument leading to \eqref{axw_sobolev_regularity}, the right hand side is in fact in $H^m_M(X)$ since $H^m_{\mathrm{comp}}(X) \cap L^{\infty}(X)$ is an algebra, see \cite{KP} for the latter fact. Regularity for real principal type operators (see Proposition \ref{prop_real_principal_type_solvability}(a)) then implies that $\breve{w} \in H^{m+(m-1)}_M(X)$. Bootstrapping this regularity argument gives $\breve{w} \in H^{s+m-1}_M(X)$, which yields $w \in H^{s+m-1}_0(M)$ and $f \in H^s_0(M)$ as required.

Next we define the spaces 
\[
Y := \{\varphi \in H^{s+m-1}_0(M) \,;\, P\varphi \in H^s_0(M) \}
\]
and 
\[
Z := \{f \,;\, \text{$f = P\varphi$ for some $\varphi \in Y$} \}.
\]
Now $Y$ with norm $\norm{\varphi}_Y = \norm{\varphi}_{H^{s+m-1}} + \norm{P \varphi}_{H^s}$ is a Hilbert space. Moreover, $Z$ is a closed subspace of $H^s(M)$ (and of $H^s_0(M)$): if $P\varphi_j \to f$ in $H^s(M)$ where $\varphi_j \in Y$, then the condition $N(P) = \{0\}$ and the estimate \eqref{v_htm_equation} imply that 
\[
\norm{\varphi_j - \varphi_k}_{H^{s+m-1}} \lesssim \norm{P \varphi_j - P \varphi_k}_{H^s}.
\]
Thus $(\varphi_j)$ is a Cauchy sequence in $H^{s+m-1}_0(M)$ and converges to some $\varphi$ in this space. Since $P \varphi_j \to f$ in $H^s_0(M)$, we have $f = P \varphi \in H^s_0(M)$ showing that $Z$ is closed in $H^s(M)$.

The map $P: Y \to Z, \ \varphi \to P \varphi$ is linear, bounded and bijective by the assumption $N(P) = \{0 \}$. The open mapping theorem ensures that there is a bounded inverse $G: Z \to Y$. Let $Q$ be the orthonormal projection to $Z$ in $H^s(M)$. We may now rewrite the equation 
\[
Pw = \tilde{a}(x,u-w) - a(x,u), \qquad w \in Y,
\]
equivalently as 
\begin{equation} \label{w_gq_equation}
w = GQ(\tilde{a}(x,u-w) - a(x,u)).
\end{equation}
Consider the map 
\[
F: X_{\delta} \times X_{\delta} \to H^s(M), \ \ F(v,w) = w - GQ(\tilde{a}(x,S(v)-w) - a(x,S(v))).
\]
Then $F$ is $C^{\infty}$ and $D_w F_{(0,0)} h = h$, hence by the implicit function theorem there is a smooth map $H$ near $0$ in $X_{\delta}$ so that $F(v,w) = 0$ near $0$ iff $w = H(v)$. But $F(v,w_v) = 0$, showing that the map $v \mapsto w_v$ is smooth. Moreover, writing $u = v + r$ where $r$ satisfies 
\[
\norm{r}_{H^s} \lesssim \norm{v}_{H^s}^2,
\]
the equation \eqref{w_gq_equation} together with \eqref{semilinear_condition_two} implies that 
\[
\norm{w}_{H^s} \lesssim \norm{v+r-w}_{H^s}^2 + \norm{v+r}_{H^s}^2 \lesssim \norm{w}_{H^s}^2 + \norm{v}_{H^s}^2.
\]
Since $\norm{w}_{H^s}$ is small, this gives $\norm{w}_{H^s} \lesssim \norm{v}_{H^s}^2$, and \eqref{utilde_v_difference_estimate} follows.
\end{proof}

\begin{proof}[Proof of Theorem \ref{thm_semilinear_uniqueness_intro}]
We begin by fixing solutions $v_1, v_2, v_3 \in H^s(M)$ of $P v_j = 0$ (we will later choose $v_j$ to concentrate near certain null bicharacteristics). Let $\eps = (\eps_1, \eps_2, \eps_3)$, and define $u_{\eps} = S(\eps_1 v_1 + \eps_2 v_2 + \eps_3 v_3)$. Then $u_{\eps}$ depends smoothly on each $\eps_j$, and by \eqref{u_v_difference_quadratic_estimate} one has 
\[
u_0 = 0, \qquad \p_{\eps_j} u_{\eps}|_{\eps = 0} = v_j.
\]
Differentiating the equation $Pu_{\eps} + a(x,u_{\eps}) = 0$ and using the conditions $\p_u^k a(x,0) = 0$ for $k = 0,1,2$, we see that $v_{123} := \p_{\eps_1 \eps_2 \eps_3} u_{\eps}|_{\eps = 0}$ satisfies 
\begin{equation} \label{vonetwothree_equation}
Pv_{123} + \p_u^3 a(x,0) v_1 v_2 v_3 = 0.
\end{equation}
We now let $\tilde{u}_{\eps} = \tilde{S}(\eps_1 v_1 + \eps_2 v_2 + \eps_3 v_3)$ be the function in Lemma \ref{lemma_utilde_differentiability}. Then $\tilde{u}_{\eps}$ depends smoothly on the $\eps_j$, and using \eqref{utilde_v_difference_estimate} we have 
\[
\tilde{u}_0 = 0, \qquad \p_{\eps_j} \tilde{u}_{\eps}|_{\eps=0} = v_j.
\]
Similarly, differentiating the equation $P\tilde{u}_{\eps} + \tilde{a}(x,\tilde{u}_{\eps}) = 0$ shows that  $\tilde{v}_{123} := \p_{\eps_1 \eps_2 \eps_3} \tilde{u}_{\eps}|_{\eps = 0}$ satisfies  
\begin{equation} \label{vtildeonetwothree_equation}
P \tilde{v}_{123} + \p_u^3 \tilde{a}(x,0) v_1 v_2 v_3 = 0.
\end{equation}
Since $u_{\eps} - \tilde{u}_{\eps} \in H^m_0(M)$, also $v_{123} - \tilde{v}_{123} \in H^m_0(M)$. Subtracting the equations \eqref{vonetwothree_equation} and \eqref{vtildeonetwothree_equation} and integrating against $v_4$ where $\bar{P}^* v_4 = 0$ (so that $P^* \bar{v}_4$ = 0) yields that 
\begin{equation} \label{third_derivative_orthogonality}
\int_M (\p_u^3 a(x,0) - \p_u^3 \tilde{a}(x,0)) v_1 v_2 v_3 v_4 \,dV = 0.
\end{equation}

We have proved that \eqref{third_derivative_orthogonality} holds for any solutions $v_j$ with $P v_j = 0$ for $1 \leq j \leq 3$ and $\bar{P}^* v_4 = 0$. We now assume that $x_0 \in B$, i.e.\ there are two maximal null bicharacteristics $\gamma_j: [-S_j,T_j] \to T^* M \setminus 0$, $\gamma_j(t) = (x_j(t), \xi_j(t))$, $j=1,2$, so that the curves $x_1$ and $x_2$ only intersect at $x_0$ when $t=0$ and $\dot{x}_1(0) \nparallel \dot{x}_2(0)$. Let $\gamma_3(t) = (x_3(t), \xi_3(t))$ be the null bicharacteristic with $x_3(0) = x_0$ and $\xi_3(0) = - \xi_1(0)$ (if $m$ is odd $\gamma_3(t) = (x_1(t),-\xi_1(t))$, while if $m$ is even $\gamma_3(t) = (x_1(-t), -\xi_1(-t))$). Similarly, let $\gamma_4(t) = (x_4(t), \xi_4(t))$ be the null bicharacteristic with $x_4(0) = x_0$ and $\xi_4(0) = - \xi_2(0)$ (this is a null bicharacteristic for $\bar{P}^*$ since $P$ has real principal symbol). We consider solutions 
\[
v_j = q_j + r_j
\]
where $q_j$ is a quasimode given by Theorem \ref{thm_quasimode_direct} associated with $\gamma_j$ extended slightly outside $M$, so that $\norm{P q_j}_{H^s(M)} = O(h^{\infty})$ and $q_j$ is supported in a small neighborhood of the curve $x_j(t)$, and $r_j$ are solutions of $Pr_j = -Pq_j$ in $M$ obtained from Proposition \ref{prop_real_principal_type_solvability} and satisfy $\norm{r_j}_{H^{s+m-1}(M)} = O(h^{\infty})$. Then \eqref{third_derivative_orthogonality} gives 
\begin{equation} \label{third_derivative_orthogonality_second}
\int_M (\p_u^3 a(x,0) - \p_u^3 \tilde{a}(x,0)) q_1 q_2 q_3 q_4 \,dV = O(h^{\infty}).
\end{equation}

Since $x_1(t)$ and $x_2(t)$ only intersect at $x_0$, the product $q_1 q_2 q_3 q_4$ is supported in a small neighborhood of $x_0$. By \eqref{quasimode_integral_ansatz} each $q_j$ has the form 
\[
q_j(x) = \int_{-S_j}^{T_j} e^{i\Phi_j(x,t_j)/h} a_j(x,t_j) \,dt_j
\]
where the maximal null bicharacteristic $\gamma_j$ is defined on $[-S_j,T_j]$ with $x_j(0) = x_0$ for $1 \leq j \leq 4$. Using \eqref{phi_req1_second}--\eqref{dttphi_formula}, the phase functions satisfy 
\[
\im(\Phi_j) \geq 0, \qquad \Phi_j(x_0,0) = 0, \qquad d_{x,t_j} \Phi_j(x_0,0) = (\xi_j(0),0),
\]
and  
\[
\nabla_{x,t_j}^2 \Phi_j(x_0,0) = \left( \begin{array}{cc} H_j & -H_j \dot{x}_j+\dot{\xi}_j \\ (-H_j \dot{x}_j + \dot{\xi}_j)^t & (H_j \dot{x}_j - \dot{\xi}_j) \cdot \dot{x}_j \end{array} \right)
\]
where the last expression is computed in Riemannian normal coordinates at $x_0$ and evaluated at $t_j = 0$. Recall that $\im(H_j)$ are positive definite matrices.

Inserting the formulas for $q_j$ into \eqref{third_derivative_orthogonality_second} and writing $\mathbf{t} = (t_1, \ldots, t_4)$, we obtain that 
\begin{equation} \label{third_derivative_orthogonality_third}
\int_M \int_{\mR^4} (\p_u^3 a(x,0) - \p_u^3 \tilde{a}(x,0)) e^{i\Theta(x,\mathbf{t})/h} b(x,\mathbf{t}) \,d\mathbf{t} \,dV = O(h^{\infty})
\end{equation}
where $b(x,\mathbf{t})$ is supported in a small neighborhood of $(x_0,\mathbf{0})$ in $M^{\mathrm{int}} \times \mR^4$ by the assumption that $x_1(t)$ and $x_2(t)$ only intersect once at $t=0$. Here we have written 
\[
\Theta(x,\mathbf{t}) = \Phi_1(x,t_1) + \ldots + \Phi_4(x,t_4), \qquad b(x,\mathbf{t}) = a_1(x,t_1) \cdots a_4(x,t_4).
\]
Since we have arranged that 
\[
\xi_1(0) + \xi_2(0) + \xi_3(0) + \xi_4(0) = 0,
\]
the properties of $\Phi_j$ above ensure that 
\[
\im(\Theta) \geq 0, \qquad \Theta(x_0,\mathbf{0}) = 0, \qquad d_{x,\mathbf{t}} \Theta(x_0,\mathbf{0}) = 0.
\]
To show that the Hessian of $\Theta$ is invertible at $(x_0,\mathbf{0})$, we write $\mathbf{z} = (z_1, \ldots, z_4)$ and observe that 
\begin{equation} \label{theta_hessian_semilinear}
\nabla_{x,\mathbf{t}}^2 \Theta(x_0,\mathbf{0})\left( \begin{array}{c} v \\ \mathbf{z} \end{array} \right) = \left( \begin{array}{c} \sum_{j=1}^4 [ H_j (v - \dot{x}_j z_j)  + \dot{\xi}_j z_j ] \\ -(H_1 \dot{x}_1 - \dot{\xi}_1) \cdot (v - \dot{x}_1 z_1) \\ \vdots \\ -(H_4 \dot{x}_4 - \dot{\xi}_4) \cdot (v - \dot{x}_4 z_4) \end{array} \right).
\end{equation}
If \eqref{theta_hessian_semilinear} vanishes, this would imply that 
\[
0 = \nabla_{x,\mathbf{t}}^2 \im(\Theta)(x_0,\mathbf{0})\left( \begin{array}{c} v \\ \mathbf{z} \end{array} \right) \cdot \ol{\left( \begin{array}{c} v \\ \mathbf{z} \end{array} \right)} = \sum_{j=1}^4 \,\abs{\im(H_j)^{1/2}(v-\dot{x}_j z_j)}^2
\]
and hence $v-\dot{x}_j z_j = 0$ for $1 \leq j \leq 4$. Thus in particular $\dot{x}_1 z_1 = \dot{x}_2 z_2$, and the condition $\dot{x}_1(0) \nparallel \dot{x}_2(0)$ implies that $z_1 = z_2 = 0$.
%Inserting these back to \eqref{theta_hessian_semilinear} would further imply that 
%\begin{equation} \label{xij_linear_combination}
%\sum_{j=1}^4 \dot{\xi}_j z_j = 0.
%\end{equation}
But the choice of $\gamma_3$ and $\gamma_4$ implies that when $t_j=0$, 
\[
\dot{x}_3 = (-1)^{m-1} \dot{x}_1, \qquad \dot{\xi}_3 = (-1)^m \dot{\xi}_1, \qquad \dot{x}_4 = (-1)^{m-1} \dot{x}_2, \qquad \dot{\xi}_4 = (-1)^m \dot{\xi}_2.
\]
Since $v-\dot{x}_j z_j = 0$ and since $\dot{x}_1(0)$ and $\dot{x}_2(0)$ are nonvanishing, this yields $z_3 = (-1)^{m-1} z_1$ and $z_4 = (-1)^{m-1} z_2$. Thus $z_j = 0$ for $1 \leq j \leq 4$ and also $v = 0$, showing that $\nabla_{x,\mathbf{t}}^2 \Theta(x_0,\mathbf{0})$ is invertible.

Now using stationary phase in \eqref{third_derivative_orthogonality_third}, we obtain 
\[
c_0 (\p_u^3 a(x_0,0) - \p_u^3 \tilde{a}(x_0,0)) h^{\frac{n+4}{2}} = O(h^{\frac{n+4}{2}+1})
\]
where $c_0 \neq 0$ since the amplitudes $a_j(x,t_j)$ are nonvanishing at $(x_0,0)$. This proves that $\p_u^3 a(x_0,0) = \p_u^3 \tilde{a}(x_0,0)$ whenever $x_0 \in B$.

We have proved that 
\[
\p_u^k a(x_0,0) = \p_u^k \tilde{a}(x_0,0), \qquad k \leq 3.
\]
To prove that this holds also for $k = 4$, we consider $u_{\eps} = S(\eps_1 v_1 + \ldots + \eps_4 v_4)$ where $P v_j = 0$ and repeat the above argument (now looking at the $\p_{\eps_1 \eps_2 \eps_3 \eps_4}$ derivative at $\eps=0$) to obtain that 
\[
\int_M (\p_u^4 a(x,0) - \p_u^4 \tilde{a}(x,0)) v_1 v_2 v_3 v_4 v_5 \,dV = 0
\]
whenever $\bar{P}^* v_5 = 0$. We now let $v_j$ be a solution related to $\gamma_j$ above for $1 \leq j \leq 4$, with the difference that now $v_4$ solves $P v_4 = 0$ instead of $\bar{P}^* v_4 = 0$. We also choose $v_5$ to be any solution of $\bar{P}^* v_5 = 0$ with $v_5(x_0) \neq 0$ so that $v_5$ is independent of $h$. For instance, it is enough to choose $v_5$ to be related to the null bicharacteristic $\gamma_1$ through $x_0 = x_1(0)$, so that $\dot{x}_1(0) \neq 0$ and $x_1(t) \neq x_1(0)$ for $t \neq 0$ by the assumption that $x_0 \in B$. Then Lemma \ref{lemma_bicharacteristic_solution_values} gives that $v_5(x_0) \neq 0$, when $h=h_0$ and $h_0$ is fixed but sufficiently small. The argument above now gives that 
\[
(\p_u^4 a(x_0,0) - \p_u^4 \tilde{a}(x_0,0)) v_5(x_0) = 0.
\]
Since $v_5(x_0) \neq 0$, we have proved that 
\[
\p_u^k a(x_0,0) = \p_u^k \tilde{a}(x_0,0), \qquad k \leq 4.
\]
Continuing in this way shows that all derivatives of $a(x_0,\,\cdot\,) - \tilde{a}(x_0,\,\cdot\,)$ vanish at $0$, which implies that $a(x_0,\,\cdot\,) = \tilde{a}(x_0,\,\cdot\,)$ by analyticity.
\end{proof}

\appendix

\section{Additional proofs} \label{sec_appendix}

In this appendix we give for completeness the proofs of some auxiliary results used in the article.

\begin{proof}[Proof of Lemma \ref{lemma_conjugated_p}]
Fix some local coordinates and write 
\[
Pu = \sum_{r=0}^m \sum_{j_1, \ldots, j_r=1}^n p_{j_1 \cdots j_r}(x) D_{j_1} \cdots D_{j_r}.
\]
In terms of these local coordinates 
\begin{align*}
p_m(x,\xi) &= \sum_{j_1, \ldots, j_m=1}^n p_{j_1 \cdots j_m}(x) \xi_{j_1} \cdots \xi_{j_m}, \\
\p_{\xi_a} p_m(x,\xi) v_a &= \sum_{j_1, \ldots, j_m=1}^n p_{j_1 \cdots j_m}(x) (v_{j_1} \xi_{j_2} \cdots \xi_{j_m} + \ldots + \xi_{j_1} \cdots \xi_{j_{m-1}} v_{j_m}), \\
\p_{\xi_a \xi_b} p_m(x,\xi) v_{ab} &= \sum_{j_1, \ldots, j_m=1}^n p_{j_1 \cdots j_m}(x) \Big[ (v_{j_1 j_2}+v_{j_2 j_1}) \xi_{j_3} \cdots \xi_{j_m} \\
 &\qquad + (v_{j_1 j_3} + v_{j_3 j_1}) \xi_{j_2} \xi_{j_4} \cdots \xi_{j_m} + \ldots + (v_{j_{m-1} j_m} + v_{j_m j_{m-1}})\xi_{j_1} \cdots \xi_{j_{m-2}} \Big].
\end{align*}
One also has $p_{m-1}(x,\xi) = \sum_{j_1, \ldots, j_{m-1}=1}^n p_{j_1 \cdots j_{m-1}}(x) \xi_{j_1} \cdots \xi_{j_{m-1}}$. We now compute 
\begin{align*}
e^{-i\Phi/h} P(e^{i\Phi/h} u) &= \sum_{r=0}^m \sum_{j_1, \ldots, j_r=1}^n p_{j_1 \cdots j_r}(x) (D_{j_1} + h^{-1} \p_{j_1} \Phi) \cdots (D_{j_r} + h^{-1} \p_{j_r} \Phi) u \\
 &= \sum_{j=0}^m h^{j-m} R_j u
\end{align*}
where each $R_j$ is a differential operator of order $j$ and 
\begin{align*}
R_0 u &= p_m(x,\nabla \Phi) u, \\
R_1 u &= p_{m-1}(x,\nabla \Phi) + \sum_{j_1,\ldots,j_m} p_{j_1 \cdots j_m}(x) \big[ D_{j_1} (\p_{j_2} \Phi \cdots \p_{j_m} \Phi \,u) \\
 &\qquad + \p_{j_1} \Phi D_{j_2} (\p_{j_3} \Phi \cdots \p_{j_m} \Phi \,u) + \ldots + \p_{j_1} \Phi \cdots \p_{j_{m-1}} \Phi D_{j_m} u \big] \\
 &= p_{m-1}(x,\nabla \Phi) + \frac{1}{i} \p_{\xi_a} p_m(x,\nabla \Phi) \p_a u + \frac{1}{2i} \p_{\xi_a \xi_b} p_m(x,\nabla \Phi) \p_{ab} \Phi. \qedhere
\end{align*}
\end{proof}

\begin{proof}[Proof of Lemma \ref{lemma_hyperbolic_real_principal_type}]
We argue by contradiction and assume that there is a maximal null bicharacteristic $\gamma: I \to T^* X \setminus 0$ and a compact set $K \subset X$ with $\gamma(I) \subset T^* K$. As before, we write $x(t) = \pi(\gamma(t))$ where  $\pi: T^* X \to X$ is the standard projection.

We fix an auxiliary Riemannian metric $g$ on $X$ and consider the projection $\eta(t)$ of $\gamma(t)$ to the unit cosphere bundle $S^* X$. We think of $\xi(t)$ (i.e.\ of $\gamma(t)$) as a covector field on $x(t)$, and then $\eta(t)$ corresponds to $\omega(t) = \xi(t)/\abs{\xi(t)}$. Note that $\eta(I)$ is contained in the compact set $S^* K \cap p^{-1}(0)$. The metric $g$ induces the Sasaki metric on $T^*X$ and corresponding horizontal and vertical subbundles, see \cite{Paternain}. The Hamilton vector field $H_p$ on $T^* X$ decomposes in horizontal and vertical parts as 
\[
H_{p} = (d_{\xi} p, -\nabla_x p)
\]
where $d_{\xi} p = \pi_*(H_p) \in T X$ and $\nabla_x p \in T^* X$. By the Hamilton equations and homogeneity 
\begin{align*}
\dot{x}(t) &= d_{\xi} p(\gamma(t)) = \abs{\xi(t)}^{m-1} d_{\xi} p(\eta(t)), \\
D_t \xi(t) &= -\nabla_x p(\gamma(t)) = - \abs{\xi(t)}^m \nabla_x p(\eta(t)).
\end{align*}
Here $D_t$ is the covariant derivative along $x(t)$.

If $[t_1,t_2] \subset I$, we have 
\[
\phi(x(t_2)) - \phi(x(t_1)) = \int_{t_1}^{t_2} \frac{d}{dt} \phi(x(t)) \,dt = \int_{t_1}^{t_2} d\phi(d_{\xi} p(\gamma(t)) \,dt = \int_{t_1}^{t_2} \abs{\xi(t)}^{m-1} d\phi(d_{\xi} p(\eta(t))) \,dt.
\]
Note that since $\tau \mapsto p(y,\zeta_0 + \tau d\phi(y))$ has only simple zeros, differentiating in $\tau$ yields that 
\[
d\phi(d_{\xi} p(y,\zeta)) \neq 0, \qquad (y,\zeta) \in p^{-1}(0).
\]
Now $\eta(I)$ is contained in a compact set, so $\abs{d\phi(d_{\xi} p(\eta(t)))} \sim 1$ uniformly over $t \in I$. Taking limits where $t_j$ converge to the endpoints of $I$, we see that 
\begin{equation} \label{xit_integral}
\int_I \abs{\xi(t)}^{m-1} \,dt < \infty.
\end{equation}
On the other hand, we compute 
\[
\frac{d}{dt} \frac{1}{\abs{\xi(t)}^{m-1}} = -(m-1) \frac{\langle D_t \xi(t), \xi(t) \rangle}{\abs{\xi(t)}^{m+1}} = (m-1) \langle \nabla_x p(\eta(t)), \omega(t) \rangle.
\]
Again, since $\eta(I)$ is contained in a compact set, we have $1/\abs{\xi(t)}^{m-1} \leq C(1+\abs{t})$. In other words, for some $c > 0$ one has 
\begin{equation} \label{xit_lower_bound}
\abs{\xi(t)}^{m-1} \geq \frac{c}{1+\abs{t}}, \qquad t \in I.
\end{equation}
Now the combination of \eqref{xit_integral} and \eqref{xit_lower_bound} implies that $I$ must be a finite interval.

We have also
\[
\frac{d}{dt} \abs{\xi(t)} 
= 
\frac{\langle D_t \xi(t), \xi(t) \rangle}{\abs{\xi(t)}}
= -\abs{\xi(t)}^{m} \langle \nabla_x p(\eta(t)), \omega(t) \rangle
\le C \abs{\xi(t)}^{m}, \quad t \in I.
\]
Hence using \eqref{xit_integral} we see that $\ln \abs{\xi(t)}$
is bounded from above on $I$.
Due to this and \eqref{xit_lower_bound} one has $C^{-1} \leq \abs{\xi(t)} \leq C$ for some $C > 0$ and all $t \in I$. It follows that $\gamma(I)$ is contained in a compact set of $T^* X \setminus 0$.
This again implies, together with $\dot \gamma(t) = H_p(\gamma(t))$, that $\gamma : I \to T^* X$ is uniformly continuous. 
Hence $\gamma$ extends to $\overline I$ as a continuous function.
We may assume without loss of generality that $\inf I = 0$. Then $\lim_{t \to 0} \dot \gamma(t) = H_p(\gamma(0))$ and also $\dot \gamma$ is continuous up to $0$.
Therefore $\gamma$ is the bicharacteristic from $\gamma(0)$
and thus it can be extended on $(-\epsilon, 0)$ for small $\epsilon > 0$, a contradiction with maximality of $\gamma$.
\end{proof}

\begin{proof}[Proof of Theorem \ref{thm_partial_data}]
Let us show (\ref{subprin_part_data}). The proof is similar with Step 8 in the proof of Theorem \ref{thm_main1}, and it is enough to show 
    \begin{align}\label{subp_part_data}
\exp \left[ i\int_0^T q_{m-1}(\gamma(t)) \,dt \right]
= 1,
    \end{align}
where $q_{m-1}$ is the principal symbol of $Q = P_1 - P_2$. 
We denote by $\breve{\gamma}(t) = (x(t), \xi(t))$ the continuation of $\gamma$ as a null bicharacteristic curve in $X$.

We use Theorem \ref{thm_quasimode_direct} with $P = P_1$ and $\tilde{P} = P_2$ and construct quasimodes $v_1, v_2$ associated with $\breve{\gamma}|_{[-\eps_1,T+\eps_2]}$ with $\eps_j$ chosen so that the end points are outside $M$ and $x|_{[-\eps_1,T+\eps_2]}$ does not intersect $\p M \setminus \Gamma$. As $x(t)$ was assumed to intersect $\Gamma$ transversally, we can also arrange that $x|_{[-\eps_1,T+\eps_2]}$ only intersects $\p M$ at $x(0)$ and $x(T)$. The quasimodes $v_j$ can be constructed so that $\supp(v_j) \cap (\p M \setminus \Gamma) = \emptyset$.
Then $\norm{P_1 v_1}_{H^s(M)} = \norm{P_2^* v_2}_{H^s(M)} = O(h^{\infty})$. Using Proposition \ref{prop_real_principal_type_solvability}, we construct again solutions $u_j = v_j + r_j$ of $P_1 u_1 = 0$ and $P_2^* u_2 = 0$ in $M$ with $\norm{r_j}_{H^{m}(M)} = O(h^{\infty})$.

Let us show that $\tr_\Gamma^{m-1} u_1 \in C_{P_1, \Gamma}$.
As 
    \begin{align*}
\norm{\tr_{\p M \setminus \Gamma}^{m-1} u_1}_{\mathcal H^{m-1}(\p M \setminus \Gamma)} = \norm{\tr_{\p M \setminus \Gamma}^{m-1} r_1}_{\mathcal H^{m-1}(\p M \setminus \Gamma)} = O(h^\infty),
    \end{align*}
it is enough to show that $\norm{v_1|_\Gamma}_{L^2(\Gamma)}$ is not $O(h^\infty)$. 
Since $x|_{[-\eps_1,T+\eps_2]}$ only meets $\p M$ at $x(0)$ and $x(T)$, it is enough to show that 
$\norm{v_1|_{\tilde \Gamma}}_{L^2(\tilde \Gamma)} \sim h^{\frac{n+1}{2}}$
where $\tilde \Gamma \subset \Gamma$ is a small neighborhood 
of $x(0)$.

We argue similarly to Step 8 of the proof of Theorem 
\ref{thm_quasimode_direct}. One has 
    \begin{align*}
\int_{\tilde \Gamma} |v_1|^2\, dS 
= \int_{-\eps_1}^{T+\eps_2} \int_{-\eps_1}^{T+\eps_2} \int_{\tilde \Gamma} e^{i \Theta(x,t,s)/h} a(x,t,s)\, dS \,dt \,ds,
    \end{align*}
where $a$ does not vanish at $(x(s),s,s)$ and 
    \begin{align*}
\Theta(x,t,s) = \Phi(x,t) - \overline{\Phi(x,s)},
    \end{align*} 
with $\Phi$ as in the proof of Theorem 
\ref{thm_quasimode_direct}. Note also that if $\chi(t,s)$ is a cutoff function supported near $(0,0)$ one has 
    \begin{align*}
\int_{\tilde \Gamma} |v_1|^2\, dS 
= \int_{-\eps_1}^{T+\eps_2} \int_{-\eps_1}^{T+\eps_2} \int_{\tilde \Gamma} e^{i \Theta(x,t,s)/h} a(x,t,s) \chi(t,s) \, dS \,dt \,ds + O(e^{-C/h})
    \end{align*}
since $\im(\Theta) \geq c > 0$ when $x \in \tilde{\Gamma}$ and either $s$ or $t$ is bounded away from $0$.
We have
    \begin{align*}
\im(\Theta) \geq 0 \text{ on $\supp(a)$}, \qquad \Theta(x(0),0,0) = 0, \qquad d_{x,t,s} \Theta(x(0), 0, 0) = 0,
    \end{align*}
and $\norm{v_1|_{\tilde \Gamma}}_{L^2(\tilde \Gamma)} \sim h^{\frac{n+1}{2}}$ follows from the method of stationary phase after we have shown that $\nabla_{x,t,s}^2 \Theta$ is non-singular at $(x(0), 0, 0)$. 
Analogously to (\ref{def_Ms}), the Hessian of $\Theta$ is given by 
    \begin{align*}
\nabla_{x,t,s}^2 \Theta|_{(x(0),0,0)} = 
\left( \begin{array}{ccc} 
2i\im(H) & \dot{\xi} - H \dot{x} & -\dot{\xi} +  \overline{H} \dot{x}
\\ 
(\dot{\xi} - H \dot{x})^t & (H \dot{x} - \dot{\xi}) \cdot \dot{x} & 0
\\
(-\dot{\xi} +  \overline{H} \dot{x})^t & 0 & -(\overline H \dot{x} - \dot{\xi}) \cdot \dot{x} 
\end{array} \right). 
    \end{align*}
Let $\zeta = (v,z,w)^t$, with $v = v_1 + iv_2$, $v_j \in T_{x(0)} \Gamma$, $z,w \in \mathbb C$, be in the kernel of the above matrix, then
    \begin{align*}
0 = \im(\nabla_{x,t,s}^2 \Theta|_{(x(0),0,0)}) \zeta \cdot \overline{\zeta}
= |\im(H)^{1/2}(v - z \dot x)|^2 + |\im(H)^{1/2}(v - w \dot x)|^2.
    \end{align*}
As $\im(H)$ is positive definite and $\dot x(0) \notin T_{x(0)} \Gamma$,
we see that $v = 0$ and $z=w=0$.
Hence $\nabla_{x,t,s}^2 \Theta$ is non-singular at $(x(0), 0, 0)$ and $\norm{v_1|_{\tilde \Gamma}}_{L^2(\tilde \Gamma)}\sim h^{\frac{n+1}{2}}$.

We have now proved that $\tr_\Gamma^{m-1} u_1 \in C_{P_1, \Gamma}$. We will need the following variation of Lemma \ref{lemma_integral_identity_potential}, 
\begin{equation} \label{pone_ptwo_v_part_data}
\abs{((P_1-P_2)u_1, u_2)_{L^2(M)}} = O(h^\infty).
\end{equation}
To show (\ref{pone_ptwo_v_part_data}), we observe that 
due to $C_{P_1,\Gamma} \subset C_{P_2,\Gamma}$ 
there is $\tilde u_2 \in H^{m}(M)$ such that
    \begin{align}\label{tildeu2_props_part_data}
P_2 \tilde u_2 = 0 \quad \text{and} \quad
\nabla^{k} u_1|_{\Gamma} = \nabla^{k} \tilde u_2|_{\Gamma}
\ \text{for} \ k=0,1,\dots,m-1. 
    \end{align}
Therefore
\begin{align*}
|( (P_1-P_2) u_1, u_2)_{L^2(M)}| 
&= |(P_2 (u_1 - \tilde{u}_2), u_2)_{L^2(M)}| 
\\&\le 
C \norm{\tr^{m-1}_{\p M \setminus \Gamma} (u_1 - \tilde u_2)}_{\mathcal H^{m-1}(\p M \setminus \Gamma)} \norm{\tr^{m-1}_{\p M \setminus \Gamma} u_2}_{\mathcal H^{m-1}(\p M \setminus \Gamma)}.
\end{align*}
We have $\norm{\tr^{m-1}_{\p M \setminus \Gamma} u_2}_{\mathcal H^{m-1}(\p M \setminus \Gamma)} = O(h^\infty)$.
The definition of $C_{P_j, \Gamma}$, together with (\ref{tildeu2_props_part_data}), implies that
    \begin{align*}
\norm{\tr^{m-1}_{\p M \setminus \Gamma} (u_1 - \tilde u_2)}_{\mathcal H^{m-1}(\p M \setminus \Gamma)}
\le 2 \norm{\tr^{m-1}_{\Gamma} u_1}_{\mathcal H^{m-1}(\Gamma)}.
    \end{align*}
Finally, for $k$ large the trace theorem gives that $\norm{\tr^{m-1}_{\Gamma} u_1}_{\mathcal H^{m-1}(\Gamma)} \lesssim \norm{u_1}_{H^k(M)} = O(h^{\frac{n+1}{4}-k})$ by (\ref{u_hsm_norm}).

As $Q = P_1-P_2$ vanishes outside $M$, the identity (\ref{pone_ptwo_v_part_data}) implies that 
\[
O(h^\infty) = (Q u_1, u_2)_{L^2(M)} = (Q v_1, v_2)_{L^2(X)} + O(h^{\infty}).
\]
Multiplying this identity by $h^{-\frac{n+1}{2}+m-1}$ and using \eqref{concentration}, we obtain (\ref{subp_part_data}) as in the proof of Theorem~\ref{thm_main1}.

We omit giving a proof of (\ref{lot_part_data}). The modifications needed in the proof of Theorem \ref{thm_main2} are analogous to those above.
\end{proof}

\begin{proof}[Proof of Theorem \ref{thm_first_order_equivalence}]
We first give some initial remarks related to a first order real principal type differential operator $P$ on $M$. Note that $P$ is of the form $P = \frac{1}{i} L + V$ where $L$ is a real vector field on $M$ and $V \in C^{\infty}(M)$. The null bicharacteristic curves for $P$ are of the form $\gamma(t) = (x(t), \xi(t))$ where $x(t)$ is an integral curve for $L$ in $M$. The real principal type condition implies that for any $x_0 \in M$, the maximal integral curve $x: [-\tau_-^L(x_0), \tau_+^L(x_0)] \to M$ of $L$ with $x(0) = x_0$ has finite length, i.e.\ $\tau_{\pm}^L(x_0) < \infty$. Moreover, the equation $Pu = 0$ in $M$ becomes an ODE along each fixed integral curve, and thus any solution $u$ satisfies 
\[
u(x(t)) = u(x(0)) \exp \left[ -i \int_0^t V(x(s)) \,ds \right], \qquad 0 \leq t \leq T,
\]
whenever $x: [0,T] \to M$ is a maximal integral curve of $L$.

Now suppose that $P_j = \frac{1}{i} L_j + V_j$ are as in the theorem. The assumption that $\alpha_{P_1} = \alpha_{P_2}$, together with the fact that the spatial projections of null bicharacteristics are integral curves, implies that 
\[
x_1(T_1) = x_2(T_2)
\]
whenever $x_j: [0,T_j] \to M$ are maximal integral curves of $L_j$ with $x_1(0) = x_2(0)$.

Let $f_1 \in C_{P_1}$, so that $f_1 = u_1|_{\p M}$ where $u_1 \in C^{\infty}(M)$ and $P_1 u_1 = 0$. For any maximal $L_2$-integral curve $x_2: [0,T_2] \to M$, define 
\[
u_2(x_2(t)) := f_1(x_2(0)) \exp \left[ -i \int_0^t V_2(x_2(s)) \,ds \right].
\]
This defines a function $u_2: M \to \mC$ that is smooth along integral curves of $L_2$ and satisfies $P_2 u_2 = 0$. We claim that 
\[
u_2|_{\p M} = u_1|_{\p M}.
\]
In fact, this holds at any $x \in \p M$ with $\tau^{L_2}_-(x) = 0$ by definition. On the other hand, if $x \in \p M$ and $\tau^{L_2}_+(x) = 0$, then $x = x_2(T_2)$ for some maximal $L_2$-integral curve $x_2: [0,T_2] \to M$, and 
\[
u_2(x) = u_2(x_2(T_2)) = f_1(x_2(0)) \exp \left[ -i \int_0^{T_2} V_2(x_2(s)) \,ds \right].
\]
By assumption, if $x_1: [0,T_1] \to M$ is the maximal $L_1$-integral curve with $x_1(0) = x_2(0)$, then also $x_1(T_1) = x_2(T_2)$ and $\exp \left[ -i \int_0^{T_1} V_1(x_1(s)) \,ds \right] = \exp \left[ -i \int_0^{T_2} V_2(x_2(s)) \,ds \right]$. It follows that 
\[
u_2(x) =  f_1(x_1(0)) \exp \left[ -i \int_0^{T_1} V_1(x_1(s)) \,ds \right] = u_1(x_1(T_1)) = u_1(x).
\]
Thus we have proved that $u_1(x) = u_2(x)$ at each $x \in \p M$ so that  $\tau^{L_2}_-(x) = 0$ or $\tau^{L_2}_+(x) = 0$, but this covers all points of $\p M$ by strict convexity.

For any $f_1 \in C_{P_1}$, we have produced a function $u_2$ in $M$ that is smooth along integral curves of $L_2$ and satisfies $P_2 u_2 = 0$ with $u_2|_{\p M} = f_1$. If we can show that $u_2 \in C^{\infty}(M)$, then we have 
\[
C_{P_1} \subset C_{P_2}.
\]
The converse inclusion follows by changing the roles of $P_1$ and $P_2$. Thus it remains to prove a regularity statement: if $M$ is compact with smooth boundary, if $L$ is a nontrapping real vector field in $M$ so that $M$ is strictly convex for $L$, and if $u: M \to \mC$ solves $\frac{1}{i} Lu + Vu = 0$ in $M$ with $u|_{\p M} \in C^{\infty}(\p M)$, then $u \in C^{\infty}(M)$. This is proved in \cite[Lemma 1.1]{PestovUhlmann} when $M$ is the unit sphere bundle, $L$ is the geodesic vector field and $V=0$, and in \cite[Lemma 5.1]{PSU2} for $V \neq 0$. However, inspecting the arguments in \cite{PestovUhlmann, PSU2} shows that the proofs also give the more general result stated above. The details will appear in \cite{PSU_2Dbook}.
\end{proof}

\bibliographystyle{alpha}

\begin{thebibliography}{DKSU09}

\bibitem[AB59]{AB}
Y.\ Aharonov, D.\ Bohm, {\it Significance of electromagnetic potentials in quantum theory}, Phys. Rev. {\bf 115} (1959), 485--491.

\bibitem[AS19]{AssylbekovStefanov}
Y.M.\ Assylbekov, P.\ Stefanov, {\it Sharp stability estimate for the geodesic ray transform}, Inverse Problems (to appear).

\bibitem[BL67]{BabichLazutkin}
V.M.\ Babich, V.F.\ Lazutkin, \textit{{The eigenfunctions which are concentrated near a closed geodesic}} (in Russian), 1967, Problems of Mathematical Physics, No. 2, Spectral Theory, Diffraction Problems, pp. 15--25, Izdat. Leningrad. Univ., Leningrad.

\bibitem[Be87]{BelishevBC}
M.I.\ Belishev. {\it An approach to multidimensional inverse problems for the wave equation}, Doklady Akademii Nauk SSSR, \textbf{297} (1987), 524--527.

\bibitem[CLV19]{CLV}
M.\ Capoferri, M.\ Levitin, D.\ Vassiliev, 
{\it Geometric wave propagator on Riemannian manifolds},
arXiv:1902.06982.

\bibitem[DKSU09]{DKSaU}
D.\ Dos Santos Ferreira, C.E.\ Kenig, M.\ Salo, G.\ Uhlmann,
  \emph{{Limiting Carleman weights and anisotropic inverse problems}}, Invent. Math. \textbf{178} (2009), 119--171.
  
\bibitem[DKLS16]{DKLS} D.\ Dos Santos Ferreira, Y.\ Kurylev, M.\ Lassas, M.\ Salo, \emph{{The Calder{\'o}n problem in transversally anisotropic geometries}}, J. Eur. Math. Soc. (JEMS) \textbf{18} (2016), 2579--2626.

\bibitem[DH72]{DH72}
J.J.\ Duistermaat, L.\ H\"ormander, \emph{{Fourier integral operators. II}}. Acta Math. {\bf 128} (1972),  183--269.

\bibitem[EHN96]{EnglHankeNeubauer}
H.W.\ Engl, M.\ Hanke, A.\ Neubauer, Regularization of
inverse problems. Dordrecht: Kluwer Academic Publishers Group, 1996.

\bibitem[Es15]{Eskin_AB}
G.\ Eskin, {\it Aharonov--Bohm effect revisited}, Rev. Math. Phys. {\bf 27} (2015), no. 2, 1530001.

\bibitem[FIKO19]{FIKO}
A.\ Feizmohammadi, J.\ Ilmavirta, Y.\ Kian, L.\ Oksanen,
{\it Recovery of time dependent coefficients from boundary data for hyperbolic equations}, arXiv:1901.04211.

\bibitem[FIO19]{FIO}
A.\ Feizmohammadi, J.\ Ilmavirta, L.\ Oksanen,
{\it The light ray transform in stationary and static Lorentzian geometries}, arXiv:1911.04834.

\bibitem[FO19]{FO}
A.\ Feizmohammadi, L.\ Oksanen, {\it An inverse problem for a semi-linear elliptic equation in Riemannian geometries}, arXiv:1904.00608.

\bibitem[Ge54]{Gelfand}
I.M.\ Gel'fand, \textit{{Some aspects of functional analysis and algebra.}} Proc. Intern. Cong. Math. {\bf 1} (1954), 253--277.

\bibitem[Gu17]{Guillarmou_trapping}
C.\ Guillarmou, {\it Lens rigidity for manifolds with hyperbolic trapped sets}, J. Amer. Math. Soc. {\bf 30} (2017), 561--599.

\bibitem[Gu76]{Guillemin_scattering_relation}
V.\ Guillemin, {\it Sojourn times and asymptotic properties of the scattering matrix}, Proceedings of the Oji
Seminar on Algebraic Analysis and the RIMS Symposium on Algebraic Analysis (Kyoto Univ., Kyoto,
1976). Publ. Res. Inst. Math. Sci., 12, supplement, 69--88 (1976/77).

\bibitem[GUW16]{GuilleminUribeWang}
V.\ Guillemin, A.\ Uribe, Z.\ Wang, {\it Semiclassical states associated with isotropic submanifolds of phase space}, Lett. Math. Phys. {\bf 106} (2016), 1695--1728.

%\bibitem[H{\"o}70]{Hormander1970} L. H\"ormander, \textit{{On the singularities of solutions of partial differential equations}}, Comm. Pure Appl. Math. {\bf 23} (1970), 329--358.

\bibitem[H{\"o}71]{Hormander1971} L.\ H\"ormander, \textit{{On the existence and the regularity of linear pseudo-differential equations}}, L'Enseignement Math. {\bf 17} (1971), 99--163.

\bibitem[H{\"o}85]{Hormander} L.\ H\"ormander, The analysis of linear partial differential operators I-IV. Springer, 1983--1985.

\bibitem[IM19]{IlmavirtaMonard} J.\ Ilmavirta, F.\ Monard, {\it Integral geometry on manifolds with boundary and applications}, chapter in "The Radon Transform: The First 100 Years and Beyond" (Ronny Ramlau, Otmar Scherzer, eds.), de Gruyter, 2019.

\bibitem[Is91]{Isakov_general} V.\ Isakov, {\it Completeness of products of solutions and some inverse problems for PDE}, J. Diff. Eq. {\bf 92} (1991), 305--316.

\bibitem[KV98]{KarasevVorobjev} M.\ Karasev, Y.\ Vorobjev, {\it Integral representations over isotropic submanifolds and equations of zero curvature}, Adv. Math. {\bf 135} (1998), 220--286.

\bibitem[KKL01]{KKL} A.\ Katchalov, Y.\ Kurylev, M.\ Lassas, \textit{Inverse Boundary Spectral Problems}, Monographs and Surveys in Pure and Applied Mathematics 123, Chapman Hall/CRC, 2001.

\bibitem[KP88]{KP} 
T.\ Kato, G.\ Ponce,
{\it Commutator estimates and the Euler and Navier-Stokes equations},
Comm. Pure Appl. Math. {\bf 41} (1988), 891--907. 

\bibitem[KS13]{KenigSalo}
C.E.\ Kenig, M.\ Salo, \textit{{The Calder\'on problem with partial data on manifolds and applications}}, Analysis \& PDE {\bf 6} (2013), 2003--2048.

\bibitem[KV84]{KohnVogelius} R.\ Kohn, M.\ Vogelius, \emph{Determining conductivity by boundary measurements}, Comm. Pure Appl. Math. {\bf 37} (1984), 289--298.

\bibitem[KLOU14]{KLOU} Y.\ Kurylev, M.\ Lassas, L.\ Oksanen, G.\ Uhlmann, {\it Inverse problems for Einstein-scalar field equations}, arXiv:1406.4776.

\bibitem[KLU18]{KLU_nonlinear} Y.\ Kurylev, M.\ Lassas, G.\ Uhlmann, {\it Inverse problems for Lorentzian manifolds and non-linear hyperbolic equations}, Invent. Math. {\bf 212} (2018), 781--857.

\bibitem[LSV94]{LSV}
A.\ Laptev, Y.\ Safarov, D.\ Vassiliev, {\it On global representation of Lagrangian distributions and solutions of hyperbolic equations}, Comm. Pure Appl. Math. {\bf 47} (1994), 1411--1456. 

\bibitem[LLLS19]{LLLS} M.\ Lassas, T.\ Liimatainen, Y.-H.\ Lin, M.\ Salo, {\it Inverse problems for elliptic equations with power type nonlinearities}, arXiv:1903.12562.

\bibitem[LOSU19]{LOSU}
M.\ Lassas, L.\ Oksanen, P.\ Stefanov, G.\ Uhlmann,
{\it The light ray transform on Lorentzian manifolds},
arXiv:1907.02210.

\bibitem[LUW18]{LUW} M.\ Lassas, G.\ Uhlmann, Y.\ Wang, {\it Inverse problems for semilinear wave equations on Lorentzian manifolds}, Comm. Math. Phys. {\bf  360} (2018), 555--609.

\bibitem[LU89]{LeeUhlmann} J.M.\ Lee, G.\ Uhlmann, \textit{{Determining anisotropic real-analytic conductivities by boundary measurements}}, Comm. Pure Appl. Math. {\bf 42} (1989), 1097--1112.

\bibitem[MS74]{MelinSjostrand}
A.\ Melin, J.\ Sj\"ostrand, \textit{{Fourier integral operators with complex-valued phase functions}}, Springer Lecture Notes in Mathematics {\bf 459} (1974), 120--223.

\bibitem[Mo14]{Montalto}
C.\ Montalto, \textit{{Stable determination of a simple metric, a covector field and a potential from the hyperbolic Dirichlet-to-Neumann map}}, Comm. PDE {\bf 39} (2014), 120--145.

\bibitem[O'N83]{O}
B.\ O'Neill, \textit{Semi-Riemannian geometry}, Academic Press Inc, 1983.

\bibitem[Pa99]{Paternain} G.P.\ Paternain, Geodesic flows. Progress in Mathematics {\bf 180}, Birkh\"auser, 1999.

\bibitem[PSU12]{PSU2} G.P.\ Paternain, M.\ Salo, G.\ Uhlmann, {\it The attenuated ray transform for connections and Higgs fields}, Geom. Funct. Anal. {\bf 22} (2012), 1460--1489.

\bibitem[PSU19]{PSU_2Dbook} G.P.\ Paternain, M.\ Salo, G.\ Uhlmann, Geometric inverse problems in two dimensions. Textbook in preparation.

\bibitem[PU93]{PaulUribe} T.\ Paul, A.\ Uribe, {\it A construction of quasi-modes using coherent states}, Annales de l'I.H.P., section A {\bf 59} (1993), 357--381.

\bibitem[PU05]{PestovUhlmann} L.\ Pestov, G.\ Uhlmann, {\it Two dimensional compact simple Riemannian manifolds are boundary distance rigid}, Ann. of Math. {\bf 161} (2005), 1089--1106.

\bibitem[Ra90]{Rakesh_geometric_optics}
Rakesh, {\it Reconstruction for an inverse problem for the wave equation with constant velocity}, Inverse Problems {\bf 6} (1990), 91--98.

\bibitem[RS88]{RakeshSymes}
Rakesh, W.W.\ Symes, {\it Uniqueness for an inverse problem for the wave equation}, Comm. PDE {\bf 13} (1988), 87--96.

\bibitem[Ra82]{Ralston}
J.\ Ralston, \textit{{Gaussian beams and the propagation of singularities}}, Studies in partial differential
equations, 206--248, MAA Stud. Math. \textbf{23}, Math. Assoc. America, Washington, DC, 1982.

\bibitem[Ra01]{Ralston_note}
J.\ Ralston, \textit{{Gaussian beams}}. Unpublished note, Institute for Pure and Applied Mathematics, April 18-19, 2001.

\bibitem[RR04]{RenardyRogers}
M.\ Renardy, R.C.\ Rogers, An introduction to partial differential equations. Texts in applied mathematics, Springer, 2004.

\bibitem[Ri09]{R}
H.\ Ringstr\"om, \textit{The Cauchy Problem in General Relativity},
European Mathematical Society Publishing House, 2009. 

%\bibitem[Sa07]{Salo_stability}
%M.\ Salo, \textit{{Stability for solutions of wave equations with $C^{1,1}$ coefficients}},  Inverse Probl. Imaging {\bf 1} (2007), 537--556.

\bibitem[SV97]{SV}
Y.\ Safarov, D.\ Vassiliev, {\it The asymptotic distribution of eigenvalues of partial differential operators}, Amer. Math. Soc., Providence (RI), 1997.

\bibitem[Sa17]{Salo_normalforms}
M.\ Salo, \textit{{The Calder\'on problem and normal forms}}, arXiv:1702.02136.

\bibitem[St89]{Stefanov_timedependent}
P.\ Stefanov, {\it Uniqueness of the multi-dimensional inverse scattering problem for time dependent potentials}, Math. Z. {\bf 201} (1989), 541--559.

\bibitem[St17]{Plamen_lightray}
P.\ Stefanov, \textit{Support theorems for the light ray transform on analytic Lorentzian manifolds}, Proc. Amer. Math. Soc. {\bf 145} (2017), 1259--1274.

\bibitem[SU05]{SU_generic}
P.\ Stefanov, G.\ Uhlmann, \textit{{Boundary rigidity and stability for generic simple metrics}}, J. Amer. Math. Soc. {\bf 18} (2005), 975--1003.

\bibitem[SUV16]{StefanovUhlmannVasy}
P.\ Stefanov, G.\ Uhlmann, A.\ Vasy, {\it Boundary rigidity with partial data}, J. Amer. Math. Soc. {\bf 29} (2016), 299--332.

\bibitem[SUV17]{StefanovUhlmannVasy2}
P.\ Stefanov, G.\ Uhlmann, A.\ Vasy, {\it Local and global boundary rigidity and the geodesic X-ray transform in the normal gauge}, arXiv:1702.03638.

\bibitem[SY18]{StefanovYang}
P.\ Stefanov, Y.\ Yang, {\it The inverse problem for the Dirichlet-to-Neumann map on Lorentzian manifolds}, Analysis  \& PDE {\bf 11} (2018), 1381--1414.

%\bibitem[SU87]{SylvesterUhlmann} J.\ Sylvester, G.\ Uhlmann, \textit{A global uniqueness theorem for an inverse boundary value problem}, Ann. of Math. \textbf{125} (1987), 153--169.

\bibitem[SU88]{SylvesterUhlmann_boundary} J.\ Sylvester, G.\ Uhlmann, \emph{Inverse boundary value problems at the boundary - continuous dependence}, Comm. Pure Appl. Math. {\bf 41} (1988), 197--219.

\bibitem[SU90]{SylvesterUhlmann_anisotropic}
J.\ Sylvester, G.\ Uhlmann, \textit{Inverse problems in anisotropic media}, Contemp. Math. {\bf 122} (1991), Inverse scattering and applications (Amherst, MA, 1990), 105--117.

\bibitem[Ta04]{Tataru_survey}
D.\ Tataru, \emph{{Unique continuation problems for partial differential equations}}, Geometric methods in inverse problems and PDE control, 239--255, IMA Vol. Math. Appl., 137, Springer, New York, 2004.

\bibitem[Ta81]{Taylor1981}
M.E.\ Taylor, Pseudodifferential operators. Princeton Mathematical Series 34, Princeton University Press, 1981.

%\bibitem[Ta96]{Taylor3} M.E.\ Taylor, \textit{Partial differential equations III}. Springer-Verlag, 1996.

\bibitem[Tr80]{Treves2}
F.\ Treves, Introduction to pseudodifferential and Fourier integral operators, vol. 2. The University series in mathematics, Plenum Press, 1980.

\bibitem[Uh04]{Uhlmann_scattering_relation}
G.\ Uhlmann, \textit{{The Cauchy data and the scattering relation}}, Geometric methods in inverse problems and PDE control, 263--287, IMA Vol. Math. Appl., 137, Springer, New York, 2004.

%\bibitem[Uh14]{Uhlmann_survey} G.\ Uhlmann, \textit{{Inverse problems: seeing the unseen}}, Bull. Math. Sci. {\bf 4} (2014), 209--279.

%\bibitem[WZ01]{WunschZworski} J.\ Wunsch, M.\ Zworski, \textit{{The FBI transform on compact $C^{\infty}$ manifolds}}, Trans. AMS {\bf 353} (2001), 1151--1167.

\bibitem[VW19]{VasyWang}
A.\ Vasy, Y.\ Wang, \textit{{On the light ray transform with wave constraints}}, arXiv:1912.02848.

\bibitem[Zw12]{Zworski} 
M.\ Zworski, Semiclassical analysis. AMS, 2012.

\end{thebibliography}

\end{document}